\documentclass[a4paper]{amsart}

\usepackage{amscd,amsthm,amsfonts,amssymb,amsmath}
\usepackage[colorlinks=true]{hyperref}
\usepackage{fullpage}
\usepackage[all]{xy}
\usepackage{mathrsfs}
\usepackage{bm}
\usepackage{tikz}
\usepackage{bbm}
\usepackage{graphics}
\usepackage{enumerate}

\newcommand{\BA}{{\mathbb {A}}}\newcommand{\BC}{{\mathbb {C}}}
\newcommand{\BH}{{\mathbb {H}}}
\newcommand{\BK}{{\mathbb {K}}}

\newcommand{\BQ}{{\mathbb {Q}}}\newcommand{\BR}{{\mathbb {R}}}

\newcommand{\BZ}{{\mathbb {Z}}}

\newcommand{\bfA}{{\mathbf {A}}}

\newcommand{\CI}{{\mathcal {I}}}
\newcommand{\CO}{{\mathcal {O}}}\newcommand{\CP}{{\mathcal {P}}}
\newcommand{\CS}{{\mathcal {S}}}

\newcommand{\msp}{\mathscr{P}}

\newcommand{\fa}{{\mathfrak{a}}} \newcommand{\fc}{{\mathfrak{c}}} 
 \newcommand{\fg}{{\mathfrak{g}}} \newcommand{\fh}{{\mathfrak{h}}}
  \newcommand{\fl}{{\mathfrak{l}}}
\newcommand{\fm}{{\mathfrak{m}}} \newcommand{\fn}{{\mathfrak{n}}}\newcommand{\fo}{{\mathfrak{o}}} \newcommand{\fp}{{\mathfrak{p}}}
\newcommand{\fq}{{\mathfrak{q}}} \newcommand{\fr}{{\mathfrak{r}}}\newcommand{\fs}{{\mathfrak{s}}} 
 \newcommand{\fv}{{\mathfrak{v}}}

 \newcommand{\fS}{{\mathfrak{S}}}

                      \newcommand{\Ad}{{\mathrm{Ad}}}

\newcommand{\adeles}{ad\`{e}les~}

                      \newcommand{\bs}{\backslash}

\newcommand{\diag}{{\mathrm{diag}}}

\newcommand{\Hom}{{\mathrm{Hom}}}

\newcommand{\ideles}{id\`{e}les~}

                      		\newcommand{\Mat}{{\mathrm{Mat}}}

			\newcommand{\Nrd}{{\mathrm{Nrd}}}

\newcommand{\ov}{\overline}

                  \newcommand{\ra}{\rightarrow}

         	\newcommand{\Prd}{{\mathrm{Prd}}}

                                  \newcommand{\Res}{{\mathrm{Res}}}

\newcommand{\Spec}{{\mathrm{Spec}}}

 \newcommand{\sk}{\medskip}                      \newcommand{\s}{\sk\noindent}

\newcommand{\Trd}{{\mathrm{Trd}}}

\newcommand{\vol}{{\mathrm{vol}}}

\newcommand{\wt}{\widetilde}                        \newcommand{\wh}{\widehat}

\newtheorem{thm}{Theorem}[section]
\newtheorem{coro}[thm]{Corollary}
\newtheorem{lem}[thm]{Lemma}
\newtheorem{prop}[thm]{Proposition}

\newtheorem{defn}[thm]{Definition}

\theoremstyle{definition}

\theoremstyle{remark}
\newtheorem{remark}[thm]{Remark}

\numberwithin{equation}{subsection}
\def\mat(#1,#2,#3,#4){
  \begin{pmatrix}
  #1 & #2 \\ #3 & #4
  \end{pmatrix}
}

\begin{document}
\title{An infinitesimal variant of Guo-Jacquet trace formula I: the case of $(GL_{2n, D}, GL_{n, D}\times GL_{n, D})$}
\author{Huajie Li}
\date{\today}
\maketitle

\begin{abstract}
We establish an infinitesimal variant of Guo-Jacquet trace formula for the case of $(GL_{2n, D}, $\newline$GL_{n, D}\times GL_{n, D})$. It is a kind of Poisson summation formula obtained by an analogue of Arthur's truncation process. It consists in the equality of the sums of two types of distributions which are non-equivariant in general: one type is associated to rational points in the categorical quotient, while the other type is the Fourier transform of the first type. For regular semi-simple points in the categorical quotient, we obtain weighted orbital integrals. 
\end{abstract}

\tableofcontents


\section{\textbf{Introduction}}

The Guo-Jacquet conjecture proposed in \cite{MR1382478} is a possible generalisation in higher dimensions of Waldspurger's well-known theorem on central values of automorphic $L$-functions for $GL_2$. We briefly recall it as follows. Let $E/F$ be a quadratic extension of number fields and $\eta$ the quadratic character of $\BA^\times/F^\times$ attached to it, where $\BA$ denotes the ring of \adeles of $F$. Consider the group $G = GL_{2n}$ and its subgroup $H = GL_n \times GL_n$ defined over $F$. Let $\pi$ be a cuspidal automorphic representation of $G(\BA)$ with trivial central character. We say that $\pi$ is $H$-distinguished if the two linear forms (called ``periods'') on it
$$ \CP_{H}: \phi \mapsto \int_{H(F)Z(\BA)\backslash H(\BA)} \phi(h) dh $$
and
$$ \CP_{H,\eta}: \phi \mapsto \int_{H(F)Z(\BA)\backslash H(\BA)} \phi(h)\eta(\det(h)) dh $$
are both non-zero, where $Z$ denotes the centre of $G$. This property is directly connected with the non-vanishing of some central $L$-values (see Friedberg-Jacquet's work \cite{MR1241129}). We also need to deal with another pair of groups. Let $X(E)$ denote the set of isomorphic classes of quaternion algebras $D/F$ in which $E$ embeds. For any $D\in X(E)$, let $G_D=GL_{n,D}$ be the algebraic group defined over $F$ whose $F$-points are $GL_n(D)$ and $H_{D} = Res_{E/F}GL_{n,E}$ be its subgroup. Let $\pi_{D}$ be a cuspidal automorphic representation of $G_D(\BA)$ with trivial central character. We say that $\pi_{D}$ is $H_{D}$-distinguished if the linear form on it
$$ \CP_{H_D}: \phi \mapsto \int_{H_{D}(F)Z(\BA)\backslash H_{D}(\BA)} \phi(h) dh, $$
is not zero,  where we identify the centre of $G_{D}$ with $Z$. One part of the Guo-Jacquet conjecture says that if $\pi_D$ is $H_D$-distinguished and $\pi$ is transferred from $\pi_D$ by the Jacquet-Langlands correspondence, then $\pi$ is $H$-distinguished. We can also expect a converse at least when $n$ is odd. For $n=1$, these were known by Waldspurger \cite{MR783511} and reproved by Jacquet \cite{MR868299}. 

Now we formally describe the approach of relative trace formulae following Jacquet \cite{MR868299}. This was adopted by Feigon-Martin-Whitehouse \cite{MR3805647} to obtain some partial results. Let $f^G$ be a smooth function on $G(\BA)$ with compact support. As an analogue of Arthur-Selberg trace formula, the relative trace formula for the case $(G,H)$ roughly says that there are two ways to write the integral (viewed as a distribution)
$$ \int_{H(F)\bs H(\BA)\cap G(\BA)^1} \int_{H(F)\bs H(\BA)\cap G(\BA)^1} \BK_{f^G}(x,y)\eta(\det(x))dxdy, $$
where $G(\BA)^1$ denotes the elements in $G(\BA)$ with absolute-value-$1$ determinant and $\BK_{f^G}(x,y)=\sum_{\gamma\in G(F)} \newline f^G(x^{-1}\gamma y)$. The geometric side is expected to be a sum of relative (weighted) orbital integrals while the spectral side should be an expansion of periods. Similarly there is also another formula for the case of $(G_D,H_D)$. Then the comparison of periods of different pairs of groups predicted by the Guo-Jacquet conjecture is reduced to the comparison of relative (weighted) orbital integrals, for which there are already some works such as Guo's fundamental lemma \cite{MR1382478} and Zhang's transfer \cite{MR3414387}. 

However, we have neglected analytic difficulty in the above discussion. That is to say, the double integral above is not convergent and neither are two ways of its expansions. This is the reason why some restrictive local conditions are needed in the main results of \cite{MR3805647} though they seem kind of artificial. The aim of this article is to solve this kind of problem at the level of Lie algebras for the case of $(G,H)$. 
Notice that such a double integral can be formally written as a single integral
$$ \int_{H(F)\bs H(\BA)\cap G(\BA)^1} K_{f^{G/H}}(x)\eta(\det(x))dx, $$
where $f^{G/H}(x)=\int_{H(\BA)\cap G(\BA)^1} f^G(xy) dy$ defines a smooth function on $(G/H)(\BA)$ with compact support and $K_{f^{G/H}}(x)=\sum_{\gamma\in (G/H)(F)} f^{G/H}(x^{-1}\gamma x)$. Replacing the symmetric space $G/H$ by its tangent space $\fs\simeq\fg\fl_n\oplus\fg\fl_n$ at the neutral element, we are faced with the divergence of the integral
$$ \int_{H(F)\bs H(\BA)\cap G(\BA)^1} k_{f}(x)\eta(\det(x))dx, $$
where $f$ is a Bruhat-Schwartz function on $\fs(\BA)$ and $k_{f}(x)=\sum_{\gamma\in \fs(F)} f(x^{-1}\gamma x)$. 

Our main results can be described as follows. 

First of all, as in \cite{MR518111}, we replace $k_{f}(x)$ with some explicit $k_{f}^T(x)$ (see its definition in (\ref{deftruncation1}) and (\ref{deftruncation2})) to make the last integral absolutely convergent, where $T\in\BR^{2n}$ is a truncation parameter. Moreover, there is a relation of equivalence on $\fs(F)$ defined by the categorical quotient $\fs//H$; we denote by $\CO$ the set of classes of equivalence. For each class $\fo\in\CO$, we define $k_{f,\fo}^T(x)$ and its integral similarly by replacing $\fs(F)$ with $\fo$. Then we have
$$ k_{f}^T(x)=\sum_{\fo\in\CO}k_{f,\fo}^T(x),  $$
and prove the following theorem which gives the geometric expansion of
$$ \int_{H(F)\bs H(\BA)\cap G(\BA)^1} k_{f}^T(x)\eta(\det(x))dx. $$

\begin{thm}[see Theorem \ref{convergence1}]
For $T$ in a suitable cone in $\BR^{2n}$, 
$$ \sum_{\fo\in\CO} \int_{H(F)\bs H(\BA)\cap G(\BA)^1} k_{f,\fo}^T(x)\eta(\det(x))dx $$
is absolutely convergent. 
\end{thm}

Moreover, we see that each summand in the geometric expansion is a sum of products of polynomials and exponential functions in $T$. In fact, most (namely regular semi-simple) terms are simply polynomial distributions. 

\begin{thm}[see Corollary \ref{polynomial1}]
For $T$ in a suitable cone in $\BR^{2n}$ and each $\fo\in\CO$, define 
$$ J_\fo^T(\eta, f):=\int_{H(F)\bs H(\BA)\cap G(\BA)^1} k_{f,\fo}^T(x)\eta(\det(x))dx. $$
Then $T\mapsto J_\fo^T(\eta, f)$ is the restriction of an exponential polynomial in $T$. In particular, if $\fo$ is regular semi-simple, it is the restriction of a polynomial in $T$. 
\end{thm}

This property allows us to take the constant term $J_\fo(\eta, f)$ of $J_\fo^T(\eta, f)$ to eliminate the truncation parameter. In the infinitesimal setting, the geometric expansion of the Fourier transform of $f$ plays the role of the original spectral side (cf. \cite{MR1893921}). Our infinitesimal variant of Guo-Jacquet trace formula equating the geometric developments of $f$ and its Fourier transform (defined by (\ref{fourier}) and denoted by $\hat{f}$) is the following, which essentially comes from the Poisson summation formula. 

\begin{thm}[see Theorem \ref{tf1}]
For a Bruhat-Schwartz function $f$ on $\fs(\BA)$, we have the equality
$$ \sum_{\fo\in\CO} J_\fo(\eta, f)=\sum_{\fo\in\CO} J_\fo(\eta, \hat{f}). $$
\end{thm}

Such a formula should be of interest for at least two reasons. For one thing, since this formula is close to but easier than its analogue for the symmetric space, it gives us a clue to the original relative trace formula (cf. Zydor's work \cite{MR4195660} on Jacquet-Rallis trace formulae). For another, a simplified version of this formula (see \cite[Theorem 8.4 and p. 1875]{MR3414387}) has been used in Zhang's proof of the smooth transfer. 

The distributions $J_\fo(\eta, \cdot)$ on $\fs(\BA)$ that we obtained are non-equivariant under the conjugation of $H(\BA)\cap G(\BA)^1$ in general, which is close to the situation in \cite{MR518111} and different from that in \cite{MR3835522}. In fact, we have the following formula of non-equivariance. The lack of equivariance may add difficulty to the comparison of Guo-Jacquet trace formulae (cf. \cite[\S22]{MR2192011}). 

\begin{prop}[see Corollary \ref{corofnonequivar1}]
For a Bruhat-Schwartz function $f$ on $\fs(\BA)$ and $y\in H(\BA)\cap G(\BA)^1$, we denote $f^y(x):=f(yxy^{-1})$. Then
$$ J_\fo(\eta, f^y)=\eta(\det(y)) \sum_{Q} J_\fo^{Q}(\eta, f^\eta_{Q,y}), $$
where the sum on $Q$ runs over all $\omega$-stable relatively standard parabolic subgroups of $G$ (defined in Section \ref{defomgstb}). Here $J_\fo^{Q}(\eta,\cdot)$ is an analogue of $J_\fo(\eta,\cdot)$ with $G$ replaced by $Q$, and $f^\eta_{Q,y}$ is defined by (\ref{twconst1}) with $s=0$. 
\end{prop}

Nevertheless, we can write regular semi-simple terms as explicit weighted orbital integrals whose weights are the restriction to $H(\BA)$ of Arthur's in \cite{MR518111} for $G(\BA)$. 

\begin{thm}[see Theorem \ref{woi1}]
Let $\fo\in\CO$ be a regular semi-simple class, $P_1$ an $\omega$-stable relatively standard parabolic subgroup of $G$ and $X_1\in\fo$ an elliptic element relative to $P_1$ (defined in Section \ref{regssterms1}). For a Bruhat-Schwartz function $f$ on $\fs(\BA)$, we have
$$ J_\fo(\eta,f)=\vol([H_{X_1}])\cdot\int_{H_{X_1}(\BA)\bs H(\BA)}f(x^{-1}X_1 x)v_{P_1}(x)\eta(\det(x))dx, $$
where $H_{X_1}$ denotes the centraliser of $X_1$ in $H$, $\vol([H_{X_1}])$ is its associated volume and $v_{P_1}(x)$ is the volume of some convex hull. 
\end{thm}

This paper is organised in the following way. Section \ref{notation} and \ref{symmetricpair} are devoted to standard notation in Arthur's work on trace formulae and characterisation of $\CO$ in the specific symmetric pair that we consider respectively. We define the truncated kernel $k_{f,\fo}^T(x)$ and prove its integrability in Section \ref{integrability}. This key definition is partly inspired by \cite{MR3334892} \cite{MR3835522} \cite{10.1007/978-3-319-94833-1_3} (for the decomposition of groups) and \cite{MR1363072} (for the decomposition of linear spaces) apart from Arthur's pioneering work \cite{MR518111} and its Lie algebra variant \cite{MR1893921}. Section \ref{dist1} is about the quantitive behaviour of our distributions with respect to the truncation parameter $T$. In Section \ref{nonequi}, we study their variance under the conjugation of $H(\BA)\cap G(\BA)^1$. In Section \ref{traceformula}, the infinitesimal Guo-Jacquet trace formula for the case of $(GL_{2n},GL_n\times GL_n)$ is given. Section \ref{secondmodifiedkernel} and \ref{weightedorbitalintegral} aim to express the regular semi-simple distribution as weighted orbital integrals. 

Here are two final remarks. Firstly, actually we study the more general symmetric pair $(GL_{p+q,D}, $\newline$GL_{p,D}\times GL_{q,D})$ instead of $(GL_{2n},GL_n\times GL_n)$ and add an extra term $|\Nrd (x_1)|_\BA^s$ to the integrand in most of this article. Not only do we prefer more general results (including the case considered in \cite{MR3299843} for instance) or possible applications (cf. \cite{2018arXiv180307553L} for the study of the first derivative of $L$-functions), but the study of the case where $p=q$ and $s=0$ also yields consideration on a more general setting including the cases where $p\neq q$ or $s\neq 0$ (see Theorem \ref{exppol1} for example). A simple reason for this comes from the structure of the intersection of $H$ and semi-standard Levi subgroups of $G$. Secondly, there are some similarities between our case and the twisted trace formula (cf. \cite{MR3026269}) for $(GL_n\times GL_n)\rtimes \sigma$ where $\sigma$ exchanges two copies of $GL_n$. In fact, we obtain the same weights for regular semi-simple orbits. However, we shall see that more parabolic subgroups will be needed to define the truncation here. We shall return to this discussion at the end of this paper. 

\s{\textbf{Acknowledgement. }}I would like to express my great appreciation to my PhD advisor Professor Pierre-Henri Chaudouard for introducing me to this problem and valuable suggestions during my preparation of this work. Part of this paper was revised during my visit to the Institute for Mathematical Sciences at the National University of Singapore and I would like to thank their hospitality. This work was supported by grants from R\'{e}gion Ile-de-France. 


\section{\textbf{Notation}}\label{notation}

\subsection{\textbf{Roots and weights}}

Let $F$ be a number field and $G$ a reductive group defined over $F$. Denote by $Z_G$ the centre of $G$. Fix a minimal Levi $F$-subgroup $M_0$ of $G$. All the following groups are assumed to be defined over $F$ without further mention. We call a parabolic subgroup $P$ of $G$ semi-standard if $M_0\subseteq P$. For any semi-standard parabolic subgroup $P$ of $G$, we usually write $M_P$ for the Levi factor containing $M_0$ and $N_P$ the unipotent radical. Denote by $A_P$ the maximal $F$-split torus in the centre of $M_P$. Let $X(M_P)_F$ be the group of characters of $M_P$ defined over $F$. Then define
$$ \fa_P:=\Hom_\BZ(X(M_P)_F, \BR) $$
and its dual space
$$ \fa_P^*:=X(M_P)_F\otimes_\BZ \BR, $$
which are both $\BR$-linear spaces of dimension $\dim(A_P)$. Notice that the restriction $X(M_P)_F\hookrightarrow X(A_P)_F$ induces an isomorphism
$$ \fa_P^*\simeq X(A_P)_F\otimes_\BZ \BR. $$

Suppose that  $P_1\subseteq P_2$ are a pair of semi-standard parabolic subgroups of $G$. The restriction $X(M_{P_2})_F\hookrightarrow X(M_{P_1})_F$ induces $\fa_{P_2}^*\hookrightarrow\fa_{P_1}^*$  and its dual map $\fa_{P_1}\twoheadrightarrow\fa_{P_2}$. Denote by $\fa_{P_1}^{P_2}$ the kernel of the latter map $\fa_{P_1}\twoheadrightarrow\fa_{P_2}$. The restriction $X(A_{P_1})_F\twoheadrightarrow X(A_{P_2})_F$ induces $\fa_{P_1}^*\twoheadrightarrow\fa_{P_2}^*$ and its dual map $\fa_{P_2}\hookrightarrow\fa_{P_1}$. The latter map $\fa_{P_2}\hookrightarrow\fa_{P_1}$  provides a section of the previous map $\fa_{P_1}\twoheadrightarrow\fa_{P_2}$. Thus we have decompositions
$$ \fa_{P_1}=\fa_{P_2}\oplus\fa_{P_1}^{P_2} $$
and
$$ \fa_{P_1}^*=\fa_{P_2}^*\oplus(\fa_{P_1}^{P_2})^*. $$
When $P_1$ is a minimal semi-standard parabolic subgroup, since $\fa_{P_1}$ (resp. $A_{P_1}$) and $\fa_{P_1}^{P_2}$ are independent of the choice of $P_1$, we write them as $\fa_0$ (resp. $A_0$) and $\fa_0^{P_2}$ respectively. 

For a pair of semi-standard parabolic subgroups $P_1\subseteq P_2$ of $G$, write $\Delta_{P_1}^{P_2}$ for the set of simple roots for the action of $A_{P_1}$ on $N_{P_1}^{P_2}:=N_{P_1}\cap M_{P_2}$. Notice that $\Delta_{P_1}^{P_2}$ is a basis of $(\fa_{P_1}^{P_2})^*$. Let
$$ (\wh{\Delta}_{P_1}^{P_2})^\vee:=\{\varpi_\alpha^\vee:\alpha\in\Delta_{P_1}^{P_2}\} $$
be the basis of $\fa_{P_1}^{P_2}$ dual to $\Delta_{P_1}^{P_2}$. 
If $B$ is a minimal semi-standard parabolic subgroup contained in $P_1$, one has the coroot $\beta^\vee$ associated to any $\beta\in\Delta_B^{P_2}$. For every $\alpha\in\Delta_{P_1}^{P_2}$, let $\alpha^\vee$ be the projection of $\beta^\vee$ to $\fa_{P_1}^{P_2}$, where $\beta\in\Delta_B^{P_2}$ whose restriction to $\fa_{P_1}^{P_2}$ is $\alpha$. Such $\alpha^\vee$ is independent of the choice of $B$. Define
$$ (\Delta_{P_1}^{P_2})^\vee:=\{\alpha^\vee:\alpha\in\Delta_{P_1}^{P_2}\}, $$
which is a basis of $\fa_{P_1}^{P_2}$. Denote by
$$ \wh{\Delta}_{P_1}^{P_2}:=\{\varpi_\alpha:\alpha\in\Delta_{P_1}^{P_2}\} $$
the basis of $(\fa_{P_1}^{P_2})^*$ dual to $ (\Delta_{P_1}^{P_2})^\vee$. 

For a semi-standard parabolic subgroup $P$ of $G$, set
$$ \fa_{P}^+:=\{T\in\fa_P: \forall \alpha\in\Delta_P^G, \alpha(T)>0\}. $$
For $P_1\subseteq P_2$ as above, define $\tau_{P_1}^{P_2}$ and $\wh{\tau}_{P_1}^{P_2}$ as the characteristic functions of
$$ \{T\in\fa_0: \forall \alpha\in\Delta_{P_1}^{P_2}, \alpha(T)>0\} $$
and
$$ \{T\in\fa_0: \forall \varpi\in\wh{\Delta}_{P_1}^{P_2}, \varpi(T)>0\} $$
respectively. 

\subsection{\textbf{The functions $H_P$ and $F^P$}}

Let $\BA$ be the ring of \adeles of $F$ and $|\cdot|_\BA$ the product of normalised local absolute values on the group of \ideles $\BA^*$. Fix a maximal compact subgroup $K$ of $G(\BA)$ that is admissible relative to $M_0$ in the sense of \cite[p. 9]{MR625344}. In this paper, we choose the standard maximal compact subgroup for inner forms of $GL_n$ (see \cite[p. 191 and 199]{MR1344916} for example). More concretely, suppose that $G(F)=GL_n(D)$, where $D$ is a central division algebra over $F$. For every place $v$ of $F$, fix an isomorphism $D\otimes_{F}F_v\simeq \fg\fl_{r_v}(D_v)$, where $D_v$ is a central division algebra over $F_v$. Under this isomorphism, the completion at $v$ of $G(F)$ is $G_v\simeq GL_{n_v}(D_v)$, where $n_v=nr_v$. For $v$ a finite place of $F$, let $K_v\simeq GL_{n_v}(\CO_{D_v})$, where $\CO_{D_v}$ is the ring of integers of $D_v$; for $v$ an infinite place of $F$, we choose $K_v$ to be the orthogonal group, unitary group and compact symplectic group (see \cite[Chapter 1.2.8]{MR3331229} for example) for $G_v\simeq GL_{n_v}(\BR)$, $GL_{n_v}(\BC)$ and $GL_{n_v}(\BH)$ respectively; let $K:=\prod_{v} K_v$.  Suppose that $P$ is a semi-standard parabolic subgroup of $G$. If $m\in M_P(\BA)$, define $H_P(m)\in\fa_P$ by
$$ \langle H_P(m), \chi\rangle=\log(|\chi(m)|_\BA) $$ 
for all $\chi\in X(M_P)_F$. Write $M_P(\BA)^1$ for the kernel of $H_P$ and $A_P^\infty$ for the neutral component for the topology of $\BR$-manifolds of the group of $\BR$-points of the maximal $\BQ$-split torus in $\Res_{F/\BQ}A_P$. Then any element $x\in G(\BA)$ can be written as $x=nmak$, where $n\in N_P(\BA)$, $m\in M_P(\BA)^1$, $a\in A_P^\infty$ and $k\in K$. We can define a continuous map $H_P: G(\BA)\ra\fa_P$ by setting $H_P(x):=H_P(a)$ with respect to this decomposition. Notice that $H_P$ induces an isomorphism from $A_P^\infty$ to $\fa_P$. If $P\subseteq Q$ are a pair of semi-standard parabolic subgroups, write
$$ A_P^{Q, \infty}:=A_P^\infty\cap M_Q(\BA)^1. $$
Then $H_P$ also induces an isomorphism from $A_P^{Q, \infty}$ to $\fa_P^Q$. 

Denote by $\Omega^G$ the Weyl group of $(G, A_0)$. In the cases to be considered in this paper, for every $s\in\Omega^G$, we can always choose one representative $\omega_s\in G(F)\cap K$. 
In fact, we are dealing with the case of $G=GL_n$ or its inner forms, thus we can choose $\Omega^G$ to be the group of permutation matrices. For an $F$-subgroup $H$ of $G$ and $s\in\Omega^G$, we usually write $sH:=\omega_s H\omega_s^{-1}$. Let $P_1$ and $P_2$ be a pair of semi-standard parabolic subgroups of $G$. Denote by $\Omega^G(\fa_{P_1},\fa_{P_2})$ the (perhaps empty) set of distinct isomorphisms from $\fa_{P_1}$ to $\fa_{P_2}$ obtained by restriction of elements in $\Omega^G$. Denote by $\Omega^G(\fa_{P_1}; P_2)$ the (perhaps empty) subset of double classes in $\Omega^{M_{P_2}}\bs\Omega^G/\Omega^{M_{P_1}}$ of elements $s\in\Omega^G$ such that $s(\fa_{P_1})\supseteq\fa_{P_2}$. Suppose additionally that $P_1$ and $P_2$ contain a common minimal semi-standard parabolic subgroup $P_0$ of $G$. We can talk about positive roots and standard parabolic subgroups with respect to $P_0$. By \cite[Lemme 1.3.6]{MR3026269}, each $s\in\Omega^G(\fa_{P_1},\fa_{P_2})$ admits a unique representative (still denoted by $s$) in $\Omega^G$ such that $s^{-1}\alpha>0$ for all $\alpha\in\Delta_{P_0}^{P_2}$. By \cite[Lemme 1.3.7]{MR3026269}, each $s\in\Omega^G(\fa_{P_1}; P_2)$ admits a unique representative (still denoted by $s$) in $\Omega^G$ such that $s^{-1}\alpha>0$ for all $\alpha\in\Delta_{P_0}^{P_2}$. If $Q$ is a parabolic subgroup of $G$ containing $P_1\cup P_2$ and $\Omega^{M_Q}(\fa_{P_1}, \fa_{P_2})\neq\emptyset$, we say that $P_1$ and $P_2$ are $M_Q$-associated. There is a bijection between $\Omega^G(\fa_{P_1}; P_2)$ and the disjoint union of quotients $\Omega^{M_{P_2}}(\fa_R, \fa_R)\bs\Omega^G(\fa_{P_1}, \fa_R)$ where $R$ runs over standard parabolic subgroups of $G$ contained in $P_2$, modulo $M_{P_2}$-association (see \cite[Lemme 1.3.7]{MR3026269}). 

From the reduction theory (see \cite[p. 941]{MR518111}), we know that there exists a real number $t_0<0$ and a compact subset $\varrho_B\subseteq N_B(\BA)M_0(\BA)^1$ for each minimal semi-standard parabolic subgroup $B$ of $G$ such that for any semi-standard parabolic subgroup $P$ of $G$ containing $B$, we have
$$ G(\BA)=P(F)\fS_B^P(\varrho_B, t_0). $$
Here the Siegel set $\fS_B^P(\varrho_B, t_0)$ is defined by
$$ \fS_B^P(\varrho_B, t_0):=\varrho_B A_B^\infty(P, t_0) K, $$
where 
$$ A_B^\infty(P, t_0):=\{a\in A_B^\infty: \forall \alpha\in\Delta_B^P, \alpha(H_B(a))>t_0\}. $$
We shall fix such $t_0$ and $\varrho_B$. Additionally, we are authorised to assume that $\varrho_{sB}=\omega_s\varrho_B\omega_s^{-1}$ for $s\in\Omega^G$. Moreover, we require that $(M_P(\BA)\cap\varrho_B, M_P(\BA)\cap K, B\cap M_P, t_0)$ will play the role of $(\varrho_B, K, B, t_0)$ for any semi-standard parabolic subgroup $P$ of $G$ containing $B$. 

Let $B\subseteq P$ and $t_0$ be as above. For $T\in\fa_0$, define the truncated Siegel set
$$ \fS_B^P(\varrho_B, t_0, T):=\varrho_B A_B^\infty(P, t_0, T) K, $$
where
$$ A_B^\infty(P, t_0, T):=\{a\in A_B^\infty(P, t_0): \forall \varpi\in\wh{\Delta}_B^P, \varpi(H_B(a)-T)\leq 0\}. $$
It is known that $A_B^\infty(P, t_0, T)\cap M_P(\BA)^1$ has compact closure (see \cite[Lemme 1.8.1]{MR3026269}). Denote by $F_B^P(\cdot, T)$ the characteristic function of the projection of $\fS_B^P(\varrho_B, t_0, T)$ to $P(F)\bs G(\BA)$. 

\subsection{\textbf{Bruhat-Schwartz functions and Haar measures}}\label{BSandHaar1}

Write $\fg$ for the Lie algebra of $G$. For an $F$-linear subspace $\fs$ of $\fg$, denote by $\CS(\fs(\BA))$ the Bruhat-Schwartz space of $\fs(\BA)$, namely the $\BC$-linear space of functions on $\fs(\BA)$ generated by $f_\infty\otimes\chi^\infty$, where $f_\infty$ is a Schwartz function on $\fs(F\otimes_\BQ \BR)$ and $\chi^\infty$ is the characteristic function of an open compact subset of $\fs(\BA^\infty)$, where we denote by $\BA^\infty$ the ring of finite \adeles of $F$. 

Let $P$ be a semi-standard parabolic subgroup of $G$. For every algebraic subgroup $V$ of $N_P$ (resp. every subspace $\fh$ of $\fg$), choose the unique Haar measure on $V(\BA)$ (resp. on $\fh(\BA)$) such that $\vol(V(F)\bs V(\BA))=1$ (resp. $\vol(\fh(F)\bs\fh(\BA))=1$). We also take the Haar measure on $K$ such that $\vol(K)=1$. 

Fix a Euclidean norm $\|\cdot\|$ on $\fa_0$ invariant by the group $\Omega^G$ and Haar measures on all subspaces of $\fa_0$ compatible with this norm. If $P\subseteq Q$ are a pair of semi-standard parabolic subgroups, we obtain Haar measures on $A_P^\infty$ and $A_P^{Q, \infty}$ via the isomorphism $H_P$. 

Denote by $\rho_P\in (\fa_P^G)^*$ the half of the sum of weights  (with multiplicities) for the action of $A_P$ on $\fn_P$. We choose compatible Haar measures on $G(\BA)$ and its Levi subgroups by requiring that for any $f\in L^1(G(\BA))$, 
  \[\begin{split}
   \int_{G(\BA)}f(x)dx&=\int_{N_P(\BA)} \int_{M_P(\BA)} \int_K f(nmk) e^{-2\rho_P(H_P(m))} dndmdk \\
                              &=\int_{N_P(\BA)} \int_{M_P(\BA)^1} \int_{A_P^\infty} \int_K f(nmak) e^{-2\rho_P(H_P(a))} dndmdadk.
  \end{split}\]


\section{\textbf{The symmetric pair}}\label{symmetricpair}

Let $F$ be a number field and $D$ a central division algebra over $F$. Let $d$ be the degree of $D$, i.e., $\dim_F(D)=d^2$. Denote by $GL_{n, D}$ the reductive group over $F$ whose $F$-points are $GL_n(D)$. For $x\in GL_n(D)$, we write $\Nrd (x)$ for its reduced norm, $\Trd (x)$ for its reduced trace and $\Prd_x$ for its reduced characteristic polynomial. For $x\in GL_{p}(D)\times GL_{q}(D)$, denote by $x_1$ (resp. $x_2$) its projection to the first (resp. second) component. Until further notice, we shall work in a more general setting than that of Guo-Jacquet for later use, i.e., we shall study the case of $(GL_{p+q, D}, GL_{p, D}\times GL_{q, D})$ and add an additional term $|\Nrd (x_1)|_\BA^s$ in the integral of the modified kernel. 

\subsection{\textbf{Groups and linear spaces}}

Let $G:=GL_{p+q, D}$ and $H:=GL_{p, D}\times GL_{q, D}$ its subgroup by diagonal embedding. Define an involution $\theta$ on $G$ by $\theta(g)=\epsilon g \epsilon^{-1}$, where $\epsilon=
\left( \begin{array}{cc}
1_p & 0 \\
0 & -1_q \\
\end{array} \right)$. Thus $H=G^{\theta}$, where $G^\theta$ denotes the $\theta$-invariant subgroup of $G$.

Define an anti-involution $\iota$ on $G$ by $\iota(g)=\theta(g^{-1})$. Denote by $S$ 
the $\iota$-invariant subvariety of $G$. There is a symmetrization map
$$ s: G \ra S, s(g):=g\iota(g), $$
by which one can regard the symmetric space $G/H$ as a subvariety of $S$. We see that $H\times H$ acts on $G$ by left and right translation and that $H$ acts on $S$ by conjugation.

Let $\fg:=Lie(G)$ and $\fh:=Lie(H)$. Denote by $d\theta$ the differential of $\theta$. Thus 
$$ \fh=\{X\in\fg: (d\theta)(X)=X\}. $$ 
Let $\fs$ be the tangent space of $S$ at the neutral element. We shall always view $\fs$ as a subspace of $\fg$. Then 
$$ \fs=\{X\in\fg: (d\theta)(X)=-X\}, $$ 
and 
$$ \fs(F)=\bigg\{
\left( \begin{array}{cc}
0 & A \\
B & 0 \\
\end{array} \right): A\in \Mat_{p\times q}(D), B\in \Mat_{q\times p}(D)\bigg\} \simeq \Mat_{p\times q}(D)\oplus \Mat_{q\times p}(D). $$ 
There is an $H(F)$-action on $\fs(F)$ by conjugation, i.e., 
$$ (h_{1},h_{2})\cdot(A,B)=(h_{1}Ah_{2}^{-1},h_{2}Bh_{1}^{-1}). $$

\subsection{\textbf{Semi-simple elements}}

We say that an element $X\in\fs$ is semi-simple if the orbit $H\cdot X$ is Zariski closed in $\fs$. By a regular element $X\in\fs$, we mean that the stabiliser $H_X$ has minimal dimension.

\begin{prop}\label{ss1}
An element $X$ of $\fs(F)$ is semi-simple if and only if it is $H(F)$-conjugated to an element of the form
$$ X(A):=
\left( \begin{array}{cccc}
0 & 0 & 1_{m} & 0 \\
0 & 0 & 0 & 0 \\
A & 0 & 0 & 0 \\
0 & 0 & 0 & 0 \\
\end{array} \right) $$
with $A\in GL_{m}(D)$ being semi-simple in the usual sense. More precisely, the set of $H(F)$-conjugacy classes of semi-simple elements of $\fs(F)$ is bijective to the set of pairs $(m, \{A\})$ where $0\leq m\leq \min\{p, q\}$ is an integer and $\{A\}$ is a semi-simple conjugacy class in $GL_{m}(D)$. Moreover, $X(A)$ is regular semi-simple if and only if $m=\min\{p, q\}$ and $A$ is regular semi-simple in $GL_{\min\{p, q\}}(D)$ in the usual sense.
\end{prop}

\begin{proof}
  The case $D=F$ is \cite[Proposition 2.1 and Lemma 2.1]{MR1394521} while the case $p=q$ is \cite[Proposition 5.2]{MR3299843}. This proposition is nothing but a slightly more general one combining both cases, whose proofs are similar and still work here. 
\end{proof}

\begin{prop}\label{prdofrs1}
If $p\leq q$, an element $\mat(0,A,B,0)\in\fs$ is regular semi-simple if and only if $\Prd_{AB}$ is separable and $\Prd_{AB}(0)\neq 0$. 
If $p> q$, an element $\mat(0,A,B,0)\in\fs$ is regular semi-simple if and only if $\Prd_{BA}$ is separable and $\Prd_{BA}(0)\neq 0$. 
\end{prop}

\begin{proof}
We only consider the case $p\leq q$ since the other case can be deduced by symmetry. We may study the proposition over an algebraic closure $\ov{F}$ of $F$. For $A\in \Mat_{dp\times dq}(\ov{F})$ and $B\in \Mat_{dq\times dp}(\ov{F})$, we see that
 \begin{equation}\label{errorpf}
 \det\left(\lambda I_{d(p+q)}-\mat(0,A,B,0)\right)=\lambda^{d(q-p)} \det(\lambda^2 I_{dp}-AB). 
 \end{equation}
 Let $X:=\mat(0,A,B,0)\in\fs$ and denote by $\Prd_X$ the reduced polynomial of $X$ viewed as an element of $\fg$. Then $\Prd_X$ and $\Prd_{AB}$ determine each other. 

Suppose that $\Prd_{AB}$ is separable and $\Prd_{AB}(0)\neq 0$. Let $X=X_s+X_n$ be the Jordan decomposition in $\fg$, where $X_s$ is semi-simple, $X_n$ is nilpotent and $X_sX_n=X_nX_s$. By the uniqueness of the Jordan decomposition, we see that $X_s, X_n\in\fs$. From Proposition \ref{ss1}, up to conjugation by $H$, we may suppose that
$X_s=
\left( \begin{array}{ccc}
  0                                      & 1_{dp}  & 0  \\
  C                                      & 0         & 0  \\
  0                                      & 0         & 0  \\
\end{array} \right)$, 
where $C\in GL_{dp}(\ov{F})$ is semi-simple. Since $\Prd_X=\Prd_{X_s}$, we deduce that $\Prd_{AB}=\Prd_C$. Then $\Prd_C$ is separable and $\Prd_C(0)\neq 0$ by our assumption. By linear algebra, $C$ is regular semi-simple in $GL_{dp}(\ov{F})$, which implies that $X_s\in\fs$ is regular semi-simple by Proposition \ref{ss1}. Since $X_sX_n=X_nX_s$, simple computation (cf. \cite[Lemma 2.1]{MR1394521}) shows that 
$X_n=
\left( \begin{array}{ccc}
  0                                      & D         & 0  \\
  DC                                    & 0         & 0  \\
  0                                      & 0         & 0  \\
\end{array} \right)$, 
where $D\in\fg\fl_{dp}(\ov{F})$ and $DC=CD$. On the one hand, because $X_n$ is nilpotent and $C$ is invertible, we see that $D$ is nilpotent. On the other hand, because $DC=CD$ and $C$ is regular semi-simple, we see that $D$ is semi-simple. Hence, we have $D=0$ and thus $X_n=0$. Therefore, $X=X_s\in\fs$ is regular semi-simple. 

The other direction is a direct consequence of Proposition \ref{ss1}. 
\end{proof}

\subsection{\textbf{Invariants}}\label{inv1}

Denote by $\fc$ the affine space $\bfA^{d\min\{p, q\}}$. Define a morphism $\pi:\fs\ra\fc$ by mapping $\mat(0,A,B,0)\in\fs$ to the coefficients of the reduced characteristic polynomial of $AB$. It is constant on $H$-orbits. Denote by $\fc_{rs}$ the subset of $(c_i)_{0\leq i\leq d\min\{p, q\}-1}\in\fc$ such that the polynomial
$$ P(\lambda):=\lambda^{d\min\{p, q\}}+\sum_{i=0}^{d\min\{p, q\}-1} c_i \lambda^i $$
is separable and $c_0\neq 0$. It is a principal Zariski open subset of $\fc$. Denote by $\fc^\times$ the subset of $(c_i)_{0\leq i\leq d\min\{p, q\}-1}\in\fc$ such that $c_0\neq 0$. Then $\fc_{rs}\subseteq\fc^\times$. 

\begin{prop}\label{propcatquot1}
The pair $(\fc, \pi)$ defines a categorical quotient of $\fs$ by $H$ over $F$. 
\end{prop}

\begin{proof}
It suffices to consider the case $p\leq q$ since the case $p>q$ can be obtained by symmetry. 

We first extend the base field to an algebraic closure $\ov{F}$ of $F$. Then $H_{\ov{F}}\simeq GL_{dp, \ov{F}}\times GL_{dq, \ov{F}}$ and $\fs_{\ov{F}}\simeq \Mat_{dp\times dq, \ov{F}}\oplus \Mat_{dq\times dp, \ov{F}}$. For $(c_i)_{0\leq i\leq dp-1}\in\fc_{\ov{F}}$, denote by $A((c_i)_{0\leq i\leq dp-1})\in GL_{pd}$ its companion matrix
$$ A((c_i)_{0\leq i\leq dp-1}):=
\left( \begin{array}{ccccc}
  0        & 0        & \cdots & 0         & -c_0         \\
  1        & 0        & \cdots & 0         & -c_1         \\
  0        & 1        & \ddots & \vdots  & -c_2         \\
  \vdots & \ddots & \ddots & 0         & \vdots       \\
  0        & \cdots & 0        & 1          & -c_{dp-1} \\
\end{array} \right). $$
Define a morphism $\fc_{\ov{F}}\ra\fs_{\ov{F}}$ by mapping $(c_i)_{0\leq i\leq dp-1}$ to
$$ \left( \begin{array}{ccc}
  0                                      & 1_{dp} & 0  \\
  A((c_i)_{0\leq i\leq dp-1})  & 0         & 0  \\
  0                                      & 0         & 0  \\
\end{array} \right). $$
This is a section of $\pi$, so $\pi$ is surjective. By Propositions \ref{prdofrs1} and \ref{ss1}, the fibre of any point in the non-empty open subset $\fc_{\ov{F}, rs}\subseteq\fc_{\ov{F}}$ contains exactly one closed orbit. We may use Igusa's criterion (see \cite[Theorem 4.13]{zbMATH00007924} and Remark \ref{rmkigusa} below) to show that the pair $(\fc_{\ov{F}}, \pi)$ defines a categorical quotient of $\fs_{\ov{F}}$ by $H_{\ov{F}}$. 

The morphism $\pi: \fs\ra\fc$ defined over $F$ factors through the categorical quotient $\Spec(F[\fs]^H)$ of $\fs$ by $H$ over $F$. This induces a dual morphism $F[\fc]\ra F[\fs]^H$ of $F$-algebras. We have shown that after the base change to $\ov{F}$, it is an isomorphism of $\ov{F}$-algebras. By Galois descent, we deduce that the morphism $F[\fc]\ra F[\fs]^H$ is an isomorphism of $F$-algebras, i.e., the pair $(\fc, \pi)$ defines a categorical quotient of $\fs$ by $H$ over $F$. 
\end{proof}

\begin{remark}\label{rmkigusa}
We notice that $\fc_{\ov{F}}$ can be of dimension $1$ (when $D=F$ and $\min\{p,q\}=1$) in the proof of Proposition \ref{propcatquot1} above, so the first condition in \cite[Theorem 4.13]{zbMATH00007924} may not be satisfied. However, as is evident from the proof of Igusa's criterion, this condition can be replaced with the surjectivity of $\pi$. 
\end{remark}

The categorical quotient $(\fc, \pi)$ defines a relation of equivalence on $\fs(F)$, where two elements are in the same class if and only if they have the same image under $\pi$. We denote by $\CO$ the set of equivalent classes for this relation. By Proposition \ref{ss1}, two semi-simple elements of $\fs(F)$ belong to the same class of $\CO$ if and only if they are conjugate by $H(F)$. Denote by $\CO_{rs}$ the subset of $\CO$ with images in $\fc_{rs}$. By Proposition \ref{prdofrs1}, each class in $\CO_{rs}$ is a regular semi-simple $H(F)$-orbit in $\fs(F)$. Denote by $\CO^\times$ the subset of $\CO$ with images in $\fc^\times$. Then $\CO_{rs}\subseteq \CO^\times$. 

\subsection{\textbf{Relatively standard parabolic subgroups}}\label{relstdpar1}

Fix $\wt{P}_0$ a minimal parabolic subgroup of $H$ defined over $F$ and $M_0$ a Levi factor of $\wt{P}_0$ defined over $F$. Then $M_0$ is also a minimal Levi subgroup of $G$ defined over $F$. For a semi-standard parabolic subgroup $P$ of $G$ (namely $M_0\subseteq P$), we say that $P$ is ``relatively standard'' if $\wt{P}_0\subseteq P$, i.e., $P\cap H$ is a standard parabolic subgroup of $H$ (namely $\wt{P}_0\subseteq P\cap H$). We shall suppose that $\varrho_{\wt{P}_0}\subseteq\varrho_B$ for all relatively standard minimal parabolic subgroup $B$ of $G$. Denote by $K$ the standard maximal compact subgroup of $G(\BA)$ and by $K_H:=H(\BA)\cap K$ the maximal compact subgroup of $H(\BA)$. Up to conjugation by $G(F)$, we may assume that $M_0$ is the subgroup of diagonal matrices in $G$ and that $\wt{P}_0$ is the product of groups of upper triangular matrices. 

We can describe the embedding $H\hookrightarrow G$ via $D$-bimodules.  Let $V:=\langle e_1,\cdot\cdot\cdot,e_p\rangle_D$ (resp. $W:=\langle f_1,\cdot\cdot\cdot,f_q\rangle_D$) be the free $D$-bimodule generated by the basis $\{e_1,\cdot\cdot\cdot,e_p\}$ (resp. $\{f_1,\cdot\cdot\cdot,f_q\}$). Set $GL(V)$ to be the group of $F$-linear automorphisms on V, which acts on $V$ on the left. Denote by $GL(V)_D$ the subgroup of $GL(V)$ which respects the right $D$-module structure on $V$. Put $G:=GL(V\oplus W)_D$ and $H:=GL(V)_D\times GL(W)_D$. Then $M_0$ is the stabiliser in $G$ (or in $H$) of the $D$-lines $\langle e_i\rangle_D, 1\leq i\leq p$ and $\langle f_i\rangle_D, 1\leq i\leq q$. Suppose that $\wt{P}_0$ is the direct product of the stabiliser in $GL(V)_D$ of the flag
  $$ 0\subsetneq\langle e_1\rangle_D\subsetneq\langle e_1,e_2\rangle_D\subsetneq\cdot\cdot\cdot\subsetneq\langle e_1,\cdot\cdot\cdot,e_p\rangle_D=:V $$
and the stabiliser in $GL(W)_D$ of the flag
  $$ 0\subsetneq\langle f_1\rangle_D\subsetneq\langle f_1,f_2\rangle_D\subsetneq\cdot\cdot\cdot\subsetneq\langle f_1,\cdot\cdot\cdot,f_q\rangle_D=:W. $$
A relative standard parabolic subgroup $P$ of $G$ can be interpretated as the stabiliser in $G$ of the flag
  \[\begin{split}
    0&\subsetneq\langle e_1,\cdot\cdot\cdot,e_{p_1},f_1,\cdot\cdot\cdot,f_{q_1}\rangle_D
    \subsetneq\langle e_1,\cdot\cdot\cdot,e_{p_1},f_1,\cdot\cdot\cdot,f_{q_1},
    e_{p_1+1},\cdot\cdot\cdot,e_{p_1+p_2},f_{q_1+1},\cdot\cdot\cdot,f_{q_1+q_2}\rangle_D \\
    &\subsetneq\cdot\cdot\cdot\subsetneq\langle e_1,\cdot\cdot\cdot,e_{p_1},f_1,\cdot\cdot\cdot,f_{q_1},\cdot\cdot\cdot,
    e_{p-p_l+1},\cdot\cdot\cdot,e_{p},f_{q-q_l+1},\cdot\cdot\cdot,f_{q}\rangle_D=:V\oplus W, 
  \end{split}\]
where $\sum\limits_{i=1}^{l}p_i=p$, $\sum\limits_{i=1}^{l}q_i=q$ and we allow $p_i$ or $q_i$ to be zero. In particular, we have
  $$ M_P\simeq GL_{p_1+q_1, D}\times\cdot\cdot\cdot\times GL_{p_l+q_l, D} $$
  and
  $$ M_{P_H}\simeq GL_{p_1, D}\times\cdot\cdot\cdot\times GL_{p_l, D} \times GL_{q_1, D}\times\cdot\cdot\cdot\times GL_{q_l, D}. $$

\begin{prop}\label{propofpipar1}
  Let $P$ be a relative standard parabolic subgroup of $G$. For all $X\in(\fm_P\cap\fs)(F)$ and $U\in(\fn_P\cap\fs)(F)$, we have
  $$ \pi(X)=\pi(X+U). $$
\end{prop}

\begin{proof}
  It is a consequence of \cite[Lemma 2.1]{MR1363072}. We can also give a direct proof as follows. 
Because of \eqref{errorpf}, for any $X\in\fs(F)$, $\pi(X)$ is determined by the coefficients of the reduced characteristic polynomial of $X$ regarded as an element of $\fg(F)$. The proposition follows from the easy fact: for $X\in\fm_P(F)$ and $U\in\fn_P(F)$, the reduced characteristic polynomial of $X+U$ is equal to that of $X$.
\end{proof}

\begin{coro}\label{orbit1}
  Let $P$ be a relative standard parabolic subgroup of $G$ and $\fo\in\CO$. For all subsets $S_1\subseteq(\fm_P\cap\fs)(F)$ and $S_2\subseteq(\fn_P\cap\fs)(F)$, we have $\fo\cap(S_1\oplus S_2)=(\fo\cap S_1)\oplus S_2$.
\end{coro}

\subsection{\textbf{Fourier transform}}\label{secft}

  Fix a nontrivial unitary character $\Psi$ of $\BA/F$. Let $\langle\cdot,\cdot\rangle$ be the non-degenerate $H(\BA)$-invariant bilinear form on $\fs(\BA)$ defined by
\begin{equation}\label{bilform}
  \langle X_1,X_2\rangle:=\Trd (X_1 X_2)
\end{equation}
for all $X_1,X_2\in\fs(\BA)$. 
  For $f\in\CS(\fs(\BA))$, its Fourier transform $\hat{f}\in\CS(\fs(\BA))$ is defined by
\begin{equation}\label{fourier}
  \hat{f}(\wh{X}):=\int_{\fs(\BA)} f(X)\Psi(\langle X,\wh{X}\rangle) dX
\end{equation}
for all $\wh{X}\in\fs(\BA)$. 


\section{\textbf{Integrability of the modified kernel}}\label{integrability}

Fix a minimal semi-standard parabolic subgroup $P_0$ of $G$. For any semi-standard parabolic subgroup $P$ of $G$ and $T\in\fa_0$, denote by $T_P$ the projection of $sT$ in $\fa_P$, where $s$ is any element in $\Omega^G$ such that $sP_0\subseteq P$. Notice that this definition is independent of the choice of $s$. 

For a semi-standard parabolic subgroup $P$ of $G$, $x\in G(\BA)$ and $T\in\fa_0$, define
$$ F^{P}(x,T):=F_{sP_0}^{P}(x,T_{sP_0}), $$
where $s$ is any element in $\Omega^G$ such that $sP_0\subseteq P$. 

\begin{lem}
	The above definition of $F^{P}(x,T)$ is independent of the choice of $s$. 
\end{lem}

\begin{proof}
	For any $s\in\Omega^G$ and any minimal semi-standard parabolic subgroup $B\subseteq P$, since we choose $\omega_s\in G(F)\cap K$, we have 
	$$ F_{sB}^{sP}(x,T)=F_B^P(\omega_s^{-1}x,s^{-1}T). $$
	
	 Let $s, s'\in\Omega^G$ be such that $sP_0, s'P_0\subseteq P$. Then $s's^{-1}\in\Omega^{M_P}$. By the last equality and the left $M_P(F)$-invariance of $F_{sP_0}^{P}(\cdot,T_{sP_0})$, we have 
	 $$ F_{s'P_0}^{P}(x,T_{s'P_0})=F_{sP_0}^{P}(\omega_{s's^{-1}}^{-1}x, s{s'}^{-1}T_{s'P_0})=F_{sP_0}^{P}(x,T_{sP_0}). $$
	 This completes the proof of the lemma. 
\end{proof}

Let $f\in\CS(\fs(\BA))$, $P$ be a relatively standard parabolic subgroup of $G$ and $\fo\in\CO$. Write $P_H:=P\cap H$. For $x\in M_{P_H}(F) N_{P_H}(\BA)\bs H(\BA)$, define
$$ k_{f,P,\fo}(x):=\sum_{X\in\fm_P(F)\cap\fo} \int_{(\fn_P\cap\fs)(\BA)} f(x^{-1}(X+U)x) dU. $$
For $T\in\fa_0$ and $x\in H(F)\bs H(\BA)$, define
\begin{equation}\label{deftruncation1}
 k_{f,\fo}^T(x):=\sum_{\{P:\wt{P}_0\subseteq P\}} (-1)^{\dim(A_P/A_{G})} \sum_{\delta\in P_H(F)\bs H(F)} \wh{\tau}_P^G(H_{P}(\delta x)-T_P)\cdot k_{f,P,\fo}(\delta x). 
\end{equation}
We know that the sum over $\delta\in P_H(F)\bs H(F)$ is finite from the following lemma. 

\begin{lem}\label{artlem51}
	Let $P$ be a semi-standard parabolic subgroup of $G$. For all $x\in G(\BA)$ and $T\in\fa_P$, the sum 
	$$ \sum_{\delta\in P(F)\bs G(F)} \wh{\tau}_P^G(H_{P}(\delta x)-T) $$
	is finite. 
\end{lem}

\begin{proof}
	This is a particular case of \cite[Lemma 5.1]{MR518111}. 
\end{proof}

\subsection{Reduction theory}

\begin{lem}\label{artlem64}
  There exists a point $T_+\in\fa_{P_0}^+$ such that for any semi-standard parabolic subgroup $Q$ of $G$, any minimal semi-standard parabolic subgroup $B$ of $G$ contained in $Q$, any $T\in T_+ +\fa_{P_0}^+$ and any $x\in G(\BA)$, we have
  $$ \sum_{\{P:B\subseteq P\subseteq Q\}} \sum_{\delta\in P(F)\bs Q(F)} F^P(\delta x, T) \tau_P^Q(H_P(\delta x)-T_P)=1. $$
\end{lem}

\begin{proof}
  This is \cite[Lemma 6.4]{MR518111} in our case.
\end{proof}

We shall fix such a $T_+$. If $T\in T_+ +\fa_{P_0}^+$, we shall say that $T$ is sufficiently regular. 

\begin{lem}\label{combinatoriclemma1}
  For any relatively standard parabolic subgroup $Q$ of $G$, any sufficiently regular $T$ and any $x\in H(\BA)$, we have
  $$ \sum_{\{P:\wt{P}_0\subseteq P\subseteq Q\}} \sum_{\delta\in P_H(F)\bs Q_H(F)} F^{P}(\delta x, T) \tau_{P}^Q(H_P(\delta x)-T_P)=1. $$
\end{lem}

This is an analogue of \cite[Proposition 2.3]{MR3835522} whose proof relies on \cite[(2.5) in p. 674]{MR3334892} (cf. Lemma \ref{IY(2.5)} below). It is essentially a restricted form to $x\in H(\BA)$ from Lemma \ref{artlem64} for $x\in G(\BA)$. 
We can give a proof close to the steps in an early version of \cite{MR3835522}, which reflects that a main complexity of the truncation here arises from the fact that none of the Siegel sets of $H$ is contained in any Siegel set of $G$, as mentioned in \cite{MR3334892}. However, we shall adopt alternatively the point of view in \cite{10.1007/978-3-319-94833-1_3} to give a more conceptual proof here, which might be useful in other relative trace formulae as well.  

First we introduce a variant (see \cite[\S1.5]{10.1007/978-3-319-94833-1_3}) of some concepts and results in \cite[\S2]{10.1007/978-3-319-94833-1_3} without reproducing proofs. We say that a semi-standard parabolic subgroup $Q$ of $G$ is standard if $P_0\subseteq Q$. For $P\subseteq Q$  a pair of standard parabolic subgroups of $G$, denote by $\rho_P^Q$ the half of the sum of weights (with multiplicities) for the action of $A_P$ on $\fn_P\cap\fm_Q$. We denote by $\ov{\fa_{P_0}^+}$ the closure of $\fa_{P_0}^+$ in $\fa_0$.  

\begin{defn}
  For $g\in G(\BA)$, $Q$ a standard parabolic subgroup of $G$ and $T\in\ov{\fa_{P_0}^+}$, we define the degree of $T$-instability of $g$ with respect to $Q$ by the following formula
  $$ \deg_{T}^Q(g):=\max_{(P,\delta)}\langle \rho_P^Q, H_P(\delta g)-T \rangle $$
where $(P,\delta)$ runs over the pairs of a standard parabolic subgroup $P\subseteq Q$ and an element $\delta\in P(F)\bs Q(F)$. 
\end{defn}

By Lemma \ref{artlem51}, we know that the supremum of $\langle \rho_P^Q, H_P(\delta g)-T \rangle$ in the definition is finite and attainable. 

\begin{lem}[cf. {\cite[Lemme 2.2.1]{10.1007/978-3-319-94833-1_3}}]\label{chau2.2.1}
Let $g\in G(\BA)$, $Q$ be a standard parabolic subgroup of $G$ and $T\in\ov{\fa_{P_0}^+}$. The following two conditions are equivalent: 
\begin{enumerate}[\indent (1)]
\item $\deg_{T}^Q(g)\leq 0$; 
\end{enumerate}
\begin{enumerate}[\indent (2)]
\item for all parabolic subgroup $P\subseteq Q$, all $\delta\in P(F)\bs Q(F)$ and all $\varpi\in\wh{\Delta}_P^Q$, we have $\langle\varpi, H_P(\delta g)-T\rangle\leq 0$. 
\end{enumerate}
\end{lem}

\begin{defn}
  Let $g\in G(\BA)$ and $T\in\ov{\fa_{P_0}^+}$. We say that a pair $(P, \delta)$ of a standard parabolic subgroup $P\subseteq G$ and an element $\delta\in P(F)\bs G(F)$ is $T$-canonical for $g$ if it satisfies the following two conditions: 
\begin{enumerate}[\indent (1)]
\item $\langle \rho_P^G, H_P(\delta g)-T \rangle=\deg_{T}^G(g)$; 
\end{enumerate}
\begin{enumerate}[\indent (2)]
\item for any parabolic subgroup $Q\supseteq P$ such that $\langle \rho_Q^G, H_Q(\delta g)-T \rangle=\deg_{T}^G(g)$, we have $Q=P$. 
\end{enumerate}
\end{defn}

\begin{lem}[cf. {\cite[Lemme 2.3.2]{10.1007/978-3-319-94833-1_3}}]\label{canpaircond}
  Let $g\in G(\BA)$ and $T\in\ov{\fa_{P_0}^+}$. Then $(P, \delta)$ is a $T$-canonical pair for $g$ if and only if it satisfies the following two conditions: 
\begin{enumerate}[\indent (1)]
\item $\deg_{T}^P(\delta g)\leq 0$; 
\end{enumerate}
\begin{enumerate}[\indent (2)]
\item for any $\alpha\in\Delta_P^G$, we have $\langle \alpha, H_P(\delta g)-T \rangle>0$. 
\end{enumerate}
\end{lem}

\begin{prop}[cf. {\cite[Proposition 2.4.1]{10.1007/978-3-319-94833-1_3}}]
  Let $g\in G(\BA)$ and $T\in\ov{\fa_{P_0}^+}$. Then there exists a unique $T$-canonical pair for $g$. 
\end{prop}

Let $T\in \fa_0$ and $Q$ be a standard parabolic subgroup of $G$. Define $\wt{F}^Q(\cdot, T)$ as the characteristic function of $g\in G(\BA)$ such that $\deg_{T}^Q(g)\leq 0$. 

\begin{prop}[cf. {\cite[Proposition 2.5.1]{10.1007/978-3-319-94833-1_3}}]
  For $g\in G(\BA)$, $Q$ a standard parabolic subgroup of $G$ and $T\in\ov{\fa_{P_0}^+}$, we have
\begin{enumerate}[\indent (1)]
\item $$ \sum_{\{P:P_0\subseteq P\subseteq Q\}} \sum_{\delta\in P(F)\bs Q(F)} \wt{F}^{P}(\delta g, T) \tau_{P}^Q(H_P(\delta g)-T_P)=1; $$
\end{enumerate}
\begin{enumerate}[\indent (2)]
\item $$ \wt{F}^Q(g, T)=\sum_{\{P:P_0\subseteq P\subseteq Q\}} (-1)^{\dim(A_P/A_{Q})}  \sum_{\delta\in P(F)\bs Q(F)} \wh{\tau}_P^Q(H_{P}(\delta g)-T_P). $$
\end{enumerate}
\end{prop}

Since we have similar formulae for $F^Q(\cdot, T)$ for sufficiently regular $T$ (see Lemma \ref{artlem64}), we know that $\wt{F}^Q(\cdot, T)=F_{P_0}^Q(\cdot, T)$ for such $T$. Now we can return to the proof of Lemma \ref{combinatoriclemma1}. 

\begin{proof}[Proof of Lemma \ref{combinatoriclemma1}]
  It is noticeable that the identity is reduced to its analogues for semi-standard Levi factors of $Q$, which is a product of $GL_{p_i+q_i, D}$ whose intersection with $H$ is $GL_{p_i, D}\times GL_{q_i, D}$. By induction on the rank of $G$, it suffices to prove the identity for $Q=G$. 

  For a standard parabolic subgroup $P$ of $G$, fix a set of representatives $\Omega_{P,G}$ in $\{s\in\Omega^G|\wt{P}_0\subseteq s^{-1}P\}$ for the relation $s_1\sim s_2$ if and only if $s_2 s_1^{-1}\in\Omega^{M_P}$. We can rewrite the equality in the lemma as
  $$ \sum_{\{P:P_0\subseteq P\}} \sum_{s\in\Omega_{P,G}} \sum_{\delta\in (s^{-1}P)_H(F)\bs H(F)} F_{P_0}^{P}(\omega_s\delta x, T_{P_0}) \tau_{P}^G(H_P(\omega_s\delta x)-T_P)=1. $$
  In fact, this follows from
  $$ F^{s^{-1}P}(\delta x,T)=F_{s^{-1}P_0}^{s^{-1}P}(\delta x,T_{s^{-1}P_0})=F_{P_0}^P(\omega_s\delta x,T_{P_0}) $$
  and
  $$ \tau_{s^{-1}P}^G(H_{s^{-1}P}(\delta x)-T_{s^{-1}P})=\tau_{P}^G(H_P(\omega_s\delta x)-T_P). $$

  Combining the double sums over $s$ and $\delta$, we claim that the equality above is equivalent to
  $$ \sum_{\{P:P_0\subseteq P\}} \sum_{\delta\in P(F)\bs P(F)\Omega^G H(F)} F_{P_0}^{P}(\delta x, T_{P_0}) \tau_{P}^G(H_P(\delta x)-T_P)=1. $$
  In fact, for any $s\in\Omega_{P,G}$, consider the map
  $$ (s^{-1}P)_H(F)\bs H(F)\ra P(F)\bs P(F)\Omega^G H(F),\delta\mapsto\omega_s\delta. $$
  Firstly, it is well-defined: if $\delta_1=\omega_s^{-1}p\omega_s\delta_2$ with $p\in P(F)$, then $\omega_s\delta_1=p\omega_s\delta_2$. Secondly, it is injective: if $\omega_s\delta_1=p\omega_s\delta_2$ with $p\in P(F)$, then $\delta_1=\omega_s^{-1}p\omega_s\delta_2$ with $\omega_s^{-1}p\omega_s=\delta_1\delta_2^{-1}\in (s^{-1}P)_H(F)$. Thirdly, for $s_1\neq s_2$ in $\Omega_{P,G}$, we have $\omega_{s_1}\delta_1\neq p\omega_{s_2}\delta_2$ with $p\in P(F)$: otherwise, $s_1^{-1}P=(\delta_2\delta_1^{-1})^{-1}(s_2^{-1}P)(\delta_2\delta_1^{-1})$ with $\delta_2\delta_1^{-1}\in H(F)$, so $(s_1^{-1}P)_H=(\delta_2\delta_1^{-1})^{-1}(s_2^{-1}P)_H(\delta_2\delta_1^{-1})$, and then $\delta_2\delta_1^{-1}\in (s_1^{-1}P)_H(F)=(s_2^{-1}P)_H(F)$ for both of $(s_1^{-1}P)_H$ and $(s_2^{-1}P)_H$ are standard parabolic subgroups of $H$, which implies $s_1^{-1}P=s_2^{-1}P$ contradicting $s_1\neq s_2$. Fourthly, any $\wt{s}\in \Omega^G$ appears in the image of the map for some $s\in\Omega_{P,G}$: since $(\wt{s}^{-1}P)_H$ is a semi-standard parabolic subgroup of $H$, there exists an $s_0\in\Omega^H$ such that $\wt{P}_0\subseteq s_{0}^{-1}((\wt{s}^{-1}P)_H)=(s_{0}^{-1}(\wt{s}^{-1}P))_H=((\wt{s}s_0)^{-1}P)_H$, i.e., $\wt{s}s_0\in\Omega_{P, G}$. To sum up, we finish the argument of the claim. 

  It suffices to prove an analogue of the last equality by replacing $F_{P_0}^P$ with $\wt{F}^P$ for $T\in\ov{\fa_{P_0}^+}$, as they are identical for sufficiently regular $T$. That is to say, for $x\in H(\BA)=G^\theta(\BA)$, if $(P, \delta)$ is the unique $T$-canonical pair for $x$, we need to prove that $\delta\in P(F)\bs P(F)\Omega^G H(F)$. Recall that $\theta(g)=\epsilon g\epsilon^{-1}$ for $g\in G(\BA)$, where $\epsilon=
\left( \begin{array}{cc}
1_p & 0 \\
0 & -1_q \\
\end{array} \right)$. Since $\epsilon\in M_0(F)\cap K$, from Lemma \ref{canpaircond}, we deduce that $(P, \delta)$ is the unique $T$-canonical pair for $g\in G(\BA)$ if and only if $(P, \theta(\delta))$ is the unique $T$-canonical pair for $\theta(g)$. In particular, if $(P, \delta)$ is the unique $T$-canonical pair for $x\in H(\BA)$, we have $\delta=\theta(\delta)$. Denote by $\delta_0$ a representative of $\delta\in P(F)\bs G(F)$. Then $\delta_0\epsilon\delta_0^{-1}\in P(F)$. 

Suppose that $\delta_0\epsilon\delta_0^{-1}=mu$, where $m\in M_P(F)$ and $u\in N_P(F)$. Both of $mu$ and $m$ are semi-simple in $G(F)$ (in the classical sense) for $(mu)^2=m^2=1$. Applying \cite[Lemma 2.1]{MR518111} to the characteristic function of the singleton $\{u\}$, one obtains that $mu$ is $N_P(F)$-conjugated to $mu'$ for some $u'\in N_P(F)$ such that $mu'=u'm$. Since both of $mu'$ and $m$ are semi-simple in $G(F)$, by the uniqueness of Jordan decomposition, we have $u'=1$, i.e., $\delta_0\epsilon\delta_0^{-1}$ is $N_P(F)$-conjugated to $m$. By linear algebra, $m$ is $M_P(F)$-conjugated to a diagonal matrix with entries $\{\pm 1\}$ with expected multiplicities $p$ and $q$ respectively. In sum, $\delta_0\epsilon\delta_0^{-1}$ is $P(F)$-conjugated to $\omega_s\epsilon\omega_s^{-1}$ for some $s\in\Omega^G$. Suppose that $p_0\in P(F)$ satisfies $\delta_0\epsilon\delta_0^{-1}=p_0(\omega_s\epsilon\omega_s^{-1})p_0^{-1}$. Then $\omega_s^{-1}p_0^{-1}\delta_0\in G^\theta(F)=H(F)$, i.e., $\delta=P(F)\delta_0\in P(F)\bs P(F)\Omega^G H(F)$. 
\end{proof}

\begin{lem}\label{IY(2.5)}
  Let $P$ be a relatively standard parabolic subgroup of $G$. For any $a\in A_{\wt{P}_0}^\infty(P_{H},t_0)$, there exists a relatively standard minimal parabolic subgroup $B\subseteq P$ such that $a\in A_B^\infty(P,t_0)$. 
\end{lem}

\begin{proof}
 This is an analogue of \cite[(2.5) in p. 674]{MR3334892}. 
 To begin with, we reduce ourselves to proving the case $P=G$. For this purpose, suppose that the assertion is true in this case, and consider the case $P\neq G$. Since $A_{\wt{P}_0}^\infty(P_{H},t_0)=A_{\wt{P}_0\cap M_P}^\infty(M_{P_H},t_0)$, by our assumption, there exists a minimal parabolic subgroup $B_\ast$ of $M_P$ containing $\wt{P}_0\cap M_P$ such that $a\in A_{B_\ast}^\infty(M_P,t_0)$. Let $B:=B_\ast N_P$. Then $B$ is a relatively standard minimal parabolic subgroup of $G$ contained in $P$, and $A_{B_\ast}^\infty(M_P,t_0)=A_B^\infty(P,t_0)$. Thus $a\in A_B^\infty(P,t_0)$, which implies the lemma. Therefore, we may and shall only consider the case $P=G$ in the rest of the proof. 

Let $a\in A_{\wt{P}_0}^\infty(H,t_0)$. Then $a=\diag(a_1,\cdot\cdot\cdot,a_{p+q})$, where $\frac{a_i}{a_{i+1}}>e^{t_0}$ for $1\leq i\leq p-1$ and $p+1\leq i\leq p+q-1$. In the definition of Siegel sets, we suppose that $t_0<0$, so $0<e^{t_0}<1$. Note that $A_{P_0}^\infty(G,t_0)=\{\diag(b_1,\cdot\cdot\cdot,b_{p+q})|\frac{b_i}{b_{i+1}}>e^{t_0},\forall 1\leq i\leq p+q-1\}$. Thus we need to show that there exists a permutation $s\in\Omega^G$ such that $s\cdot a=\diag(a_{s^{-1}(1)},\cdot\cdot\cdot,a_{s^{-1}(p+q)})$ satisfies the following two conditions: 
\begin{enumerate}[\indent (1)]
\item $s(i)<s(i+1)$ for $1\leq i\leq p-1$ and $p+1\leq i\leq p+q-1$; 
\end{enumerate}
\begin{enumerate}[\indent (2)]
\item $\frac{a_{s^{-1}(i)}}{a_{s^{-1}(i+1)}}>e^{t_0}$ for $1\leq i\leq p+q-1$. 
\end{enumerate}

Firstly, we show that one can move $a_{p+1}$ to its left hand side in $(a_1,\cdot\cdot\cdot,a_{p+q})$ such that both the first $p+1$ elements and the last $q-1$ ones in the new sequence are in ``good'' order (which means that the quotient of any consecutive pairs is $>e^{t_0}$), while keeping the original relative orders among $(a_1,\cdot\cdot\cdot,a_p)$ and among $(a_{p+1},\cdot\cdot\cdot,a_{p+q})$. If $\frac{a_p}{a_{p+1}}>e^{t_0}$, we are already done (one can take $s=1$). In general, write
  $$ i_1:=\max\left\{0,\max\left\{1\leq i\leq p\bigg|\frac{a_i}{a_{p+1}}>e^{t_0}\right\}\right\}. $$
When $1\leq i_1\leq p-1$, since $e^{t_0}<1$, $\frac{a_{i_1+1}}{a_{p+1}}\leq e^{t_0}$ implies $\frac{a_{p+1}}{a_{i_1+1}}\geq e^{-t_0}>1$; there is an $s\in\Omega^G$ such that $s\cdot a=\diag(a_1,\cdot\cdot\cdot,a_{i_1},a_{p+1},a_{i_1+1},\cdot\cdot\cdot,a_p,a_{p+2},\cdot\cdot\cdot,a_{p+q})$. When $i_1=0$, which implies $a_{p+1}>a_1$, there is an $s\in\Omega^G$ such that $s\cdot a=\diag(a_{p+1},a_1,\cdot\cdot\cdot,a_p,a_{p+2},\cdot\cdot\cdot,a_{p+q})$. 

Secondly, we consider moving $a_{p+2}$ as before. One should check that $a_{p+2}$ will not exceed the new place of $a_{p+1}$, which results from the fact that $\frac{a_{p+1}}{a_{p+2}}>e^{t_0}$. Thus one can move $a_{p+1}$ and $a_{p+2}$ to their left hand side in $(a_1,\cdot\cdot\cdot,a_{p+q})$ such that both the first $p+2$ elements and the last $q-2$ ones in the new sequence are in "good" order, while still keeping the original relative orders among $(a_1,\cdot\cdot\cdot,a_p)$ and among $(a_{p+1},\cdot\cdot\cdot,a_{p+q})$. 

To finish the argument of our claim, it suffices to move $a_{p+3},\cdot\cdot\cdot, a_{p+q}$ one by one as above. After moving $a_{i}$ to its left hand side, where $p+3\leq i\leq p+q$, one requires that both the first $i$ elements and the last $p+q-i$ ones in the new sequence are in "good" order, while the original relative orders among $(a_1,\cdot\cdot\cdot,a_p)$ and among $(a_{p+1},\cdot\cdot\cdot,a_{p+q})$ are kept. As in the second step above, this is possible since $\frac{a_{i-1}}{a_{i}}>e^{t_0}$. After one moves the last element $a_{p+q}$, the sequence satisfies the desired two conditions. 
\end{proof}

\begin{prop}\label{propF^P(a)}
Let $B$ be an arbitrary minimal semi-standard parabolic subgroup of $G$. Let $P$ be a parabolic subgroup of $G$ containing $B$. Suppose that $T$ is sufficiently regular. If $m\in\varrho_B\cap M_P(\BA), a\in A_B^\infty(P, t_0)$ and $k\in K\cap M_P(\BA)$ satisfy $F_B^P(mak, T_B)=1$, then $a\in A_B^\infty(P,t_0,T_B)$. 
\end{prop}

\begin{proof}
It results from Lemma \ref{chau2.2.1}, since $\wt{F}_B^P(\cdot, T)=F_B^P(\cdot, T)$ for sufficiently regular $T$. Here we write $\wt{F}_B^P(\cdot, T)$ for $\wt{F}^P(\cdot, T)$ when $B$ plays the role of $P_0$. 
\end{proof}

  For a relatively standard parabolic subgroup $P$ of $G$, denote by $\CP(\wt{P}_0,P)$ the set of relatively standard minimal parabolic subgroups of $G$ contained in $P$. For $B\in\CP(\wt{P}_0,P)$, write  
$$ A_B^{G,\infty}(P,t_0):=A_B^\infty(P,t_0)\cap G(\BA)^1 $$
and for all $T\in\fa_0$, 
$$ A_B^{G,\infty}(P,t_0,T):=A_B^\infty(P,t_0,T)\cap G(\BA)^1. $$

\begin{coro}\label{zydcor2.5}
 Let $P$ be a relatively standard parabolic subgroup of $G$. For sufficiently regular $T$, the following subset of $M_{P_{H}}(\BA)\cap G(\BA)^1$
  $$ \bigcup_{B\in \CP(\wt{P}_0,P)} (\varrho_{\wt{P}_0}\cap M_{P_{H}}(\BA)) \cdot (A_{\wt{P}_0}^\infty(P_{H},t_0)\cap A_B^{G,\infty}(P,t_0,T_B)) \cdot (K_H\cap M_{P_{H}}(\BA)) $$
  projects surjectively on $\{m\in M_{P_{H}}(F)\bs M_{P_{H}}(\BA)\cap G(\BA)^1|F^{P}(m,T)=1\}$.
\end{coro}

\begin{proof}
This is an analogue of \cite[Corollaire 2.5]{MR3835522}. By Lemma \ref{IY(2.5)}, the following subset of $M_{P_{H}}(\BA)\cap G(\BA)^1$
  $$ \bigcup_{B\in \CP(\wt{P}_0,P)} (\varrho_{\wt{P}_0}\cap M_{P_{H}}(\BA)) \cdot (A_{\wt{P}_0}^\infty(P_{H},t_0)\cap A_B^{G,\infty}(P,t_0)) \cdot (K_H\cap M_{P_{H}}(\BA)) $$
  projects surjectively on $M_{P_{H}}(F)\bs M_{P_{H}}(\BA)\cap G(\BA)^1$. Recall that $\varrho_{\wt{P}_0}\subseteq\varrho_B$ for all $B\in \CP(\wt{P}_0,P)$ and that $K_H\subseteq K$ by our choices (see Section \ref{relstdpar1}). Therefore, the statement to be proved follows from Proposition \ref{propF^P(a)}. 
\end{proof}

\subsection{Integrability}

\begin{thm}\label{convergence1}
  For all sufficiently regular $T$ and all $s\in \BR$,
  $$ \sum_{\fo\in\CO} \int_{H(F)\bs H(\BA)\cap G(\BA)^1} |k_{f,\fo}^T(x)| |\Nrd (x_1)|_\BA^s dx < \infty, $$
  where we write $x=(x_1, x_2)\in GL_{p,D}(\BA)\times GL_{q,D}(\BA)$. 
\end{thm}

\begin{proof}
  Let $P_1\subseteq P_2$ be a pair of relatively standard parabolic subgroups of $G$. Following \cite[\S6]{MR518111}, for $T_1\in \fa_{P_1}$, we define the characteristic function
  $$ \sigma_{P_1}^{P_2}(T_1):=\sum_{\{Q:P_2\subseteq Q\}} (-1)^{\dim(A_{P_2}/A_Q)} \tau_{P_1}^Q(T_1) \wh{\tau}_Q^G(T_1). $$
  Recall that for $P\supseteq P_1$ a relatively standard parabolic subgroup of $G$, we have (see \cite[p. 943]{MR518111})
  $$ \tau_{P_1}^P(T_1)\wh{\tau}_P^G(T_1)=\sum_{\{P_2:P\subseteq P_2\}}\sigma_{P_1}^{P_2}(T_1). $$
  Denote $P_{1,H}:=P_1\cap H$. For $x\in P_{1,H}(F)\bs H(\BA)$, we put
  $$ \chi_{P_1, P_2}^T(x):=F^{P_1}(x, T)\sigma_{P_1}^{P_2}(H_{P_1}(x)-T_{P_1}), $$
  and
  $$ k_{P_1, P_2, \fo}(x):=\sum_{\{P: P_1\subseteq P\subseteq P_2\}} (-1)^{\dim(A_P/A_{G})} k_{f,P,\fo}(x). $$

  Using Lemma \ref{combinatoriclemma1} and the left invariance of $H_P$ and $k_{f,P,\fo}$ by $P_H(F)$, we have
\begin{equation}\label{eqpfint}
 k_{f, \fo}^T(x)=\sum_{\{P_1,P_2: \wt{P}_0\subseteq P_1\subseteq P_2\}} \sum_{\delta\in P_{1,H}(F)\bs H(F)} \chi_{P_1, P_2}^T(\delta x) k_{P_1, P_2, \fo}(\delta x) 
\end{equation}
for $x\in H(F)\bs H(\BA)$. Thus
  \[\begin{split} 
        &\sum_{\fo\in\CO} \int_{H(F)\bs H(\BA)\cap G(\BA)^1} |k_{f,\fo}^T(x)| |\Nrd (x_1)|_\BA^s dx \\
  \leq &\sum_{\fo\in\CO} \sum_{\{P_1,P_2: \wt{P}_0\subseteq P_1\subseteq P_2\}} \int_{P_{1,H}(F)\bs H(\BA)\cap G(\BA)^1} \chi_{P_1, P_2}^T(x) |k_{P_1, P_2, \fo}(x)| |\Nrd (x_1)|_\BA^s dx. 
  \end{split}\]
  It suffices to prove that for any pair of relatively standard parabolic subgroups $P_1\subseteq P_2$ of $G$,
  $$ \sum_{\fo\in\CO} \int_{P_{1,H}(F)\bs H(\BA)\cap G(\BA)^1} \chi_{P_1, P_2}^T(x) |k_{P_1, P_2, \fo}(x)| |\Nrd (x_1)|_\BA^s dx <\infty. $$
  If $P_1=P_2\neq G$, by \cite[Lemma 6.1]{MR518111}, we have $\sigma_{P_1}^{P_2}=0$ and then $\chi_{P_1, P_2}^T=0$, so the integration is zero. 
 If $P_1=P_2=G$,  by Corollary \ref{zydcor2.5}, every $x\in H(F)\bs H(\BA)\cap G(\BA)^1$ with $F^G(x,T)=1$ has a representative in the compact subset
  $$ \bigcup_{B\in \CP(\wt{P}_0,G)} \varrho_{\wt{P}_0} \cdot A_B^{G,\infty}(G,t_0,T_B) \cdot K_H, $$
so the integral is bounded by an integral of a continuous function over a compact subset and thus convergent. Therefore, we reduce ourselves to proving the following proposition.
\end{proof}

\begin{prop}\label{proppfofconv1}
  Let $f\in\CS(\fs(\BA))$, $s\in\BR$ and $P_1\subsetneq P_2$ be two relatively standard parabolic subgroups of $G$. Fix any two positive real numbers $\epsilon_0$ and $N$. Then there exists a constant $C$ such that
  $$ \sum_{\fo\in\CO} \int_{P_{1,H}(F)\bs H(\BA)\cap G(\BA)^1} \chi_{P_1, P_2}^T(x) |k_{P_1, P_2, \fo}(x)| |\Nrd (x_1)|_\BA^s dx\leq Ce^{-N\parallel T\parallel} $$
  for all sufficiently regular $T$ satisfying $\alpha(T)\geq\epsilon_0\parallel T\parallel$ for any $\alpha\in\Delta_{P_0}^G$.
\end{prop}

For $x\in H(F)\bs H(\BA)$, define
$$ k_{f,G}(x):=\sum_{\fo\in\CO} k_{f,G,\fo}(x)=\sum_{X\in\fs(F)} f(x^{-1}Xx) $$
and
\begin{equation}\label{deftruncation2}
 k_{f}^T(x):=\sum_{\fo\in\CO}k_{f,\fo}^{T}(x). 
\end{equation}

\begin{coro}\label{comparewithnaive1}
  Let $f\in\CS(\fs(\BA))$ and $s\in\BR$. Fix any two positive real numbers $\epsilon_0$ and $N$. Then there exists a constant $C$ such that
  $$ \int_{H(F)\bs H(\BA)\cap G(\BA)^1} |k_{f}^{T}(x)-F^{G}(x,T)k_{f,G}(x)||\Nrd (x_1)|_\BA^s dx\leq Ce^{-N\parallel T\parallel} $$
  for all sufficiently regular $T$ satisfying $\alpha(T)\geq\epsilon_0\parallel T\parallel$ for any $\alpha\in\Delta_{P_0}^G$.
\end{coro}

\begin{proof}
	For $x\in H(F)\bs H(\BA)$, we have 
	$$ k_{f}^{T}(x)-F^{G}(x,T)k_{f,G}(x)=\sum_{\fo\in\CO}k_{f,\fo}^{T}(x)-F^{G}(x,T)\sum_{\fo\in\CO} k_{f,G,\fo}(x)=\sum_{\fo\in\CO} (k_{f,\fo}^{T}(x)-F^{G}(x,T)k_{f,G,\fo}(x)). $$
	By \eqref{eqpfint}, since $\chi_{P_1, P_2}^T=0$ for $P_1=P_2\neq G$, we have 
	$$ k_{f,\fo}^{T}(x)-F^{G}(x,T)k_{f,G,\fo}(x)=\sum_{\{P_1,P_2: \wt{P}_0\subseteq P_1\subsetneq P_2\}} \sum_{\delta\in P_{1,H}(F)\bs H(F)} \chi_{P_1, P_2}^T(\delta x) k_{P_1, P_2, \fo}(\delta x). $$
	Therefore, 
	\[\begin{split}
		&\int_{H(F)\bs H(\BA)\cap G(\BA)^1} |k_{f}^{T}(x)-F^{G}(x,T)k_{f,G}(x)||\Nrd (x_1)|_\BA^s dx \\ 
		\leq&\sum_{\fo\in\CO} \int_{H(F)\bs H(\BA)\cap G(\BA)^1} | k_{f,\fo}^{T}(x)-F^{G}(x,T)k_{f,G,\fo}(x)||\Nrd (x_1)|_\BA^s dx \\ 
		\leq&\sum_{\fo\in\CO} \sum_{\{P_1,P_2: \wt{P}_0\subseteq P_1\subsetneq P_2\}} \int_{P_{1, H}(F)\bs H(\BA)\cap G(\BA)^1} \chi_{P_1, P_2}^T(x) |k_{P_1, P_2, \fo}(x)||\Nrd (x_1)|_\BA^s dx. 
	\end{split}\]
	We now conclude by applying Proposition \ref{proppfofconv1} to the last expression. 
\end{proof}

\begin{proof}[Proof of Proposition \ref{proppfofconv1}]
  Let $P$ be a relatively standard parabolic subgroup of $G$ such that $P_1\subseteq P\subseteq P_2$. For any $X\in\fm_P(F)\cap\fo$, there exists a unique relatively standard parabolic subgroup $R$ of $G$ such that $P_1\subseteq R\subseteq P$ and $X\in (\fm_P(F)\cap\fr(F)\cap\fo)-\left(\bigcup\limits_{P_1\subseteq Q\subsetneq R} \fm_P(F)\cap\fq(F)\cap\fo\right)$. Write
  $$ \wt{\fm}_{P_1}^R:=\fm_R-\left(\bigcup_{\{Q:P_1 \subseteq Q \subsetneq R\}} \fm_R\cap\fq\right) $$
  and
  $$ \fn_R^P:=\fn_R\cap\fm_P. $$
  By Corollary \ref{orbit1}, we have
  $$ (\fm_P(F)\cap\fr(F)\cap\fo)-\left(\bigcup\limits_{P_1\subseteq Q\subsetneq R} \fm_P(F)\cap\fq(F)\cap\fo\right)=(\wt{\fm}_{P_1}^R(F)\cap\fo)\oplus((\fn_R^P\cap\fs)(F)). $$
  Hence
  \[\begin{split}
   k_{f,P,\fo}(x)&=\sum_{X\in\fm_P(F)\cap\fo} \int_{\fn_P\cap\fs(\BA)} f(x^{-1}(X+U)x) dU \\
   &=\sum_{\{R:P_1\subseteq R\subseteq P\}} \sum_{\xi\in\wt{\fm}_{P_1}^R(F)\cap\fo} \sum_{X\in(\fn_R^P\cap\fs)(F)} \int_{(n_P\cap\fs)(\BA)} f(x^{-1}(\xi+X+U)x) dU. 
  \end{split}\]

  Denote by $\ov{P}$ the semi-standard parabolic subgroup of $G$ opposite to $P$ and write
  $$ \ov{\fn}_R^P:=\fn_{\ov{R}}\cap\fm_P. $$
  Note that the restriction of $\langle\cdot,\cdot\rangle$ (defined in (\ref{bilform})) to $((\fn_R^P\cap\fs)(\BA))\times((\ov{\fn}_R^P\cap\fs)(\BA))$ is also non-degenerate. For any $\xi\in(\fm_R\cap\fs)(\BA)$, applying the Poisson summation formula to the Bruhat-Schwartz function $\int_{(\fn_P\cap\fs)(\BA)} f(x^{-1}(\xi+\cdot+U)x) dU$, we get
  $$ \sum_{X\in(\fn_R^P\cap\fs)(F)} \int_{(\fn_P\cap\fs)(\BA)} f(x^{-1}(\xi+X+U)x) dU = \sum_{\wh{X}\in(\ov{\fn}_R^P\cap\fs)(F)} \Phi_\xi^{x,R}(\wh{X}), $$
  where the partial Fourier transform $\Phi_\xi^{x,R}$ of $\int_{(\fn_P\cap\fs)(\BA)} f(x^{-1}(\xi+\cdot+U)x) dU$ is defined by (recall the notation $\Psi$ in Section \ref{secft})
  $$ \Phi_\xi^{x,R}(\wh{X}):=\int_{(\fn_R^P\cap\fs)(\BA)} \left(\int_{(\fn_P\cap\fs)(\BA)} f(x^{-1}(\xi+X+U)x) dU\right) \Psi(\langle X,\wh{X}\rangle) dX $$
  for all $\wh{X}\in (\ov{\fn}_R^P\cap\fs)(\BA)$. 
  Since $\langle U,\wh{X}\rangle=0$ for $U\in (\fn_P\cap\fs)(\BA)$ and $\wh{X}\in(\ov{\fn}_R^P\cap\fs)(\BA)$, as well as $\fn_R\cap\fs=(\fn_P\cap\fs)\oplus(\fn_R^P\cap\fs)$, we have
  $$ \Phi_\xi^{x,R}(\wh{X})=\int_{(\fn_R\cap\fs)(\BA)} f(x^{-1}(\xi+U)x)\Psi(\langle U,\wh{X}\rangle) dU $$
  for all $\wh{X}\in (\ov{\fn}_R^P\cap\fs)(\BA)$, 
  whose expression is actually independent of $P$.

  To sum up,
  $$ k_{f,P,\fo}(x)=\sum_{\{R:P_1\subseteq R\subseteq P\}} \sum_{\xi\in\wt{\fm}_{P_1}^R(F)\cap\fo} \sum_{\wh{X}\in(\ov{\fn}_R^P\cap\fs)(F)} \Phi_\xi^{x,R}(\wh{X}). $$
  Hence
  \[\begin{split}
    k_{P_1, P_2, \fo}(x)&=\sum_{\{P: P_1\subseteq P\subseteq P_2\}} (-1)^{\dim(A_P/A_{G})} k_{f,P,\fo}(x) \\
    &=\sum_{\{P: P_1\subseteq P\subseteq P_2\}} (-1)^{\dim(A_P/A_{G})} \left(\sum_{\{R:P_1\subseteq R\subseteq P\}} \sum_{\xi\in\wt{\fm}_{P_1}^R(F)\cap\fo} \sum_{\wh{X}\in(\ov{\fn}_R^P\cap\fs)(F)} \Phi_\xi^{x,R}(\wh{X})\right) \\
    &=\sum_{\{R:P_1\subseteq R\subseteq P_2\}} \sum_{\xi\in\wt{\fm}_{P_1}^R(F)\cap\fo} \left(\sum_{\{P:R\subseteq P\subseteq P_2\}} (-1)^{\dim(A_P/A_{G})} \sum_{\wh{X}\in(\ov{\fn}_R^P\cap\fs)(F)} \Phi_\xi^{x,R}(\wh{X})\right)
  \end{split}\]
 For a relatively standard parabolic subgroup $P_3$ of $G$ containing $R$, we write
  $$ (\ov{\fn}_R^{P_3})':=\ov{\fn}_R^{P_3}-\left(\bigcup_{\{Q:R\subseteq Q\subsetneq P_3\}}\ov{\fn}_R^Q\right). $$
Then
$$ \sum_{\wh{X}\in(\ov{\fn}_R^P\cap\fs)(F)} \Phi_\xi^{x,R}(\wh{X})=\sum_{\{P_3: R\subseteq P_3\subseteq P\}} \sum_{\wh{X}\in((\ov{\fn}_R^{P_3})'\cap\fs)(F)} \Phi_\xi^{x,R}(\wh{X}). $$
  We have
\[\begin{split} 
&\sum_{\{P:R\subseteq P\subseteq P_2\}} (-1)^{\dim(A_P/A_{G})} \sum_{\wh{X}\in(\ov{\fn}_R^P\cap\fs)(F)} \Phi_\xi^{x,R}(\wh{X}) \\
=&\sum_{\{P:R\subseteq P\subseteq P_2\}} (-1)^{\dim(A_P/A_{G})} \sum_{\{P_3: R\subseteq P_3\subseteq P\}} \sum_{\wh{X}\in((\ov{\fn}_R^{P_3})'\cap\fs)(F)} \Phi_\xi^{x,R}(\wh{X}) \\
=&(-1)^{\dim(A_{P_2}/A_{G})} \sum_{\{P_3: R\subseteq P_3\subseteq P_2\}} \sum_{\wh{X}\in((\ov{\fn}_R^{P_3})'\cap\fs)(F)} \Phi_\xi^{x,R}(\wh{X}) \sum_{\{P:P_3\subseteq P\subseteq P_2\}} (-1)^{\dim(A_P/A_{P_2})}. \\
\end{split}\]
From \cite[Proposition 1.1]{MR518111}, we know that
\begin{displaymath}
\sum_{\{P:P_3\subseteq P\subseteq P_2\}} (-1)^{\dim(A_P/A_{P_2})} = \left\{ \begin{array}{ll}
1, & \textrm{if $P_3=P_2$; }\\
0, & \textrm{otherwise. }\\
\end{array} \right.
\end{displaymath}
We obtain
$$ \sum_{\{P:R\subseteq P\subseteq P_2\}} (-1)^{\dim(A_P/A_{G})} \sum_{\wh{X}\in(\ov{\fn}_R^P\cap\fs)(F)} \Phi_\xi^{x,R}(\wh{X})=(-1)^{\dim(A_{P_2}/A_{G})} \sum_{\wh{X}\in((\ov{\fn}_R^{P_2})'\cap\fs)(F)} \Phi_\xi^{x,R}(\wh{X}). $$
  Thus
  $$ k_{P_1, P_2, \fo}(x)=(-1)^{\dim(A_{P_2}/A_{G})} \sum_{\{R:P_1\subseteq R\subseteq P_2\}} \sum_{\xi\in\wt{\fm}_{P_1}^R(F)\cap\fo} \sum_{\wh{X}\in((\ov{\fn}_R^{P_2})'\cap\fs)(F)} \Phi_\xi^{x,R}(\wh{X}).  $$

  Now
  \[\begin{split}
    &\sum_{\fo\in\CO} \int_{P_{1,H}(F)\bs H(\BA)\cap G(\BA)^1} \chi_{P_1, P_2}^T(x) |k_{P_1, P_2, \fo}(x)| |\Nrd (x_1)|_\BA^s dx \\
    \leq &\sum_{\fo\in\CO} \int_{P_{1,H}(F)\bs H(\BA)\cap G(\BA)^1} \chi_{P_1, P_2}^T(x) \left( \sum_{\{R:P_1\subseteq R\subseteq P_2\}} \sum_{\xi\in\wt{\fm}_{P_1}^R(F)\cap\fo} \sum_{\wh{X}\in((\ov{\fn}_R^{P_2})'\cap\fs)(F)} |\Phi_\xi^{x,R}(\wh{X})| \right) |\Nrd (x_1)|_\BA^s dx \\
    = &\sum_{\{R:P_1\subseteq R\subseteq P_2\}} \int_{P_{1,H}(F)\bs H(\BA)\cap G(\BA)^1} \chi_{P_1, P_2}^T(x) \sum_{\xi\in(\wt{\fm}_{P_1}^R\cap\fs)(F)} \sum_{\wh{X}\in((\ov{\fn}_R^{P_2})'\cap\fs)(F)} |\Phi_\xi^{x,R}(\wh{X})| |\Nrd (x_1)|_\BA^s dx. 
  \end{split}\]
  We reduce ourselves to bounding
  \begin{equation}\label{equation1.1}
    \int_{P_{1,H}(F)\bs H(\BA)\cap G(\BA)^1} \chi_{P_1, P_2}^T(x) \sum_{\xi\in(\wt{\fm}_{P_1}^R\cap\fs)(F)} \sum_{\wh{X}\in((\ov{\fn}_R^{P_2})'\cap\fs)(F)} |\Phi_\xi^{x,R}(\wh{X})| |\Nrd (x_1)|_\BA^s dx
  \end{equation}
  for any fixed relatively standard parabolic subgroup $R$ of $G$ such that $P_1\subseteq R\subseteq P_2$.

  By Iwasawa decomposition and our choices of measures, the integral over $P_{1,H}(F)\bs H(\BA)\cap G(\BA)^1$ can be decomposed as integrals over 
  $$ (n_1,m_1,k)\in N_{P_{1,H}}(F)\bs N_{P_{1,H}}(\BA)\times M_{P_{1,H}}(F)\bs M_{P_{1,H}}(\BA)\cap G(\BA)^1\times K_H. $$
  Then
  \[\begin{split}
    &\int_{P_{1,H}(F)\bs H(\BA)\cap G(\BA)^1} \chi_{P_1, P_2}^T(x) \sum_{\xi\in(\wt{\fm}_{P_1}^R\cap\fs)(F)} \sum_{\wh{X}\in((\ov{\fn}_R^{P_2})'\cap\fs)(F)} |\Phi_\xi^{x,R}(\wh{X})| |\Nrd (x_1)|_\BA^s dx \\
    =&\int_{K_H} \int_{M_{P_{1,H}}(F)\bs M_{P_{1,H}}(\BA)\cap G(\BA)^1} \int_{N_{P_{1,H}}(F)\bs N_{P_{1,H}}(\BA)} F^{P_1}(m_1,T) \sigma_{P_1}^{P_2}(H_{P_1}(m_1)-T_{P_1}) \\
    &\cdot \sum_{\xi\in(\wt{\fm}_{P_1}^R\cap\fs)(F)} \sum_{\wh{X}\in((\ov{\fn}_R^{P_2})'\cap\fs)(F)} |\Phi_\xi^{n_1 m_1 k,R}(\wh{X})| e^{-2\rho_{P_{1,H}}(H_{P_{1,H}}(m_1))} |\Nrd (m_{1, 1})|_\BA^s dn_1 dm_1 dk, 
  \end{split}\]
where we write $m_1=(m_{1,1}, m_{1, 2})\in GL_{p, D}(\BA)\times GL_{q, D}(\BA)$. 

By Corollary \ref{zydcor2.5},  the following subset of $M_{P_{1,H}}(\BA)\cap G(\BA)^1$%
  $$ \bigcup_{B\in \CP(\wt{P}_0,P_1)} (\varrho_{\wt{P}_0}\cap M_{P_{1,H}}(\BA)) \cdot (A_{\wt{P}_0}^\infty(P_{1,H},t_0)\cap A_B^{G,\infty}(P_1,t_0,T_B)) \cdot (K_H\cap M_{P_{1,H}}(\BA)) $$
  projects surjectively on $\{m_1\in M_{P_{1,H}}(F)\bs M_{P_{1,H}}(\BA)\cap G(\BA)^1|F^{P_1}(m_1,T)=1\}$. 
  Let $\Gamma_1\subseteq N_{P_{1,H}}(\BA)$ be a compact subset which projects surjectively on $N_{P_{1,H}}(F)\bs N_{P_{1,H}}(\BA)$. Then $\Gamma_1\cdot (\varrho_{\wt{P}_0}\cap M_{P_{1,H}}(\BA))$ is a compact subset of $N_{\wt{P}_0}(\BA)$. Let $\Gamma_2\subseteq N_{P_{2,H}}(\BA)$, $\Gamma_3\subseteq N_{\wt{P}_0}^{P_{2,H}}(\BA)$ and $\Gamma_4\subseteq M_{\wt{P}_0}(\BA)^1$ be compact subsets (independent of $T$) such that $\Gamma_1\cdot (\varrho_{\wt{P}_0}\cap M_{P_{1,H}}(\BA))\subseteq \Gamma_2\Gamma_3\Gamma_4$. 
We have 
  \[\begin{split}
    &\int_{P_{1,H}(F)\bs H(\BA)\cap G(\BA)^1} \chi_{P_1, P_2}^T(x) \sum_{\xi\in(\wt{\fm}_{P_1}^R\cap\fs)(F)} \sum_{\wh{X}\in((\ov{\fn}_R^{P_2})'\cap\fs)(F)} |\Phi_\xi^{x,R}(\wh{X})| |\Nrd (x_1)|_\BA^s dx \\
    \leq &c_1 \sum_{B\in \CP(\wt{P}_0,P_1)} \int_{K_H} \int_{\Gamma_4} \int_{A_B^{G,\infty}(P_1,t_0,T_B)} \int_{\Gamma_3} \int_{\Gamma_2} \sigma_{P_1}^{P_2}(H_{P_1}(a)-T_{P_1}) \sum_{\xi\in(\wt{\fm}_{P_1}^R\cap\fs)(F)} \sum_{\wh{X}\in((\ov{\fn}_R^{P_2})'\cap\fs)(F)} \\
    &|\Phi_\xi^{n_2 n a m k,R}(\wh{X})| e^{-2\rho_{\wt{P}_0}(H_{B}(a))} |\Nrd (a_1)|_\BA^s dn_2 dn da dm dk,
  \end{split}\]
  where $c_1=\vol(K_H\cap M_{P_{1,H}}(\BA))$ is a constant independent of $T$. 

\begin{lem}\label{invofn_2pf1}
Let $x\in H(\BA), \xi\in(\fm_R\cap\fs)(\BA)$ and $\wh{X}\in(\ov{\fn}_R\cap\fs)(\BA)$. Let $R\subseteq P_2$ be a pair of relatively standard parabolic subgroups of $G$. For $n_2\in N_{P_{2,H}}(\BA)$, we have
  $$ \Phi_{\xi}^{n_2 x,R}(\wh{X})=\Phi_{\xi}^{x,R}(\wh{X}). $$
\end{lem}

\begin{proof}[Proof of Lemma \ref{invofn_2pf1}]
  Let $U_2:=n_2^{-1}\xi n_2-\xi.$ Then
  \[\begin{split}
    \Phi_{\xi}^{n_2 x,R}(\wh{X})&=\int_{(\fn_R\cap\fs)(\BA)} f(x^{-1}n_2^{-1}(\xi+U)n_2 x)\Psi(\langle U,\wh{X}\rangle) dU \\
    &=\int_{(\fn_R\cap\fs)(\BA)} f(x^{-1}(\xi+U_2+n_2^{-1}Un_2)x)\Psi(\langle U,\wh{X}\rangle) dU.
  \end{split}\]
  Since both $U_2$ and $n_2^{-1}Un_2-U$ belong to $(\fn_{P_2}\cap\fs)(\BA)$, we have
  $$ \langle U_2+n_2^{-1}Un_2-U,\wh{X}\rangle=0, $$
  so
  $$ \Phi_{\xi}^{n_2 x,R}(\wh{X})=\int_{(\fn_R\cap\fs)(\BA)} f(x^{-1}(\xi+U_2+n_2^{-1}Un_2)x)\Psi(\langle U_2+n_2^{-1}Un_2,\wh{X}\rangle) dU. $$
  Because the change of variables $U_2+n_2^{-1}Un_2\mapsto U$ does not change the Haar measure, we obtain
  $$ \Phi_{\xi}^{n_2 x,R}(\wh{X})=\Phi_{\xi}^{x,R}(\wh{X}). $$
\end{proof}

  Using Lemma \ref{invofn_2pf1}, we get
$$ \Phi_\xi^{n_2 n a m k,R}(\wh{X})=\Phi_\xi^{n a m k,R}(\wh{X})=\Phi_\xi^{a a^{-1} n a m k,R}(\wh{X}). $$
  By change of variables $a^{-1}Ua\mapsto U$, using the fact that
  $$ \langle U,\wh{X}\rangle=\langle a^{-1}Ua,a^{-1}\wh{X}a\rangle, $$
  we have
  $$ \Phi_\xi^{n_2 n a m k,R}(\wh{X})=e^{2\rho_{R,+}(H_{B}(a))} \Phi_{a^{-1}\xi a}^{a^{-1} n a m k,R}(a^{-1}\wh{X}a), $$
  where we denote by $\rho_{R,+}$ the half of the sum of weights (with multiplicities) for the action of $A_0$ on $\fn_R\cap \fs$.
  From the reduction theory (see \cite[p. 944]{MR518111}), we know that for $a$ satisfying $\sigma_{P_1}^{P_2}(H_{P_1}(a)-T_{P_1})\neq 0$, $a^{-1} n a$ belongs to a compact subset independent of $T$. In sum,
  \[\begin{split}
    &\int_{P_{1,H}(F)\bs H(\BA)\cap G(\BA)^1} \chi_{P_1, P_2}^T(x) \sum_{\xi\in(\wt{\fm}_{P_1}^R\cap\fs)(F)} \sum_{\wh{X}\in((\ov{\fn}_R^{P_2})'\cap\fs)(F)} |\Phi_\xi^{x,R}(\wh{X})||\Nrd (x_1)|_\BA^s dx \\
    \leq &c_2 \sum_{B\in \CP(\wt{P}_0,P_1)} \sup_{y\in \Gamma} \int_{A_B^{G,\infty}(P_1,t_0,T_B)} e^{(2\rho_{R,+}-2\rho_{\wt{P}_0})(H_B(a))} \sigma_{P_1}^{P_2}(H_{P_1}(a)-T_{P_1}) \\
    & \cdot \sum_{\xi\in(\wt{\fm}_{P_1}^R\cap\fs)(F)} \sum_{\wh{X}\in((\ov{\fn}_R^{P_2})'\cap\fs)(F)} |\Phi_{a^{-1}\xi a}^{y,R}(a^{-1}\wh{X}a)||\Nrd (a_1)|_\BA^s da,
  \end{split}\]
  where $c_2$ is a constant independent of $T$, and $\Gamma$ is a compact subset of $H(\BA)\cap G(\BA)^1$ independent of $T$.

  Denote by $\CO_F$ the ring of integers of $F$. Fix an $F$-basis for each weight space for the action of $A_0(F)$ on $\fs(F)$. Since the function $f\in\CS(\fs(\BA))$ is compactly supported at finite places, there exists an $\CO_F$-scheme structure on such weight spaces independent of $T$ such that the sums over $\xi\in(\wt{\fm}_{P_1}^R\cap\fs)(F)$ and $\wh{X}\in((\ov{\fn}_R^{P_2})'\cap\fs)(F)$ can be restricted to $\xi\in\wt{\fm}_{P_1}^R(F)\cap\fs(\CO_F)$ and $\wh{X}\in(\ov{\fn}_R^{P_2})'(F)\cap\fs(\CO_F)$ 
respectively. To see this, one may consult \cite[\S1.9 and p. 363]{MR1893921} for details, and one needs to replace $\fm_R$ and $\fn_R$ in {\it{loc. cit.}} by $\fm_R\cap\fs$ and $\fn_R\cap\fs$ respectively. 

  Fix an $\BR$-basis $\{e_1, \cdots, e_\ell\}$ of the $\BR$-linear space $\fs(F\otimes_{\BQ} \BR)$, whose dimension is denoted by $\ell$, consisting of eigenvectors for the action of $A_B^\infty$. Let $\|\cdot\|$ be the standard Euclidean norm with respect to this basis. Consider a sufficiently large integer $k>0$ to be described precisely at the end of the proof. There exists an even integer $m\geq 0$, a real number $k_\alpha\geq 0$ for each $\alpha\in\Delta_{B}^{P_2}$, and a real number $c_3>0$ satisfying the following conditions (cf. \cite[(4.10) in p. 372]{MR1893921} or \cite[p. 946-947]{MR518111}): 
\begin{enumerate}[\indent (1)]
\item if $R=P_2$, $m=0$; 
\end{enumerate}
\begin{enumerate}[\indent (2)]
\item for all $\alpha\in\Delta_{B}^{P_2}-\Delta_{B}^{R}, k_\alpha\geq k$; 
\end{enumerate}
\begin{enumerate}[\indent (3)]
\item for all $a_0\in A_B^\infty(P_2, t_0)$, 
\begin{equation}\label{equchau4.10}
 \sum_{\wh{X}\in (\ov{\fn}_R^{P_2})'(F)\cap\fs(\CO_F)} \|a_0^{-1}\wh{X}a_0\|^{-m}\leq c_3 \prod_{\alpha\in\Delta_B^{P_2}} e^{-k_\alpha \alpha(H_B(a_0))}. 
\end{equation}
\end{enumerate}
We fix such data. 

We extend any differential operator $\partial$ on $\fs(F\otimes_\BQ \BR)$ to $\fs(\BA)$ by defining $\partial(f_\infty\otimes\chi^\infty):=(\partial f_\infty)\otimes\chi^\infty$ (see Section \ref{BSandHaar1} for the notation). We also write 
  $$ \Phi_\xi^{x,R,\partial}(\wh{X}):=\int_{(\fn_R\cap\fs)(\BA)} (\partial f)(x^{-1}(\xi+U)x) \Psi(\langle U,\wh{X}\rangle) dU. $$
For a multi-index $\overrightarrow{i}=(i_1,\cdots, i_\ell)\in\BZ_{\geq0}^\ell$, denote by $\partial^{\overrightarrow{i}}:=(\frac{\partial}{\partial e_1})^{i_1}\cdots(\frac{\partial}{\partial e_\ell})^{i_\ell}$ the corresponding differential operator on $\fs(F\otimes_\BQ \BR)$. 	
Since $m$ is even, the function $\wh{X}\mapsto\|\wh{X}\|^m$ is a polynomial with $\BZ$-coefficients on $(\ov{\fn}_R^{P_2}\cap\fs)(F\otimes_\BQ \BR)$. 
  Invoking integration by parts (see \cite[Theorem 3.3.1.(f)]{MR2121678} for example), we see that there exists a differential operator $\partial^{(m)}$ on $\fs(F\otimes_\BQ \BR)$ satisfying the following two conditions. 
  \begin{enumerate}[\indent (1)]
\item $\partial^{(m)}$ is a finite $\BZ$-linear combination of $\partial^{\overrightarrow{i}}$'s with the properties: 
\begin{itemize}
	\item the sum of components of $\overrightarrow{i}$ is $m$; 
	\item all components of $\overrightarrow{i}$ are even integers;  
	\item all non-zero components of $\overrightarrow{i}$ correspond to eigenvectors lying in $(\fn_R^{P_2}\cap\fs)(F\otimes_\BQ\BR)$. 
\end{itemize}
\end{enumerate}
  \begin{enumerate}[\indent (2)]
\item    For $\wh{X}\neq 0$, we have 
  $$ |\Phi_{\xi}^{y,R}(\wh{X})|=\|\wh{X}\|^{-m} |\Phi_{\xi}^{y,R,\Ad(y^{-1})\partial^{(m)}}(\wh{X})|. $$
\end{enumerate}
We fix such a $\partial^{(m)}$. Suppose that $\partial^{(m)}=\sum\limits_{\overrightarrow{i}\in\CI}r_{\overrightarrow{i}} \partial^{\overrightarrow{i}}$, where $\CI$ is a finite set of multi-indices and  $r_{\overrightarrow{i}}\in\BZ$. Then for $\wh{X}\neq 0$, we have 
	$$ |\Phi_{a^{-1}\xi a}^{y,R}(a^{-1}\wh{X}a)|\leq c_4(y) \|a^{-1}\wh{X}a\|^{-m} \sum_{\overrightarrow{i}\in\CI} |\Phi_{a^{-1}\xi a}^{y,R,\partial^{\overrightarrow{i}}}(a^{-1}\wh{X}a)|, $$
 where $c_4(y)$ is a continuous function of $y$. 

  Denote by $\Phi(A_{B}, \fm_R\cap\fs)$ the set of weights of $A_{B}$ in $\fm_R\cap\fs$. For any $\mu\in\Phi(A_{B}, \fm_R\cap\fs)$, let $\fm_\mu$ be the corresponding weight space. From \cite[\S41]{MR0165033}, we know that there exists a function $\phi_\mu\in\CS(\fm_\mu(\BA))$ for each $\mu\in\Phi(A_{B}, \fm_R\cap\fs)$ and a function $\phi_{\fn_R\cap\fs}\in\CS((\fn_R\cap\fs)(\BA))$ such that for all $\xi+U\in(\fm_R\cap\fs)(\BA)\oplus(\fn_R\cap\fs)(\BA)$ and $y\in\Gamma$,
  $$ \sum_{\overrightarrow{i}\in\CI} |(\partial^{\overrightarrow{i}} f)(y^{-1}(\xi+U)y)|\leq \left(\prod_{\mu\in\Phi(A_{B}, \fm_R\cap\fs)}\phi_\mu(\xi_\mu)\right) \phi_{\fn_R\cap\fs}(U), $$
where $\xi_\mu$ denotes the projection of $\xi$ to $\fm_\mu(\BA)$. 

Let $a\in A_B^{G,\infty}(P_1,t_0,T_B)$ be such that $\sigma_{P_1}^{P_2}(H_{P_1}(a)-T_{P_1})\neq 0$. It is shown in \cite[p. 375]{MR1893921} that $\alpha(H_B(a))>t_0$ for all $\alpha\in\Delta_B^{P_2}$. 
  Now we have 
  \[\begin{split}
    &\sum_{\xi\in(\wt{\fm}_{P_1}^R\cap\fs)(F)} \sum_{\wh{X}\in((\ov{\fn}_R^{P_2})'\cap\fs)(F)} |\Phi_{a^{-1}\xi a}^{y,R}(a^{-1}\wh{X}a)| \\
    = &\sum_{\xi\in\wt{\fm}_{P_1}^R(F)\cap\fs(\CO_F)} \sum_{\wh{X}\in(\ov{\fn}_R^{P_2})'(F)\cap\fs(\CO_F)} |\Phi_{a^{-1}\xi a}^{y,R}(a^{-1}\wh{X}a)| \\
    \leq &\sum_{\xi\in\wt{\fm}_{P_1}^R(F)\cap\fs(\CO_F)} \sum_{\wh{X}\in(\ov{\fn}_R^{P_2})'(F)\cap\fs(\CO_F)} c_4(y) \|a^{-1}\wh{X}a\|^{-m} \sum_{\overrightarrow{i}\in\CI} |\Phi_{a^{-1}\xi a}^{y,R,\overrightarrow{i}}(a^{-1}\wh{X}a)| \\
    \leq &c_5 \sum_{\xi\in\wt{\fm}_{P_1}^R(F)\cap\fs(\CO_F)}\left(\prod_{\mu\in\Phi(A_{B}, \fm_R\cap\fs)}\phi_\mu(\mu(a)^{-1}\xi_\mu)\right) \cdot \sum_{\wh{X}\in(\ov{\fn}_R^{P_2})'(F)\cap\fs(\CO_F)} \|a^{-1}\wh{X}a\|^{-m} \\
    \leq &c_5 c_3 \sum_{\xi\in\wt{\fm}_{P_1}^R(F)\cap\fs(\CO_F)}\left(\prod_{\mu\in\Phi(A_{B}, \fm_R\cap\fs)}\phi_\mu(\mu(a)^{-1}\xi_\mu)\right) \cdot \prod_{\alpha\in\Delta_B^{P_2}}e^{-k_\alpha \alpha(H_B(a))},
  \end{split}\]
  where $c_5:=\sup\limits_{y\in\Gamma}c_4(y) \int_{(\fn_R\cap\fs)(\BA)}\phi_{\fn_R\cap\fs}(U)dU$, and we have applied (\ref{equchau4.10}) to $a_0=a$ in the last inequality. Thus
  \[\begin{split}
    &\int_{P_{1,H}(F)\bs H(\BA)\cap G(\BA)^1} \chi_{P_1, P_2}^T(x) \sum_{\xi\in(\wt{\fm}_{P_1}^R\cap\fs)(F)} \sum_{\wh{X}\in((\ov{\fn}_R^{P_2})'\cap\fs)(F)} |\Phi_\xi^{x,R}(\wh{X})||\Nrd (x_1)|_\BA^s dx \\
    \leq &c_2 c_5 c_3 \sum_{B\in \CP(\wt{P}_0,P_1)} \int_{A_B^{G,\infty}(P_1,t_0,T_B)} e^{(2\rho_{R,+}-2\rho_{\wt{P}_0})(H_B(a))} \sigma_{P_1}^{P_2}(H_{P_1}(a)-T_{P_1}) \\
    & \cdot \sum_{\xi\in\wt{\fm}_{P_1}^R(F)\cap\fs(\CO_F)}\left(\prod_{\mu\in\Phi(A_{B}, \fm_R\cap\fs)}\phi_\mu(\mu(a)^{-1}\xi_\mu)\right) \cdot \left(\prod_{\alpha\in\Delta_B^{P_2}}e^{-k_\alpha \alpha(H_B(a))}\right) |\Nrd (a_1)|_\BA^s da.
  \end{split}\]

  Denote by $\Sigma_B^{\fm_R\cap\fs}$ the positive weights of $\fm_R\cap\fs$ under the action of $A_B$. Consider the subsets $S$ of $\Sigma_B^{\fm_R\cap\fs}$ with the following property: for all $\alpha\in\Delta_B^R-\Delta_B^{P_1}$, there exists $\mu\in S$ such that its $\alpha$-coordinate is $>0$. Then
  \[\begin{split}
    &\sum_{\xi\in\wt{\fm}_{P_1}^R(F)\cap\fs(\CO_F)}\left(\prod_{\mu\in\Phi(A_{B}, \fm_R\cap\fs)}\phi_\mu(\mu(a)^{-1}\xi_\mu)\right) \\
    \leq &\sum_{S}\left[\prod_{\mu\in S} \left(\sum_{\xi_{-}\in\fm_{-\mu}(\CO_F)-\{0\}} \phi_{-\mu}(\mu(a)\xi_{-})\right)\right] \left[\prod_{\mu\in\Sigma_B^{\fm_R\cap\fs}} \left(\sum_{\xi_{+}\in\fm_{\mu}(\CO_F)} \phi_{\mu}(\mu(a^{-1})\xi_{+})\right)\right]. \\ 
  \end{split}\]
  The rest of the proof is analogous to that of \cite[Proposition 4.4]{MR1893921}, and we shall only sketch main steps and point out additional ingredients. 
By the argument in \cite[p. 373]{MR1893921}, for $a\in A_{B}^\infty(P_2, t_0)$, we have bounds for two factors 
$$ \prod_{\mu\in S} \left(\sum_{\xi_{-}\in\fm_{-\mu}(\CO_F)-\{0\}} \phi_{-\mu}(\mu(a)\xi_{-})\right) \leq c_6 \prod_{\alpha\in\Delta_B^R-\Delta_B^{P_1}}e^{-k\alpha(H_B(a))} $$
and 
$$ \prod_{\mu\in\Sigma_B^{\fm_R\cap\fs}} \left(\sum_{\xi_{+}\in\fm_{\mu}(\CO_F)} \phi_{\mu}(\mu(a^{-1})\xi_{+})\right) \leq c_7 e^{(2\rho_{B,+}-2\rho_{R,+})(H_B(a))}, $$
where $c_6$ and $c_7$ are constants independent of $T$. 
We deduce that 
  \[\begin{split}
	&e^{(2\rho_{R,+}-2\rho_{\wt{P}_0})(H_B(a))} \cdot \sum_{\xi\in\wt{\fm}_{P_1}^R(F)\cap\fs(\CO_F)}\left(\prod_{\mu\in\Phi(A_{B}, \fm_R\cap\fs)}\phi_\mu(\mu(a)^{-1}\xi_\mu)\right) \cdot \left(\prod_{\alpha\in\Delta_B^{P_2}}e^{-k_\alpha \alpha(H_B(a))}\right)|\Nrd (a_1)|_\BA^s \\
    \leq&c_8 e^{(2\rho_{R,+}-2\rho_{\wt{P}_0})(H_B(a))}  \cdot \left(\prod_{\alpha\in\Delta_B^R-\Delta_B^{P_1}}e^{-k\alpha(H_B(a))}\right) e^{(2\rho_{B,+}-2\rho_{R,+})(H_B(a))} \cdot \left(\prod_{\alpha\in\Delta_B^{P_2}}e^{-k_\alpha \alpha(H_B(a))}\right)|\Nrd (a_1)|_\BA^s \\ 
    =&c_8 e^{(2\rho_{B,+}-2\rho_{\wt{P}_0})(H_B(a))}|\Nrd (a_1)|_\BA^s \cdot \left(\prod_{\alpha\in\Delta_B^R-\Delta_B^{P_1}}e^{-k\alpha(H_B(a))}\right) \cdot \left(\prod_{\alpha\in\Delta_B^{P_2}}e^{-k_\alpha \alpha(H_B(a))}\right), \\ 
  \end{split}\]
  where $c_8$ is a constant independent of $T$. 
The last expression only differs from \cite[(4.17) in p. 375]{MR1893921} by a factor 
  $$ e^{(2\rho_{B,+}-2\rho_{\wt{P}_0})(H_B(a))}|\Nrd (a_1)|_\BA^s=e^{(2\rho_{B,+}-2\rho_{\wt{P}_0})(H_B(a))}\left(\frac{|\Nrd (a_1)|_{\BA}^{1/p}}{|\Nrd (a_2)|_{\BA}^{1/q}}\right)^{\frac{pq}{p+q}s} $$
in the form of $e^{\lambda(H_B(a))}$, where $\lambda\in(\fa_B^G)^\ast$. We shall see that this discrepancy will be unimportant when we follow the end of the proof of \cite[Proposition 4.4]{MR1893921} (cf. \cite[end of \S3.2]{MR3835522}). Write 
$$ \lambda=\left(\sum_{\alpha\in\Delta_B^{P_2}} c_\alpha \alpha\right)+\lambda', $$ 
where $c_\alpha\in\BR$ and $\lambda'\in(\fa_{P_2}^G)^\ast$. 
By \cite[Lemma 8.3.(b)]{MR2192011}, there exists a constant $c_9>0$ such that for all $T_1\in\fa_{P_1}^{G}$ satisfying $\sigma_{P_1}^{P_2}(T_1)=1$, we have 
$$ \lambda'(T_2)\leq c_9\left(\sum\limits_{\beta\in\Delta_{P_1}^{P_2}} \beta(T_1^2)\right), $$
where we write $T_1=T_1^2+T_2$ with $T_1^2\in\fa_{P_1}^{P_2}$ and $T_2\in\fa_{P_2}^G$. 
Hence, when $a\in A_B^{G,\infty}(P_1,t_0,T_B)$ satisfies $\sigma_{P_1}^{P_2}(H_{P_1}(a)-T_{P_1})\neq 0$, we have 
\[\begin{split}
	e^{\lambda(H_B(a))}&=\left(\prod_{\alpha\in\Delta_B^{P_2}} e^{c_\alpha \alpha(H_B(a))} \right) \cdot e^{\lambda'(H_{P_2}(a))} \\
	&=\left(\prod_{\alpha\in\Delta_B^{P_2}} e^{c_\alpha \alpha(H_B(a))} \right) \cdot e^{\lambda'(H_{P_2}(a)-T_{P_2}^G)} e^{\lambda'(T_{P_2}^G)} \\
	&\leq \left(\prod_{\alpha\in\Delta_B^{P_2}} e^{c_\alpha \alpha(H_B(a))} \right) \cdot \left(\prod_{\beta\in\Delta_{P_1}^{P_2}} e^{c_9t_\beta}\right) e^{c_{10}\|T\|}, 
\end{split}\]
where $T_{P_2}^G$ is the projection of $T_{P_2}$ to $\fa_{P_2}^G$ via the decomposition $\fa_{P_2}=\fa_{P_2}^G\oplus\fa_G$, we let $\sum\limits_{\beta\in\Delta_{P_1}^{P_2}}t_\beta\varpi_\beta^\vee$ with $t_\beta\in\BR$ be the projection of $H_{P_1}(a)-T_{P_1}$ to $\fa_{P_1}^{P_2}$ via the decomposition $\fa_{P_1}=\fa_{P_1}^{P_2}\oplus\fa_{P_2}$, and $c_{10}$ is a constant independent of $T$. To bound \eqref{equation1.1}, it suffices to plug this extra factor into \cite[(4.18) in p. 375]{MR1893921}. More precisely, set  
\begin{displaymath}
k'_\alpha := \left\{ \begin{array}{ll}
k_\alpha+c_\alpha, & \textrm{for all $\alpha\in\Delta_B^{P_2}-\Delta_B^R$; }\\
k_\alpha+k+c_\alpha, & \textrm{for all $\alpha\in\Delta_B^{R}-\Delta_B^{P_1}$. }\\
\end{array} \right. 
\end{displaymath}
Then there exists a constant $c_{11}>0$ independent of $T$ such that 
\[\begin{split}
	&\int_{P_{1,H}(F)\bs H(\BA)\cap G(\BA)^1} \chi_{P_1, P_2}^T(x) \sum_{\xi\in(\wt{\fm}_{P_1}^R\cap\fs)(F)} \sum_{\wh{X}\in((\ov{\fn}_R^{P_2})'\cap\fs)(F)} |\Phi_\xi^{x,R}(\wh{X})| |\Nrd (x_1)|_\BA^s dx \\
	\leq&c_{11}\vol(A_B^{P_1,\infty}(t_0,T_B)) e^{c_{10}\|T\|} \prod_{\alpha\in\Delta_B^{P_2}-\Delta_B^{P_1}} \left(e^{-k'_\alpha \alpha(T_B)} \int_0^\infty (1+t)^{n_\alpha} e^{(c_9-k'_\alpha) t} dt\right), 
\end{split}\]
where $A_B^{P_1,\infty}(t_0,T_B):=A_B^\infty(P_1, t_0, T_B)\cap M_{P_1}(\BA)^1$, and $n_\alpha$'s are positive integers independent of $T$. We know that $\vol(A_B^{P_1,\infty}(t_0,T_B))$ is of polynomial growth in $T$ and that $k_\alpha\geq k+c_\alpha$ for all $\alpha\in\Delta_B^{P_2}-\Delta_B^{P_1}\neq\emptyset$. For sufficiently regular $T$ satisfying $\alpha(T_B)\geq\epsilon_0\|T\|$ for all $\alpha\in\Delta_B^G$, if we choose sufficiently large $k$, then the last expression is bounded by $Ce^{-N\|T\|}$ with a constant $C$ independent of $T$. 
\end{proof}


\section{\textbf{Exponential polynomial distributions}}\label{dist1}

Let $T$ be sufficiently regular, $\fo\in\CO$ and $\eta$ be the quadratic character of $\BA^\times/F^\times$ attached to a quadratic field extension $E/F$. For $f\in\CS(\fs(\BA))$ and $s\in \BC$, define
\begin{equation}\label{Jo1}
 J_\fo^{G,T}(\eta, s, f):=\int_{H(F)\bs H(\BA)\cap G(\BA)^1} k_{f,\fo}^T(x) \eta(\Nrd (x)) |\Nrd (x_1)|_\BA^s dx 
\end{equation}
and
$$ J^{G,T}(\eta, s, f):=\int_{H(F)\bs H(\BA)\cap G(\BA)^1} k_{f}^T(x) \eta(\Nrd (x)) |\Nrd (x_1)|_\BA^s dx, $$
where $k_{f,\fo}^T(x)$ and $k_{f}^T(x)$ are defined by (\ref{deftruncation1}) and (\ref{deftruncation2}) respectively, and we write $x=(x_1, x_2)\in GL_{p,D}(\BA)\times GL_{q,D}(\BA)$. From Theorem \ref{convergence1}, we know that $J_\fo^{G,T}(\eta,s,\cdot)$ and $J^{G,T}(\eta,s,\cdot)$ are well-defined distributions on $\CS(\fs(\BA))$ and that
$$ J^{G,T}(\eta,s,f)=\sum_{\fo\in\CO}J_\fo^{G,T}(\eta,s,f), $$
which is an analogue of the geometric side of Arthur's trace formula.

\subsection{A generalised case in the product form}\label{pdtform1}

Let $Q$ be a relatively standard parabolic subgroup of $G$. Then
$$ M_Q\simeq GL_{p_1+q_1,D}\times\cdot\cdot\cdot\times GL_{p_l+q_l,D} $$
and
$$ M_{Q_H}\simeq GL_{p_1,D}\times\cdot\cdot\cdot\times GL_{p_l,D} \times GL_{q_1,D}\times\cdot\cdot\cdot\times GL_{q_l,D}, $$
where $\sum\limits_{i=1}^{l}p_i=p, \sum\limits_{i=1}^{l}q_i=q$ and we allow $p_i$ or $q_i$ to be zero. The tangent space of $M_Q/M_{Q_H}$ at the neutral element is
$$ \fm_Q\cap\fs\simeq \bigoplus\limits_{\{1\leq i\leq l|p_iq_i\neq0\}} \mat(0,\Mat_{{p_i}\times{q_i}, D},\Mat_{{q_i}\times{p_i}, D},0). $$
The conjugate action of $M_{Q_H}(F)$ on $(\fm_Q\cap\fs)(F)$ can be described as follows: 
\begin{enumerate}[\indent (1)]
\item if $p_iq_i\neq0$, $\mat(GL_{p_i}(D),,,{GL_{q_i}(D)})$ acts on $\mat(0,\Mat_{{p_i}\times{q_i}}(D),\Mat_{{q_i}\times{p_i}}(D),0)$ by conjugation; 
\end{enumerate}
\begin{enumerate}[\indent (2)]
\item if $p_iq_i=0$, $\mat(GL_{p_i}(D),,,{GL_{q_i}(D)})$ acts on $0$ (viewed as a 0-dimensional vector space) trivially. 
\end{enumerate}
We may generalise integrability in last section to the product setting here whose proof is similar. 

Define a relation of equivalence on $(\fm_Q\cap\fs)(F)$ which is similar to that on $\fs(F)$ on each component. We denote by $\CO^{\fm_Q\cap\fs}$ the set of equivalent classes for this relation. For $\fo\in\CO$, the intersection $\fo\cap\fm_Q(F)$ is a finite (perhaps empty) union of classes $\fo_1,\cdot\cdot\cdot,\fo_t\in\CO^{\fm_Q\cap\fs}$. Fix the minimal parabolic subgroup $\wt{P}'_0:=\wt{P}_0\cap M_{Q_H}=\wt{P}_0\cap M_Q$ of $M_{Q_H}$ and its Levi factor $M_0$. We say that a parabolic subgroup $P'$ of $M_Q$ is semi-standard (resp. relatively standard) if $M_0\subseteq P'$ (resp. $\wt{P}'_0\subseteq P'$). Notice that there exists a bijection from the set of semi-standard (resp. relatively standard) parabolic subgroups of $G$ contained in $Q$ to the set of semi-standard (resp. relatively standard) parabolic subgroups of $M_Q$ given by $P\mapsto P\cap M_Q$, whose inverse is given by $P'\mapsto P'N_Q$. 

Choose $\varsigma_Q\in\Omega^G$ (not unique) such that $\varsigma_Q P_0\subseteq Q$. Fix the minimal semi-standard parabolic subgroup $P'_0:=(\varsigma_Q P_0)\cap M_Q$ of $M_Q$ depending on the choice of $\varsigma_Q$. For any semi-standard parabolic subgroup $P'$ of $M_Q$ and $T\in\fa_0$, denote by $T_{P'}$ the projection of $sT$ in $\fa_{P'}$, where $s\in\Omega^{M_Q}$ such that $s P'_0\subseteq P'$. For $s\in\Omega^{M_Q}$ and a semi-standard parabolic subgroup $P\subseteq Q$ of $G$, we see that $sP'_0\subseteq P\cap M_Q$ if and only if $s\varsigma_Q P_0\subseteq P$. Then $(\varsigma_Q T)_{P\cap M_Q}=T_P$ which is independent of the choice of $\varsigma_Q$. 
If $T\in\fa_{P_0}^+$ is sufficiently regular with respect to $P_0\subseteq G$, then $\varsigma_Q T\in\fa_{P'_0}^+$ is sufficiently regular with respect to $P'_0\subseteq M_Q$. 

Let $f'\in\CS((\fm_Q\cap\fs)(\BA))$, $P'$ be a relatively standard parabolic subgroup of $M_Q$ and $1\leq j\leq t$. Write $P'_H:=P'\cap M_{Q_H}=P'\cap H$. For $x\in M_{P'_H}(F) N_{P'_H}(\BA)\bs M_{Q_H}(\BA)$, define
\begin{equation}\label{LevikfPo1}
 k^{M_Q}_{f',P',\fo_j}(x):=\sum_{X\in\fm_{P'}(F)\cap\fo_j} \int_{(\fn_{P'}\cap\fs)(\BA)} f'(x^{-1}(X+U)x) dU. 
\end{equation}
For $T\in\fa_0$ and $x\in M_{Q_H}(F)\bs M_{Q_H}(\BA)$, define
$$ k_{f',\fo_j}^{Q,T}(x):=\sum_{\{P':\wt{P}'_0\subseteq P'\}} (-1)^{\dim(A_{P'}/A_{M_Q})} \sum_{\delta\in P'_H(F)\bs M_{Q_H}(F)} \wh{\tau}_{P'}^{M_Q}(H_{P'}(\delta x)-T_{P'})\cdot k^{M_Q}_{f',P',\fo_j}(\delta x). $$
For sufficiently regular $T\in\fa_{P_0}^+$ and $\{s_i\}_{1\leq i\leq l}\in\BC^l$, define
$$ J_{\fo_j}^{Q,T}(\eta,\{s_i\},f'):=\int_{M_{Q_H}(F)\bs M_{Q_H}(\BA)\cap M_Q(\BA)^1} k_{f',\fo_j}^{Q,\varsigma_Q T}(x)\eta(\Nrd(x))\prod_{1\leq i\leq l}|\Nrd (x_{i,1})|_{\BA}^{s_i} dx, $$
where we write $x=(x_1,...,x_l)\in GL_{p_1+q_1,D}(\BA)\times\cdot\cdot\cdot\times GL_{p_l+q_l,D}(\BA)$ and $x_i=(x_{i,1}, x_{i,2})\in GL_{p_i,D}(\BA)\times GL_{q_i,D}(\BA)$. As explained above, $k_{f',\fo_j}^{Q,\varsigma_Q T}$ and $J_{\fo_j}^{Q,T}$ are independent of the choice of $\varsigma_Q$. Then we have well-defined distributions $J_{\fo_j}^{Q,T}(\eta,\{s_i\},\cdot)$ on $\CS((\fm_Q\cap\fs)(\BA))$. It only depends on the projection of $\varsigma_Q T$ to $\fa_{\varsigma_Q P_0}^{Q}$ and does not depend on $T_Q$. Now we define
\begin{equation}\label{LeviJo1}
 J_\fo^{Q,T}:=\sum_{j=1}^{t} J_{\fo_j}^{Q,T} 
\end{equation}
and
$$ J^{Q,T}:=\sum_{\fo\in\CO} J_{\fo}^{Q,T}. $$

For $f\in\CS(\fs(\BA))$, define $f_Q^\eta\in\CS((\fm_Q\cap\fs)(\BA))$ by
\begin{equation}\label{equation1.2}
 f_Q^\eta(X):=\int_{K_H} \int_{(\fn_Q\cap\fs)(\BA)} f(k^{-1}(X+V)k) \eta(\Nrd (k)) dVdk
\end{equation}
for all $X\in(\fm_Q\cap\fs)(\BA)$. 

\subsection{$\omega$-stable parabolic subgroups}\label{defomgstb}

In our case, we can embed $G$ into $\fg$ in the standard way. For any linear subspace $\fv$ of $\fg$, we denote by $\fv^\times$ the intersection of $\fv$ and $G$ in $\fg$. Assume that $p=q$. Let us denote $n:=p=q$. Then $\fs^\times(F)$ is the union of classes in $\CO^\times$. Let $\omega:=\mat(0,1_n,1_n,0)\in G(F)$. In this section, we shall freely use the notation in Section \ref{relstdpar1}. Then $\omega$ is the element in $G$ exchanging $e_i$ and $f_i$ for all $1\leq i\leq n$. We see that $\omega\wt{P}_0\omega^{-1}=\wt{P}_0$. We say that a semi-standard parabolic subgroup $Q$ of $G$ is ``$\omega$-stable'' if $\omega Q\omega^{-1}=Q$. By Chevalley's theorem, this condition is equivalent to $\omega\in Q$. 
Recall that a relatively standard parabolic subgroup $Q$ of $G$ is understood as the stabiliser in $G$ of the flag
  $$ 0\subsetneq\langle e_1,\cdot\cdot\cdot,e_{p_1},f_1,\cdot\cdot\cdot,f_{q_1}\rangle_D
    \subsetneq\langle e_1,\cdot\cdot\cdot,e_{p_1+p_2},f_1,\cdot\cdot\cdot,f_{q_1+q_2}\rangle_D
    \subsetneq...\subsetneq\langle e_1,\cdot\cdot\cdot,e_{p_1+\cdot\cdot\cdot+p_{l}},f_1,\cdot\cdot\cdot,f_{q_1+\cdot\cdot\cdot+q_{l}}\rangle_D, $$
    where $\sum\limits_{i=1}^l p_i=\sum\limits_{i=1}^l q_i=n$ and we allow $p_i$ or $q_i$ to be zero. 
    
\begin{prop}\label{omegapq}
	Assume that $p=q=n$. Let $Q$ be a relatively standard parabolic subgroup of $G$. Then $Q$ is $\omega$-stable if and only if $p_i=q_i$ for all $1\leq i\leq l$. 
\end{prop}

\begin{proof}
	Since $\omega\in G$ exchanges $e_i$ and $f_i$ for all $1\leq i\leq n$, the parabolic subgroup $\omega Q\omega^{-1}$ of $G$ is the stabiliser in $G$ of the flag
  $$ 0\subsetneq\langle f_1,\cdot\cdot\cdot,f_{p_1},e_1,\cdot\cdot\cdot,e_{q_1}\rangle_D
    \subsetneq\langle f_1,\cdot\cdot\cdot,f_{p_1+p_2},e_1,\cdot\cdot\cdot,e_{q_1+q_2}\rangle_D
    \subsetneq...\subsetneq\langle f_1,\cdot\cdot\cdot,f_{p_1+\cdot\cdot\cdot+p_{l}},e_1,\cdot\cdot\cdot,e_{q_1+\cdot\cdot\cdot+q_{l}}\rangle_D. $$
    Then $\omega Q\omega^{-1}=Q$ if and only if the flags associated to $\omega Q\omega^{-1}$ and $Q$ are the same, which is equivalent to $p_i=q_i$ for all $1\leq i\leq l$. 
\end{proof}

An illustrating example of $\omega$-stable relatively standard parabolic subgroups of $G$ looks like (when $l=2$)
$$ Q=
\left( \begin{array}{cccc}
* & * & * & * \\
0 & * & 0 & * \\
* & * & * & * \\
0 & * & 0 & * \\
\end{array} \right)^\times. $$

\begin{prop}
	Assume that $p=q=n$. The map $P_n\mapsto\mat(\fp_{n},\fp_{n},\fp_{n},\fp_{n})^\times$ induces a bijection from the set of standard parabolic subgroups in $GL_{n,D}$ (namely containing the group of upper triangular matrices) to the set of $\omega$-stable relatively standard parabolic subgroups in $G$. 
\end{prop}

\begin{proof}
	It is known that there is a bijection between the set of standard parabolic subgroups in $GL_{n,D}$ and the set of partitions of $n$. By Proposition \ref{omegapq}, the latter set is in bijection with the set of $\omega$-stable relatively standard parabolic subgroups in $G$. The composition of these bijection is exactly given by the map $P_n\mapsto\mat(\fp_{n},\fp_{n},\fp_{n},\fp_{n})^\times$. 
\end{proof}

\begin{prop}
	Assume that $p=q=n$. Let $Q\subseteq R$ be a pair of relatively standard parabolic subgroups of $G$. If $Q$ is $\omega$-stable, then $R$ is $\omega$-stable. 
\end{prop}

\begin{proof}
	Suppose that $R$ is the stabiliser in $G$ of the flag
  $$ 0\subsetneq\langle e_1,\cdot\cdot\cdot,e_{r_1},f_1,\cdot\cdot\cdot,f_{t_1}\rangle_D
    \subsetneq\langle e_1,\cdot\cdot\cdot,e_{r_1+r_2},f_1,\cdot\cdot\cdot,f_{t_1+t_2}\rangle_D
    \subsetneq...\subsetneq\langle e_1,\cdot\cdot\cdot,e_{r_1+\cdot\cdot\cdot+r_{l'}},f_1,\cdot\cdot\cdot,f_{t_1+\cdot\cdot\cdot+t_{l'}}\rangle_D. $$
  The condition $Q\subseteq R$ tells us that the partition $(p_1,\cdot\cdot\cdot,p_l)$ (resp. $(q_1,\cdot\cdot\cdot,q_l)$) is a refinement of the partition $(r_1,\cdot\cdot\cdot,r_{l'})$ (resp. $(t_1,\cdot\cdot\cdot,t_{l'})$) of $n$, and that for all $1\leq j\leq l'$, $r_j$ and $t_j$ are divided into the same number of segments in these two refinements. Hence, if $p_i=q_i$ for all $1\leq i\leq l$, then $r_j=t_j$ for all $1\leq j\leq l'$. Thus this proposition results from Proposition \ref{omegapq}. 
\end{proof}

For any relative standard parabolic subgroup $Q$ of $G$, define
$$ \ov{Q}^{\omega\text{-st}}:=\bigcap_{\{R:Q\subseteq R, \omega R\omega^{-1}=R\}} R, $$
which is the minimal $\omega$-stable parabolic subgroup of $G$ containing $Q$. 

\begin{prop}\label{propOinvert1}
Assume that $p=q=n$. Let $\fo\in\CO$. The following three conditions are equivalent: 
\begin{enumerate}[\indent (1)]
\item $\fo\in\CO^\times$; 
\end{enumerate}
\begin{enumerate}[\indent (2)]
\item for all relatively standard parabolic subgroup $Q$ of $G$, if $\fo\cap\fq(F)\neq\emptyset$, then $Q$ is $\omega$-stable; 
\end{enumerate}
\begin{enumerate}[\indent (3)]
\item for all relatively standard parabolic subgroup $Q$ of $G$, if $\fo\cap\fm_Q(F)\neq\emptyset$, then $Q$ is $\omega$-stable. 
\end{enumerate}
\end{prop}

\begin{proof}
The direction (2)$\Rightarrow$(3) is trivial. We actually have (2)$\Leftrightarrow$(3) from Proposition \ref{propofpipar1}. 

Next, we prove the direction (1)$\Rightarrow$(2). We assume that $\fo\in\CO^\times$ and that $\fo\cap\fq(F)\neq\emptyset$ for some relatively standard parabolic subgroup $Q$ of $G$. If $Q$ is not $\omega$-stable, let $k$ be the minimal integer such that $1\leq k\leq l-1$ and that 
$$ \sum_{1\leq i\leq k}p_i-\sum_{1\leq i\leq k}q_i\neq 0. $$
Without loss of generality, we may assume that 
$$ \sum_{1\leq i\leq k}p_i-\sum_{1\leq i\leq k}q_i< 0. $$
Let $\mat(0,A,B,0)\in\fo\cap\fq(F)$. Then $A\in \fg\fl_n(D)$ is in the form of $\mat(\ast,\ast,0,\ast)$, where the size of the zero matrix in the lower left corner is at least $\bigg(\sum\limits_{k+1\leq i\leq l} p_i\bigg)\times \bigg(1+\sum\limits_{k+1\leq i\leq l} p_i\bigg)$. Therefore, $A$ is not invertible, which contradicts with $\fo\in\CO^\times$. This establishes (1)$\Rightarrow$(2). 

Finally, we prove the direction (3)$\Rightarrow$(1). We assume (3). Suppose that $\fo\notin\CO^\times$. Let $P(\lambda):=\Prd_{AB}(\lambda)$, where $\mat(0,A,B,0)$ is any element in $\fo$. By \cite[Proposition 5]{yu2013}, there exists $1\leq m\leq n$ such that $P(\lambda)=\lambda^{dm} R(\lambda)$, where $R(\lambda)=\Prd_C(\lambda)$ for some $C\in GL_{n-m}(D)$. 
Let $Q$ be the relative standard parabolic subgroup of $G$ with $l=2$, $p_1=q_1=n-m$ and $p_2=q_2=m$. 
Then $\left( \begin{array}{cccc}
0 & 0 & 1_{n-m} & 0 \\
0 & 0 & 0 & 0 \\
C & 0 & 0 & 0 \\
0 & 0 & 0 & 0 \\
\end{array} \right)\in\fo\cap\fm_Q(F)$, which contradicts with (3). This shows (3)$\Rightarrow$(1). 
\end{proof}

Denote by $\rho_{Q,+}$ the half of the sum of weights (with multiplicities) for the action of $A_0$ on $\fn_Q\cap\fs$. We see that $\rho_{Q,+}=\rho_Q-\rho_{Q_H}$ and that for $Q\subseteq R$ a pair of relatively standard parabolic subgroup of $G$, the restriction of $(2\rho_{Q,+}-2\rho_{Q_H})\big|_{\fa_Q}$ to $\fa_R$ equals $(2\rho_{R,+}-2\rho_{R_H})\big|_{\fa_R}$. 

\begin{lem}\label{balanced1}
 Assume that $p=q=n$. Let $Q$ be a relatively standard parabolic subgroup of $G$. For all $\varpi^\vee\in\wh{\Delta}_Q^\vee$, we have $(2\rho_{Q,+}-2\rho_{Q_H})(\varpi^\vee)\geq 0$. Moreover, $2\rho_{Q,+}-2\rho_{Q_H}$ viewed as an element of $(\fa_Q^G)^\ast$ is zero if and only if $Q$ is $\omega$-stable.
\end{lem}

\begin{proof}
  Put $e_i^\ast\in\fa_0^\ast$ (resp. $f_i^\ast\in\fa_0^\ast$) to be the character of the action of $A_0$ on $e_i$ (resp. $f_i$). Write $e_i^\vee\in\fa_0$ (resp. $f_i^\vee\in\fa_0$) to be the dual basis, i.e., $e_i^\ast(e_j^\vee)=\delta_{ij}$ (resp. $f_i^\ast(f_j^\vee)=\delta_{ij}$) for $1\leq i,j\leq n$. A basis of $\fa_Q$ is given by $h_i^\vee:=e_{p_1+\cdot\cdot\cdot+p_{i-1}+1}^\vee+\cdot\cdot\cdot+e_{p_1+\cdot\cdot\cdot+p_i}^\vee+f_{q_1+\cdot\cdot\cdot+q_{i-1}+1}^\vee+\cdot\cdot\cdot+f_{q_1+\cdot\cdot\cdot+q_i}^\vee$ for $1\leq i\leq l$. Write $h_i^\ast\in(\fa_Q)^\ast$ to be the dual basis. Denote
  $$ \varpi_k^\vee:=\frac{\sum\limits_{i=k+1}^{l}(p_i+q_i)}{2n}(h_1^\vee+\cdot\cdot\cdot+h_k^\vee)
  -\frac{\sum\limits_{i=1}^{k}(p_i+q_i)}{2n}(h_{k+1}^\vee+\cdot\cdot\cdot+h_l^\vee). $$
  Recall that
  $$ (\wh{\Delta}_Q^G)^\vee=\{\varpi_k^\vee|1\leq k\leq l-1\} $$
  is a basis of $\fa_Q^G$. We can also see that
  $$ 2\rho_{Q,+}\big|_{\fa_Q}=\dim_F(D)\sum_{1\leq i< j\leq l} (p_i q_j+q_i p_j)(h_i^\ast-h_j^\ast) $$
  and that
  $$ 2\rho_{Q_H}\big|_{\fa_Q}=\dim_F(D)\sum_{1\leq i< j\leq l} (p_i p_j+q_i q_j)(h_i^\ast-h_j^\ast), $$
  so
  $$ (2\rho_{Q,+}-2\rho_{Q_H})\big|_{\fa_Q}=\dim_F(D)\sum_{1\leq i< j\leq l} (p_i-q_i)(q_j-p_j)(h_i^\ast-h_j^\ast). $$
  Since $\sum\limits_{i=1}^l p_i=\sum\limits_{i=1}^l q_i=n$, we have
  $$ (h_i^\ast-h_j^\ast)(\varpi_k^\vee)=
  \begin{cases}
    0, & \text{if $k+1\leq i< j\leq l$ or $1\leq i< j\leq k$;} \\
    1, & \text{if $1\leq i\leq k$ and $k+1\leq j\leq l$.}
  \end{cases} $$
  Then
  \[\begin{split}
    (2\rho_{Q,+}-2\rho_{Q_H})(\varpi_k^\vee)&=\dim_F(D)\sum_{\substack{1\leq i\leq k \\ k+1\leq j\leq l}} (p_i-q_i)(q_j-p_j) \\
    &=\dim_F(D)\left(\sum_{1\leq i\leq k}p_i-\sum_{1\leq i\leq k}q_i\right) \left(\sum_{k+1\leq j\leq l}q_j-\sum_{k+1\leq j\leq l}p_j\right) \\
    &=\dim_F(D)\left(\sum_{1\leq i\leq k}p_i-\sum_{1\leq i\leq k}q_i\right)^2\geq 0.
  \end{split}\]
  It is clear that $(2\rho_{Q,+}-2\rho_{Q_H})(\varpi_k^\vee)=0$ for all $1\leq k\leq l-1$ if and only if $p_i=q_i$ for all $1\leq i\leq l$.
\end{proof}

\subsection{Exponential polynomials}

Let $T_1,T_2\in\fa_0$. Following \cite[\S2]{MR625344}, define $\Gamma_P(T_1,T_2)\in\BR$ inductively on $\dim(A_P/A_{G})$ by setting
$$ \wh{\tau}_P^G(T_1-T_2)=\sum_{\{Q:P\subseteq Q\}} (-1)^{\dim(A_Q/A_{G})} \wh{\tau}_P^Q(T_1) \Gamma_Q(T_1,T_2) $$
for any relatively standard parabolic subgroup $P$ of $G$. This definition can be explicitly given by \cite[(2.1) in p. 13]{MR625344} and only depends on the projections of $T_1,T_2$ onto $\fa_P^{G}$. 
For $T=(t_1,...,t_{p+q})\in\fa_0$, we denote $\Sigma_1(T):=t_1+...+t_p$. If we use the notation in Section \ref{relstdpar1} and put $e_i^\ast\in\fa_0^\ast$ (resp. $f_i^\ast\in\fa_0^\ast$) to be the character of the action of $A_0$ on $e_i$ (resp. $f_i$), it is equivalent to say that $\Sigma_1=\sum_{1\leq i\leq p} e_i^\ast$. 
For $T_2\in\fa_Q$ and $s\in\BC$, write
\begin{equation}\label{equation1.3}
 p_{Q,s}(T_2):=\int_{\fa_Q^G} e^{(2\rho_{Q,+}-2\rho_{Q_H}+s\Sigma_1)(T_1)} \Gamma_Q(T_1,T_2) dT_1. 
\end{equation}
When $p=q=n, s=0$ and $Q$ is $\omega$-stable, it is reduced to
$$ p_{Q,0}(T_2)=\int_{\fa_Q^G} \Gamma_Q(T_1,T_2) dT_1 $$
by Lemma \ref{balanced1}. 

For $Q\subseteq R$ a pair of relatively standard parabolic subgroups of $G$, denote by $\BZ(\wh{\Delta}_Q^R)^\vee$ the lattice generated by $(\wh{\Delta}_Q^R)^\vee$ in $\fa_Q^R$ and by $\BZ(\Delta_R^G)^\vee$ the lattice generated by $(\Delta_R^G)^\vee$ in $\fa_R^G$. Following \cite[\S2]{MR625344}, for $\lambda\in\fa_{Q,\BC}^\ast:=\fa_Q^\ast\otimes_\BR \BC$, define
$$ \wh{\theta}_Q^R(\lambda):=\vol(\fa_Q^R\big/\BZ(\wh{\Delta}_Q^R)^\vee)^{-1} \prod_{\varpi^\vee\in(\wh{\Delta}_Q^R)^\vee} \lambda(\varpi^\vee)  $$
and
$$ \theta_R^G(\lambda):=\vol(\fa_R^G\big/\BZ(\Delta_R^G)^\vee)^{-1} \prod_{\alpha^\vee\in(\Delta_R^G)^\vee} \lambda(\alpha^\vee). $$

\begin{prop}\label{propofp_Q,s}
  Let Q be a relatively standard parabolic subgroup of $G$, $T_2\in\fa_{Q}$ and $s\in\BC$. The function $T_1\mapsto\Gamma_Q(T_1,T_2)$ is compactly supported on $\fa_Q^{G}$. Moreover, the function $T_2\mapsto p_{Q,s}(T_2)$ is an exponential polynomial in $T_2$; more precisely, there exists a polynomial $p_{Q,R,s}$ (not necessarily unique) on $\fa_R^G$ of degree $\leq \dim(A_Q/A_G)$ for each relatively standard parabolic subgroup $R$ containing $Q$ such that
  $$ p_{Q,s}(T_2)=\sum_{\{R:Q\subseteq R\}} e^{(2\rho_{R,+}-2\rho_{R_H}+s\Sigma_1)(T_{2,R}^G)} p_{Q,R,s}(T_{2,R}^G), $$
where we write $T_{2,R}^G$ for the projection of $T_2\in\fa_Q$ in $\fa_R^G$ via the decomposition $\fa_Q=\fa_Q^R\oplus\fa_R^G\oplus\fa_G$. When $p=q=n$ and $s=0$, the purely polynomial term of $p_{Q,0}(T_2)$ is given by
  $$ \sum_{\{R: Q\subseteq R, \omega R\omega^{-1}=R\}} p_{Q,R,0}(T_{2,R}^G), $$
  which is a homogeneous polynomial in $T_2$ of degree $\dim(A_{\ov{Q}^{\omega\text{-st}}}/A_G)$;
  in particular, if $Q$ is $\omega$-stable, then $p_{Q,0}(T_2)$ is a homogeneous polynomial in $T_2$ of degree $\dim(A_Q/A_G)$.
\end{prop}

\begin{proof}
  The first statement is \cite[Lemmas 2.1]{MR625344}. First let us prove the second one.

  From \cite[Lemma 2.2]{MR625344}, we know that the integral
  $$ \int_{\fa_Q^G} e^{\lambda (T_1)} \Gamma_Q(T_1,T_2) dT_1 $$
  is an entire function in $\lambda\in\fa_{Q,\BC}^\ast$, and its value is given by
  $$ \sum_{\{R:Q\subseteq R\}} (-1)^{\dim(A_Q/A_R)} e^{\lambda(T_{2,R}^G)} \wh{\theta}_Q^R(\lambda)^{-1} \theta_R^G(\lambda)^{-1} $$
  when the latter expression makes sense.

  Fix $\varepsilon\in\fa_{Q,\BC}^\ast$ such that $\wh{\theta}_Q^R(\varepsilon)\neq 0$ and $\theta_R^G(\varepsilon)\neq 0$ for all relatively standard parabolic subgroups $R$ containing $Q$. Then for $t\in\BR^\times$ whose absolute value is small enough, we also have $\wh{\theta}_Q^R(2\rho_{Q,+}-2\rho_{Q_H}+s\Sigma_1+t\varepsilon)\neq 0$ and $\theta_R^G(2\rho_{Q,+}-2\rho_{Q_H}+s\Sigma_1+t\varepsilon)\neq 0$ for all relatively standard parabolic subgroups $R$ containing $Q$. Let $\lambda=2\rho_{Q,+}-2\rho_{Q_H}+s\Sigma_1+t\varepsilon$ in the formula above, and we obtain
  \[\begin{split}
  p_{Q,s}(T_2)=&\lim_{t\mapsto 0} \sum_{\{R:Q\subseteq R\}} (-1)^{\dim(A_Q/A_R)} e^{(2\rho_{Q,+}-2\rho_{Q_H}+s\Sigma_1+t\varepsilon)(T_{2,R}^G)} \wh{\theta}_Q^R(2\rho_{Q,+}-2\rho_{Q_H}+s\Sigma_1+t\varepsilon)^{-1} \\
  &\cdot \theta_R^G(2\rho_{Q,+}-2\rho_{Q_H}+s\Sigma_1+t\varepsilon)^{-1}. 
  \end{split}\]

  Since the restriction of $2\rho_{Q,+}-2\rho_{Q_H}+s\Sigma_1$ to $\fa_R$ equals $2\rho_{R,+}-2\rho_{R_H}+s\Sigma_1$, we get
  $$ e^{(2\rho_{Q,+}-2\rho_{Q_H}+s\Sigma_1)(T_{2,R}^G)}=e^{(2\rho_{R,+}-2\rho_{R_H}+s\Sigma_1)(T_{2,R}^G)}. $$
  We can put $p_{Q,R,s}(T_{2,R}^G)$ to be the constant term of the Laurent series development around $t=0$ of
  $$ t\mapsto (-1)^{\dim(A_Q/A_R)} e^{(t\varepsilon)(T_{2,R}^G)} \wh{\theta}_Q^R(2\rho_{Q,+}-2\rho_{Q_H}+s\Sigma_1+t\varepsilon)^{-1} \theta_R^G(2\rho_{Q,+}-2\rho_{Q_H}+s\Sigma_1+t\varepsilon)^{-1}. $$
  Then $p_{Q,R,s}(T_{2,R}^G)$ is a polynomial in $T_{2,R}^G$ of degree $\leq\dim(A_Q/A_G)$. Hence we prove the existence in the second statement.

  Now let $p=q=n$ and $s=0$. From Lemma \ref{balanced1}, we know that the purely polynomial term of $p_{Q,0}(T_2)$ is given by
  $$ \sum_{\{R: Q\subseteq R, \omega R\omega^{-1}=R\}} p_{Q,R,0}(T_{2,R}^G). $$
  Next we compute the degree of $p_{Q,R,0}$ that we chose above for each $\omega$-stable parabolic subgroup $R$ containing $Q$. Denote
  $$ N_1:=\sharp\{\varpi^\vee\in(\wh{\Delta}_Q^R)^\vee:(2\rho_{Q,+}-2\rho_{Q_H})(\varpi^\vee)=0\} $$
  and
  $$ N_2:=\sharp\{\alpha^\vee\in(\Delta_R^G)^\vee:(2\rho_{Q,+}-2\rho_{Q_H})(\alpha^\vee)=0\}, $$
where $\sharp$ means the cardinality of a finite set. Then
  $$ \deg (p_{Q,R,0})=N_1+N_2. $$
  Recall that both of $(\wh{\Delta}_R^G)^\vee$ and $(\Delta_R^G)^\vee$ are bases of $\fa_R^G$. Since $R$ is relatively standard and $\omega$-stable, by Lemma \ref{balanced1}, we have
  $$ N_2=\dim(A_R/A_G). $$
  Keep the notation as in the proof of Lemma \ref{balanced1} for $Q$. Since $R$ is relatively standard and $\omega$-stable, by Proposition \ref{omegapq}, we may suppose that $R$ is the stabiliser in $G$ of the flag
  $$ 0\subsetneq\langle e_1,\cdot\cdot\cdot,e_{r_1},f_1,\cdot\cdot\cdot,f_{r_1}\rangle_D
    \subsetneq\langle e_1,\cdot\cdot\cdot,e_{r_1+r_2},f_1,\cdot\cdot\cdot,f_{r_1+r_2}\rangle_D
    \subsetneq...\subsetneq\langle e_1,\cdot\cdot\cdot,e_{r_1+\cdot\cdot\cdot+r_{l'}},f_1,\cdot\cdot\cdot,f_{r_1+\cdot\cdot\cdot+r_{l'}}\rangle_D. $$
  The fact that $Q\subseteq R$ tells us that both of the partitions $(p_1,\cdot\cdot\cdot,p_l)$ and $(q_1,\cdot\cdot\cdot,q_l)$ are refinements of the partition $(r_1,\cdot\cdot\cdot,r_{l'})$ of $n$, and that every $r_i$ is divided into the same number of segments in these two refinements. Then
  \[\begin{split}
    (\wh{\Delta}_Q^R)^\vee&=\left\{\text{projection of $\varpi_k^\vee\in(\wh{\Delta}_Q^G)^\vee$ to $\fa_Q^R$} \bigg| 1\leq k\leq l-1, \sum_{i=1}^k (p_i+q_i)\neq \sum_{i=1}^j 2r_i \forall 1\leq j\leq l'-1\right\} \\
    &=\left\{\text{projection of $\varpi_k^\vee\in(\wh{\Delta}_Q^G)^\vee$ to $\fa_Q^R$} \bigg| 1\leq k\leq l-1, \nexists 1\leq j\leq l'-1 s.t. \sum_{i=1}^k p_i=\sum_{i=1}^k q_i=\sum_{i=1}^j r_i\right\}.
  \end{split}\]
  Because the restriction of $2\rho_{Q,+}-2\rho_{Q_H}$ to $\fa_R$ equals $2\rho_{R,+}-2\rho_{R_H}$ and $R$ is relatively standard and $\omega$-stable, by Lemma \ref{balanced1}, we do not need the projection, i.e.,
  $$ (2\rho_{Q,+}-2\rho_{Q_H})\left(\text{projection of $\varpi_k^\vee\in(\wh{\Delta}_Q^G)^\vee$ to $\fa_Q^R$}\right)=(2\rho_{Q,+}-2\rho_{Q_H})\left(\varpi_k^\vee\in(\wh{\Delta}_Q^G)^\vee\right). $$
  From the proof of Lemma \ref{balanced1}, for any $1\leq k\leq l$, we have $(2\rho_{Q,+}-2\rho_{Q_H})(\varpi_k^\vee)=0$ if and only if $\sum\limits_{i=1}^k p_i=\sum\limits_{i=1}^k q_i$. We can also see that $\ov{Q}^{\omega\text{-st}}$ is the $\omega$-stable parabolic subgroup $R$ containing $Q$ with maximal $l':=\dim(A_R)$. To sum up, we have
  $$ N_1=\dim(A_{\ov{Q}^{\omega\text{-st}}}/A_R). $$
  Hence for each $\omega$-stable parabolic subgroup $R$ containing $Q$,
  $$ \deg (p_{Q,R,0})=N_1+N_2=\dim(A_{\ov{Q}^{\omega\text{-st}}}/A_R)+\dim(A_R/A_G)=\dim(A_{\ov{Q}^{\omega\text{-st}}}/A_G). $$
  The assertion about the particular case where $Q$ is $\omega$-stable is \cite[Lemma 2.2]{MR625344} combined with Lemma \ref{balanced1}; it can also be read from the results above that we have proved. 
\end{proof}

\subsection{Quantitive behaviour in $T$}

For a relatively standard parabolic subgroup $Q$ of $G$, let $\{s^Q_i\}_{1\leq i\leq l}\in\BZ^l$ be the explicit constants determined by
\begin{equation}\label{explicitcst}
 \prod_{1\leq i\leq l}|\Nrd (x_{i,1})|_{\BA}^{s^Q_i}=e^{(2\rho_{Q,+}-2\rho_{Q_H})(H_{Q_H}(x))}
\end{equation}
for all $x\in M_{Q_H}(\BA)\cap M_Q(\BA)^1$, 
where we write $x=(x_1,...,x_l)\in GL_{p_1+q_1,D}(\BA)\times\cdot\cdot\cdot\times GL_{p_l+q_l,D}(\BA)$ and $x_i=(x_{i,1},x_{i,2})\in GL_{p_i,D}(\BA)\times GL_{q_i,D}(\BA)$. If $p_i q_i=0$ for some $1\leq i\leq l$, we shall take $|\Nrd (x_{i,1})|_{\BA}^{s^Q_i}=1$ and $s^Q_i=0$ by convention. Then such constants are unique. 

\begin{prop}\label{propexplicitcst}
Let $Q$ be a relatively standard parabolic subgroup of $G$. If $p_i q_i\neq 0$ for some $1\leq i\leq l$, then
$$ s^Q_i=2d\left(\sum_{k<i}(p_k-q_k)+\sum_{k>i}(q_k-p_k)\right). $$
When $p=q=n$, if $Q$ is $\omega$-stable, then $s^Q_i=0$ for all $1\leq i\leq l$. 
\end{prop}

\begin{proof}
Assume that $p_i q_i\neq 0$ for some $1\leq i\leq l$. Let $x\in M_{Q_H}(\BA)$. We have
\begin{enumerate}[\indent (1)]
\item the contribution of $x_{i,1}$ to $e^{2\rho_{Q,+}(H_{Q_H}(x))}$ is the $d\Big(\sum\limits_{k>i}q_k-\sum\limits_{k<i}q_k\Big)$-th power of $|\Nrd (x_{i,1})|_{\BA}$; 
\end{enumerate}
\begin{enumerate}[\indent (2)]
\item the contribution of $x_{i,1}$ to $e^{2\rho_{Q_H}(H_{Q_H}(x))}$ is the $d\Big(\sum\limits_{k>i}p_k-\sum\limits_{k<i}p_k\Big)$-th power of $|\Nrd (x_{i,1})|_{\BA}$; 
\end{enumerate}
\begin{enumerate}[\indent (3)]
\item the contribution of $x_{i,2}$ to $e^{2\rho_{Q,+}(H_{Q_H}(x))}$ is the $d\Big(\sum\limits_{k>i}p_k-\sum\limits_{k<i}p_k\Big)$-th power of $|\Nrd (x_{i,2})|_{\BA}$; 
\end{enumerate}
\begin{enumerate}[\indent (4)]
\item the contribution of $x_{i,2}$ to $e^{2\rho_{Q_H}(H_{Q_H}(x))}$ is the $d\Big(\sum\limits_{k>i}q_k-\sum\limits_{k<i}q_k\Big)$-th power of $|\Nrd (x_{i,2})|_{\BA}$. 
\end{enumerate}
In sum, the contribution of $x_i$ to $e^{(2\rho_{Q,+}-2\rho_{Q_H})(H_{Q_H}(x))}$ is the product of the $d\Big(\sum\limits_{k<i}(p_k-q_k)+\sum\limits_{k>i}(q_k-p_k)\Big)$-th power of $|\Nrd (x_{i,1})|_{\BA}$ and the $d\Big(\sum\limits_{k<i}(q_k-p_k)+\sum\limits_{k>i}(p_k-q_k)\Big)$-th power of $|\Nrd (x_{i,2})|_{\BA}$. 

Now let $x\in M_{Q_H}(\BA)\cap M_Q(\BA)^1$. Then $|\Nrd (x_{i,1})\Nrd (x_{i,2})|_{\BA}=|\Nrd (x_{i})|_{\BA}=1$. Therefore, the contribution of $x_i$ to $e^{(2\rho_{Q,+}-2\rho_{Q_H})(H_{Q_H}(x))}$ is the $2d\Big(\sum\limits_{k<i}(p_k-q_k)+\sum\limits_{k>i}(q_k-p_k)\Big)$-th power of $|\Nrd (x_{i,1})|_{\BA}$. We have proved the first statement. 

The second statement is nothing but a special case of the first one, since we have $p_k=q_k$ for $1\leq k\leq l$ in this case. 
\end{proof}

\begin{thm}\label{exppol1}
  Let $T'$ be sufficiently regular, $\fo\in\CO$ and $f\in\CS(\fs(\BA))$. Then for all sufficiently regular $T$ and $s\in\BC$, we have
  $$ J_\fo^{G,T}(\eta,s,f)=\sum_{\{Q:\wt{P}_0\subseteq Q\}} p_{Q,s}(T_Q-T'_Q) e^{(2\rho_{Q,+}-2\rho_{Q_H}+s\Sigma_1)((T')_Q^G)} J_\fo^{Q,T'}(\eta,\{s^Q_i+s\},f_Q^\eta), $$
where we write $(T')_Q^G$ for the projection of $T'_Q\in\fa_Q$ in $\fa_Q^G$ via the decomposition $\fa_Q=\fa_Q^G\oplus\fa_G$, the distributions $J_\fo^{G,T}$ and $J_\fo^{Q,T'}$ are defined by the formulae (\ref{Jo1}) and (\ref{LeviJo1}) respectively, and $f_Q^\eta$ and $p_{Q,s}$ are defined by the formulae (\ref{equation1.2}) and (\ref{equation1.3}) respectively. 
\end{thm}

\begin{coro}\label{polynomial1}
  Let $\fo\in\CO$, $f\in\CS(\fs(\BA))$ and $s\in\BC$. Then the functions $T\mapsto J_\fo^{G,T}(\eta,s,f)$ and $T\mapsto J^{G,T}(\eta,s,f)$ are the restriction of exponential polynomials in $T$, so we can extend them to all $T\in\fa_{0}$. When $p=q=n$ and $s=0$, their purely polynomial terms have degree $\leq n-1$; in particular, if $\fo\in\CO^\times$ (e.g., $\fo\in\CO_{rs}$), $T\mapsto J_\fo^{G,T}(\eta,0,f)$ is the restriction of a polynomial in $T$ of degree $\leq n-1$.
\end{coro}

\begin{proof}[Proof of Corollary \ref{polynomial1}]
It results from Theorem \ref{exppol1}, Propositions \ref{propofp_Q,s} and \ref{propOinvert1}. 
\end{proof}

\begin{remark}\label{LevibehT1}
We may extend our result to the product form in Section \ref{pdtform1} by similar argument. Let $Q$ be a relatively standard parabolic subgroup of $G$. Let $\fo\in\CO$, $f'\in\CS((\fm_Q\cap\fs)(\BA))$ and $\{s_i\}_{1\leq i\leq l}\in\BC^l$. Then the functions $J_\fo^{Q,T}(\eta,\{s_i\},f')$ and $J^{Q,T}(\eta,\{s_i\},f')$ are the restriction of exponential polynomials in $T$ independent of $T_Q$, so we can extend them to all $T\in\fa_{0}$. 
\end{remark}

\begin{proof}[Proof of Theorem \ref{exppol1}]
  Let $P$ be a relatively standard parabolic subgroup of $G$, $\delta\in P_H(F)\bs H(F)$ and $x\in H(\BA)\cap G(\BA)^1$. Substituting $T_1=H_P(\delta x)-T'_P$ and $T_2=T_P-T'_P$ in the definition of $\Gamma_P(T_1,T_2)$, we get
  $$ \wh{\tau}_P^G(H_P(\delta x)-T_P)=\sum_{\{Q:P\subseteq Q\}} (-1)^{\dim(A_Q/A_{G})} \wh{\tau}_P^Q(H_P(\delta x)-T'_P) \Gamma_Q(H_P(\delta x)-T'_P,T_P-T'_P). $$
  Then
  \[\begin{split}
    J_\fo^{G,T}(\eta,s,f)=&\int_{H(F)\bs H(\BA)\cap G(\BA)^1} \left(\sum_{\{P: \wt{P}_0\subseteq P\}} (-1)^{\dim(A_P/A_{G})} \sum_{\delta\in P_H(F)\bs H(F)} \wh{\tau}_P^G(H_{P}(\delta x)-T_P) \cdot k_{f,P,\fo}(\delta x)\right) \\ &\cdot \eta(\Nrd (x)) |\Nrd (x_1)|_\BA^s dx \\
    =&\int_{H(F)\bs H(\BA)\cap G(\BA)^1} \sum_{\{P: \wt{P}_0\subseteq P\}} (-1)^{\dim(A_P/A_{G})} \sum_{\delta\in P_H(F)\bs H(F)} \\ &\left(\sum_{\{Q:P\subseteq Q\}} (-1)^{\dim(A_Q/A_{G})} \wh{\tau}_P^Q(H_P(\delta x)-T'_P) \Gamma_Q(H_P(\delta x)-T'_P,T_P-T'_P)\right) k_{f,P,\fo}(\delta x) \\ &\cdot \eta(\Nrd (x)) |\Nrd (x_1)|_\BA^s dx.
  \end{split}\]
  Exchanging the order of two sums over $P$ and $Q$, and decomposing the sum over $P_H(F)\bs H(F)$ into two sums over $P_H(F)\bs Q_H(F)$ and $Q_H(F)\bs H(F)$, we have
  \[\begin{split}
    J_\fo^{G,T}(\eta,s,f)=&\sum_{\{Q:\wt{P}_0\subseteq Q\}} \int_{H(F)\bs H(\BA)\cap G(\BA)^1} \sum_{\{P:\wt{P}_0\subseteq P\subseteq Q\}} (-1)^{\dim(A_P/A_{Q})} \sum_{\delta'\in Q_H(F)\bs H(F)} \sum_{\delta\in P_H(F)\bs Q_H(F)}  \\
    &\wh{\tau}_P^Q(H_P(\delta\delta' x)-T'_P) \Gamma_Q(H_P(\delta\delta' x)-T'_P,T_P-T'_P) k_{f,P,\fo}(\delta\delta' x) \eta(\Nrd (x)) |\Nrd (x_1)|_\BA^s dx.
  \end{split}\]
  Combining the integral over $H(F)\bs H(\BA)\cap G(\BA)^1$ and the sum over $Q_H(F)\bs H(F)$ into an integral over $Q_H(F)\bs H(\BA)\cap G(\BA)^1$, and using the fact that
  $$ P_H(F)\bs Q_H(F)\simeq (P_H(F)\cap M_{Q_H}(F))\bs M_{Q_H}(F), $$
  we obtain
  \[\begin{split}
    J_\fo^{G,T}(\eta,s,f)=&\sum_{\{Q:\wt{P}_0\subseteq Q\}} \int_{Q_H(F)\bs H(\BA)\cap G(\BA)^1} \sum_{\{P:\wt{P}_0\subseteq P\subseteq Q\}} (-1)^{\dim(A_P/A_{Q})} \sum_{\delta\in (P_H(F)\cap M_{Q_H}(F))\bs M_{Q_H}(F)} \\
    &\wh{\tau}_P^Q(H_P(\delta x)-T'_P) \Gamma_Q(H_P(\delta x)-T'_P,T_P-T'_P) k_{f,P,\fo}(\delta x) \eta(\Nrd (x)) |\Nrd (x_1)|_\BA^s dx.
  \end{split}\]

  By Iwasawa decomposition and our choices of measures, the integral over $Q_H(F)\bs H(\BA)\cap G(\BA)^1$ can be decomposed as integrals over 
  $$ (n,a,m,k)\in N_{Q_H}(F)\bs N_{Q_H}(\BA)\times A_Q^{G,\infty}\times M_{Q_H}(F)\bs M_{Q_H}(\BA)\cap M_Q(\BA)^1\times K_H. $$
  Then
  \[\begin{split}
    J_\fo^{G,T}(\eta,s,f)=&\sum_{\{Q:\wt{P}_0\subseteq Q\}} \int_{K_H}\int_{M_{Q_H}(F)\bs M_{Q_H}(\BA)\cap M_Q(\BA)^1}\int_{A_Q^{G,\infty}}\int_{N_{Q_H}(F)\bs N_{Q_H}(\BA)} \sum_{\{P:\wt{P}_0\subseteq P\subseteq Q\}} (-1)^{\dim(A_P/A_{Q})} \\
    &\sum_{\delta\in (P_H(F)\cap M_{Q_H}(F))\bs M_{Q_H}(F)} \wh{\tau}_P^Q(H_P(\delta namk)-T'_P) \Gamma_Q(H_P(\delta namk)-T'_P,T_P-T'_P) \\
    &\cdot k_{f,P,\fo}(\delta namk) \eta(\Nrd (mk)) |\Nrd(a_1 m_1)|_\BA^s e^{-2\rho_{Q_H}(H_{Q_H}(am))}dndadmdk.
  \end{split}\]
  Notice that
  $$ \wh{\tau}_P^Q(H_P(\delta namk)-T'_P)=\wh{\tau}_P^Q(H_P(\delta m)+H_P(a)-T'_P)=\wh{\tau}_P^Q(H_P(\delta m)-T'_P), $$
  and that
  $$ \Gamma_Q(H_P(\delta namk)-T'_P,T_P-T'_P)=\Gamma_Q(H_Q(\delta namk)-T'_Q,T_Q-T'_Q)=\Gamma_Q(H_Q(a)-T'_Q,T_Q-T'_Q). $$
  In addition, by change of variables, we see that
  \[\begin{split}
    k_{f,P,\fo}(\delta namk)&=\sum_{X\in\fm_P(F)\cap\fo} \int_{(\fn_P\cap\fs)(\BA)} f((\delta namk)^{-1}(X+U)\delta namk) dU \\
    &=\sum_{X\in\fm_P(F)\cap\fo} \int_{(\fn_P\cap\fs)(\BA)} f((\delta a^{-1}namk)^{-1}(X+a^{-1}Ua)\delta a^{-1}namk) dU \\
    &=\sum_{X\in\fm_P(F)\cap\fo} \int_{(\fn_P\cap\fs)(\BA)} f((\delta a^{-1}namk)^{-1}(X+U)\delta a^{-1}namk) e^{2\rho_{Q,+}(H_{Q}(a))} dU \\
    &=e^{2\rho_{Q,+}(H_{Q}(a))}k_{f,P,\fo}(\delta a^{-1}namk).
  \end{split}\]
  Since $\delta a^{-1}na\delta^{-1}\in N_{Q_H}(\BA)\subseteq N_{P_H}(\BA)$ and $k_{f,P,\fo}$ is left invariant by $N_{P_H}(\BA)$, we deduce that
  $$ k_{f,P,\fo}(\delta namk)=e^{2\rho_{Q,+}(H_{Q}(a))}k_{f,P,\fo}(\delta mk). $$
  In sum, the integrand for the term indexed by $Q$ in the above formula for $J_\fo^{G,T}(\eta,s,f)$ is independent of $n\in N_{Q_H}(F)\bs N_{Q_H}(\BA)$. We can choose the Haar measure such that $\vol(N_{Q_H}(F)\bs N_{Q_H}(\BA))=1$. Then
  \[\begin{split}
    J_\fo^{G,T}(\eta,s,f)=&\sum_{\{Q:\wt{P}_0\subseteq Q\}} \left(\int_{A_Q^{G,\infty}} |\Nrd (a_1)|_\BA^s e^{(2\rho_{Q,+}-2\rho_{Q_H})(H_Q(a))} \Gamma_Q(H_Q(a)-T'_Q,T_Q-T'_Q) da\right) \\
    &\int_{M_{Q_H}(F)\bs M_{Q_H}(\BA)\cap M_Q(\BA)^1} \sum_{\{P:\wt{P}_0\subseteq P\subseteq Q\}} (-1)^{\dim(A_P/A_{Q})} \sum_{\delta\in (P_H(F)\cap M_{Q_H}(F))\bs M_{Q_H}(F)} \\
    &\wh{\tau}_P^Q(H_P(\delta m)-T'_P) \left(\int_{K_H} k_{f,P,\fo}(\delta mk) \eta(\Nrd (k)) dk\right) \eta(\Nrd (m)) |\Nrd (m_1)|_\BA^s e^{-2\rho_{Q_H}(H_{Q_H}(m))} dm.
  \end{split}\]

  By the definition of the Haar measure on $A_Q^{G,\infty}$, we have
  \[\begin{split}
    &\int_{A_Q^{G,\infty}} |\Nrd (a_1)|_\BA^s e^{(2\rho_{Q,+}-2\rho_{Q_H})(H_Q(a))} \Gamma_Q(H_Q(a)-T'_Q,T_Q-T'_Q) da \\
    :=&\int_{\fa_Q^{G}} e^{(2\rho_{Q,+}-2\rho_{Q_H}+s\Sigma_1)(T_1)} \Gamma_Q(T_1-T'_Q,T_Q-T'_Q) dT_1 \\
    =&e^{(2\rho_{Q,+}-2\rho_{Q_H}+s\Sigma_1)((T')_Q^G)} \int_{\fa_Q^{G}} e^{(2\rho_{Q,+}-2\rho_{Q_H}+s\Sigma_1)(T_1)} \Gamma_Q(T_1,T_Q-T'_Q) dT_1 \\
    =&e^{(2\rho_{Q,+}-2\rho_{Q_H}+s\Sigma_1)((T')_Q^G)} p_{Q,s}(T_Q-T'_Q).
  \end{split}\]
  Since $\fn_P=\fn_P^Q\oplus\fn_Q$, by change of variables, we see that
  \[\begin{split}
    k_{f,P,\fo}(\delta mk)&=\sum_{X\in\fm_P(F)\cap\fo} \int_{(\fn_P^Q\cap\fs)(\BA)} dU \int_{(\fn_Q\cap\fs)(\BA)} f((\delta mk)^{-1}(X+U+V)\delta mk) dV \\
    &=e^{2\rho_{Q,+}(H_{Q_H}(m))} \sum_{X\in\fm_P(F)\cap\fo} \int_{(\fn_P^Q\cap\fs)(\BA)} dU \int_{(\fn_Q\cap\fs)(\BA)} f(k^{-1}((\delta m)^{-1}(X+U)\delta m+V)k) dV,
  \end{split}\]
  so we can write
  \[\begin{split}
    \int_{K_H} k_{f,P,\fo}(\delta mk)\eta(\Nrd (k)) dk&=e^{2\rho_{Q,+}(H_{Q_H}(m))} \sum_{X\in\fm_P(F)\cap\fo} \int_{(\fn_P^Q\cap\fs)(\BA)} f_Q^\eta((\delta m)^{-1}(X+U)\delta m) dU \\
    &=e^{2\rho_{Q,+}(H_{Q_H}(m))} \sum_{j=1}^{t} k_{f_Q^\eta,P\cap M_Q,\fo_j}^{M_Q} (\delta m)
  \end{split}\]
  by (\ref{LevikfPo1}). Now we can draw our conclusion by noting that
  \[\begin{split}
    J_{\fo}^{Q,T'}(\eta,\{s^Q_i+s\},f_Q^\eta)=&\sum_{j=1}^{t} \int_{M_{Q_H}(F)\bs M_{Q_H}(\BA)\cap M_Q(\BA)^1} \sum_{\{P:\wt{P}_0\subseteq P\subseteq Q\}} (-1)^{\dim(A_{P\cap M_Q}/A_{M_Q})} \\
    &\sum_{\delta\in ((P\cap M_Q)(F)\cap M_{Q_H}(F))\bs M_{Q_H}(F)} \wh{\tau}_{P\cap M_Q}^{M_Q}(H_{P\cap M_Q}(\delta m)-(\varsigma_Q T')_{P\cap M_Q}) \\
    &\cdot k_{f_Q^\eta,P\cap M_Q,\fo_j}^{M_Q}(\delta m) \eta(\Nrd (m)) |\Nrd (m_1)|_\BA^s e^{(2\rho_{Q,+}-2\rho_{Q_H})(H_{Q_H}(m))} dm \\
    =&\int_{M_{Q_H}(F)\bs M_{Q_H}(\BA)\cap M_Q(\BA)^1} \sum_{\{P:\wt{P}_0\subseteq P\subseteq Q\}} (-1)^{\dim(A_P/A_{Q})} \sum_{\delta\in (P_H(F)\cap M_{Q_H}(F))\bs M_{Q_H}(F)} \\
    &\wh{\tau}_P^Q(H_P(\delta m)-T'_P) \left(\sum_{j=1}^{t} k_{f_Q^\eta,P\cap M_Q,\fo_j}^{M_Q}(\delta m)\right) \eta(\Nrd (m)) |\Nrd (m_1)|_\BA^s \\         &\cdot e^{(2\rho_{Q,+}-2\rho_{Q_H})(H_{Q_H}(m))} dm.
  \end{split}\]
\end{proof}

\subsection{Independence of constant terms}\label{indep1}

  Let $J_\fo^{G}(\eta,s,f)$ and $J^{G}(\eta,s,f)$ be the constant terms of $J_\fo^{G,T}(\eta,s,f)$ and $J^{G,T}(\eta,s,f)$ respectively. We fix a common minimal Levi subgroup $M_0$ of $H$ and $G$. 

Firstly, the distributions $J_\fo^{G}(\eta,s,f)$ and $J^{G}(\eta,s,f)$ are independent of the choice of the relatively standard minimal parabolic subgroup $P_0$ of $G$ at the very beginning of last section. In fact, let $P'_0$ be another relatively standard minimal parabolic subgroup of $G$ and $\sigma\in\Omega^G$ such that $P'_0=\sigma P_0$. Denote by $J_{P'_0,\fo}^{G,T}(\eta,s,f)$ and $J_{P'_0,\fo}^{G}(\eta,s,f)$ the distributions obtained starting from $P'_0$. Then if $T\in\fa_{P'_0}$, we have $J_{P'_0,\fo}^{G,T}(\eta,s,f)=J_\fo^{G,\sigma^{-1}T}(\eta,s,f)$, so $J_{P'_0,\fo}^{G}(\eta,s,f)=J_\fo^{G}(\eta,s,f)$. 

Secondly, the distributions $J_\fo^{G}(\eta,s,f)$ and $J^{G}(\eta,s,f)$ are independent of the choice of the minimal parabolic subgroup $\wt{P}_0$ of $H$. In fact, let $\wt{P}'_0$ be another minimal parabolic subgroup of $H$ and $\sigma\in\Omega^H$ such that $\wt{P}'_0=\sigma^{-1}\wt{P}_0$. Put $P'_0:=\sigma^{-1}P_0$. Denote by $J_{\wt{P}'_0,\fo}^{G,T}(\eta,s,f)$ and $J_{\wt{P}'_0,\fo}^{G}(\eta,s,f)$ the distributions obtained starting from $\wt{P}'_0$ and $P'_0$. We can apply the argument of \cite[Proposition 4.6]{MR1893921} after some minor modifications here to prove that $J_\fo^{G,T}(\eta,s,f)=J_{\wt{P}'_0,\fo}^{G,\sigma^{-1}T}(\eta,s,f)$, so $J_\fo^{G}(\eta,s,f)=J_{\wt{P}'_0,\fo}^{G}(\eta,s,f)$. 


\section{\textbf{Non-equivariance}}\label{nonequi}

Let $Q$ be a relatively standard parabolic subgroup of $G$, $s\in\BR$ and $y\in H(\BA)\cap G(\BA)^1$. For $f\in\CS(\fs(\BA))$, define $f_{Q,s,y}^\eta\in\CS((\fm_Q\cap\fs)(\BA))$ by
\begin{equation}\label{twconst1}
 f_{Q,s,y}^\eta(X):=\int_{K_H} \int_{(\fn_Q\cap\fs)(\BA)} f(k^{-1}(X+V)k) \eta(\Nrd (k))p_{Q,s}(-H_Q(ky)) dVdk
\end{equation}
for all $X\in(\fm_Q\cap\fs)(\BA)$, where $p_{Q,s}$ is defined by the formula (\ref{equation1.3}). 

\begin{prop}\label{nonequivariance1}
For $f\in\CS(\fs(\BA))$ and $y\in H(\BA)\cap G(\BA)^1$, we denote $f^y(x):=f(yxy^{-1})$. Then for all sufficiently regular $T$, $\fo\in\CO$ and $s\in\BR$, we have
$$ J_\fo^{G,T}(\eta,s,f^y)=\eta(\Nrd (y)) |\Nrd (y_1)|_\BA^s \sum_{\{Q:\wt{P}_0\subseteq Q\}} e^{(2\rho_{Q,+}-2\rho_{Q_H}+s\Sigma_1)(T_Q^G)} J_\fo^{Q,T}(\eta,\{s^Q_i+s\},f_{Q,s,y}^\eta), $$
where $J_\fo^{G,T}$ and $J_\fo^{Q,T}$ are defined by the formulae (\ref{Jo1}) and (\ref{LeviJo1}) respectively, $\{s^Q_i\}_{1\leq i\leq l}\in\BZ^l$ are the explicit constants determined by (\ref{explicitcst}), and we write $T_Q^G$ for the projection of $T_Q\in\fa_Q$ in $\fa_Q^G$ via the decomposition $\fa_Q=\fa_Q^G\oplus\fa_G$. 
\end{prop}

For $\fo\in\CO$ and $f\in\CS(\fs(\BA))$ (resp. $f'\in \CS((\fm_Q\cap\fs)(\BA))$), thanks to Corollary \ref{polynomial1} (resp. Remark \ref{LevibehT1}), we may take the constant term $J_\fo^{G}(\eta,s,f)$ of $J_\fo^{G,T}(\eta,s,f)$ (resp. $J_\fo^{Q}(\eta,\{s_i\}, f')$ of $J_\fo^{Q,T}(\eta,\{s_i\},f')$) for $s\in\BC$ (resp. $\{s_i\}_{1\leq i\leq l}\in\BC^l$). When $s=0$ (resp. $s_i=0$ for all $1\leq i\leq l$), denote $J_\fo^{G}(\eta,f):=J_\fo^{G}(\eta,0,f)$ (resp. $J_\fo^{Q}(\eta,f'):=J_\fo^{Q}(\eta,\{0\}, f')$). 

\begin{coro}\label{corofnonequivar1}
Assume that $p=q=n$. Let $f\in\CS(\fs(\BA)), y\in H(\BA)\cap G(\BA)^1$ and $\fo\in\CO$. We have
$$ J_\fo^{G}(\eta,f^y)=\eta(\Nrd (y)) \sum_{\{Q:\wt{P}_0\subseteq Q, \omega Q \omega^{-1}=Q\}} J_\fo^{Q}(\eta,f_{Q,0,y}^\eta). $$
\end{coro}

\begin{proof}[Proof of Corollary \ref{corofnonequivar1}]
We apply Proposition \ref{nonequivariance1} to the case $s=0$ and consider the constant terms of both sides. Because $J_\fo^{Q,T}$ is independent of $T_Q$, by Lemma \ref{balanced1}, only $\omega$-stable $Q$ contribute to the purely polynomial term. Then we apply Proposition \ref{propexplicitcst} to the case $p=q=n$ to conclude. 
\end{proof}

\begin{proof}[Proof of Proposition \ref{nonequivariance1}]
By definition, 
\[\begin{split}
    J_\fo^{G,T}(\eta,s,f^y)=&\int_{H(F)\bs H(\BA)\cap G(\BA)^1} \left(\sum_{\{P: \wt{P}_0\subseteq P\}} (-1)^{\dim(A_P/A_{G})} \sum_{\delta\in P_H(F)\bs H(F)} \wh{\tau}_P^G(H_{P}(\delta x)-T_P) \cdot k_{f^y,P,\fo}(\delta x)\right) \\ &\cdot \eta(\Nrd (x)) |\Nrd (x_1)|_\BA^s dx, 
  \end{split}\]
where
$$ k_{f^y,P,\fo}(\delta x)=\sum_{X\in\fm_P(F)\cap\fo}\int_{(\fn_P\cap\fs)(\BA)} f(y(\delta x)^{-1}(X+U)\delta xy^{-1}) dU=k_{f,P,\fo}(\delta xy^{-1}). $$
By change of variables, we have
\[\begin{split}
    J_\fo^{G,T}(\eta,s,f^y)=&\int_{H(F)\bs H(\BA)\cap G(\BA)^1} \left(\sum_{\{P: \wt{P}_0\subseteq P\}} (-1)^{\dim(A_P/A_{G})} \sum_{\delta\in P_H(F)\bs H(F)} \wh{\tau}_P^G(H_{P}(\delta xy)-T_P) \cdot k_{f,P,\fo}(\delta x)\right) \\ &\cdot \eta(\Nrd (xy)) |\Nrd (x_1 y_1)|_\BA^s dx. 
  \end{split}\]

For $x\in H(\BA)$ and $P$ a relatively standard parabolic subgroup of $G$, let $k_P(x)$ be an element in $K_H$ such that $xk_P(x)^{-1}\in P_H(\BA)$. Then
$$ \wh{\tau}_P^G(H_{P}(\delta xy)-T_P)=\wh{\tau}_P^G(H_{P}(\delta x)-T_P+H_P(k_P(\delta x)y)). $$
Substituting $T_1=H_{P}(\delta x)-T_P$ and $T_2=-H_P(k_P(\delta x)y)$ in the definition of $\Gamma_P(T_1,T_2)$, we get
  $$ \wh{\tau}_P^G(H_P(\delta xy)-T_P)=\sum_{\{Q:P\subseteq Q\}} (-1)^{\dim(A_Q/A_{G})} \wh{\tau}_P^Q(H_{P}(\delta x)-T_P) \Gamma_Q(H_{P}(\delta x)-T_P,-H_P(k_P(\delta x)y)). $$
Thus
\[\begin{split}
    J_\fo^{G,T}(\eta,s,f^y)=&\int_{H(F)\bs H(\BA)\cap G(\BA)^1} \sum_{\{P: \wt{P}_0\subseteq P\}} (-1)^{\dim(A_P/A_{G})} \sum_{\delta\in P_H(F)\bs H(F)} \\ & \left(\sum_{\{Q:P\subseteq Q\}} (-1)^{\dim(A_Q/A_{G})} \wh{\tau}_P^Q(H_{P}(\delta x)-T_P) \Gamma_Q(H_{P}(\delta x)-T_P,-H_P(k_P(\delta x)y))\right) \\ & \cdot k_{f,P,\fo}(\delta x) \eta(\Nrd (xy)) |\Nrd (x_1 y_1)|_\BA^s dx, 
  \end{split}\]
Exchanging the order of two sums over $P$ and $Q$, and decomposing the sum over $P_H(F)\bs H(F)$ into two sums over $P_H(F)\bs Q_H(F)$ and $Q_H(F)\bs H(F)$, we obtain
  \[\begin{split}
    J_\fo^{G,T}(\eta,s,f^y)=&\sum_{\{Q:\wt{P}_0\subseteq Q\}} \int_{H(F)\bs H(\BA)\cap G(\BA)^1} \sum_{\{P:\wt{P}_0\subseteq P\subseteq Q\}} (-1)^{\dim(A_P/A_{Q})} \sum_{\delta'\in Q_H(F)\bs H(F)} \sum_{\delta\in P_H(F)\bs Q_H(F)} \\
    &\wh{\tau}_P^Q(H_P(\delta\delta' x)-T_P) \Gamma_Q(H_P(\delta\delta' x)-T_P,-H_P(k_P(\delta\delta' x)y)) k_{f,P,\fo}(\delta\delta' x) \eta(\Nrd (xy)) \\ &\cdot |\Nrd (x_1y_1)|_\BA^s dx.
  \end{split}\]
Combining the integral over $H(F)\bs H(\BA)\cap G(\BA)^1$ and the sum over $Q_H(F)\bs H(F)$ into the integral over $Q_H(F)\bs H(\BA)\cap G(\BA)^1$, and using the fact that
  $$ P_H(F)\bs Q_H(F)\simeq (P_H(F)\cap M_{Q_H}(F))\bs M_{Q_H}(F), $$
  we have
  \[\begin{split}
    J_\fo^{G,T}(\eta,s,f^y)=&\sum_{\{Q:\wt{P}_0\subseteq Q\}} \int_{Q_H(F)\bs H(\BA)\cap G(\BA)^1} \sum_{\{P:\wt{P}_0\subseteq P\subseteq Q\}} (-1)^{\dim(A_P/A_{Q})} \sum_{\delta\in (P_H(F)\cap M_{Q_H}(F))\bs M_{Q_H}(F)} \\
    &\wh{\tau}_P^Q(H_P(\delta x)-T_P) \Gamma_Q(H_P(\delta x)-T_P,-H_P(k_P(\delta x)y)) k_{f,P,\fo}(\delta x) \eta(\Nrd (xy)) |\Nrd (x_1 y_1)|_\BA^s dx.
  \end{split}\]

By Iwasawa decomposition and our choices of measures, the integral over $Q_H(F)\bs H(\BA)\cap G(\BA)^1$ can be decomposed as integrals over 
  $$ (n,a,m,k)\in N_{Q_H}(F)\bs N_{Q_H}(\BA)\times A_Q^{G,\infty}\times M_{Q_H}(F)\bs M_{Q_H}(\BA)\cap M_Q(\BA)^1\times K_H. $$
  Then
  \[\begin{split}
    J_\fo^{G,T}(\eta,s,f^y)=&\sum_{\{Q:\wt{P}_0\subseteq Q\}} \int_{K_H}\int_{M_{Q_H}(F)\bs M_{Q_H}(\BA)\cap M_Q(\BA)^1}\int_{A_Q^{G,\infty}}\int_{N_{Q_H}(F)\bs N_{Q_H}(\BA)} \sum_{\{P:\wt{P}_0\subseteq P\subseteq Q\}} (-1)^{\dim(A_P/A_{Q})} \\
    &\sum_{\delta\in (P_H(F)\cap M_{Q_H}(F))\bs M_{Q_H}(F)} \wh{\tau}_P^Q(H_P(\delta namk)-T_P) \Gamma_Q(H_P(\delta namk)-T_P,-H_P(k_P(\delta namk)y)) \\
    &\cdot k_{f,P,\fo}(\delta namk) \eta(\Nrd (mky)) |\Nrd(a_1 m_1 y_1)|_\BA^s e^{-2\rho_{Q_H}(H_{Q_H}(am))}dndadmdk.
  \end{split}\]
As in the proof of Theorem \ref{exppol1}, we see that
$$ \wh{\tau}_P^Q(H_P(\delta namk)-T_P)=\wh{\tau}_P^Q(H_P(\delta m)-T_P), $$
and that
$$ k_{f,P,\fo}(\delta namk)=e^{2\rho_{Q,+}(H_{Q}(a))}k_{f,P,\fo}(\delta mk). $$
In addition, 
\[\begin{split}
\Gamma_Q(H_P(\delta namk)-T_P,-H_P(k_P(\delta namk)y))&=\Gamma_Q(H_Q(\delta namk)-T_Q,-H_Q(k_P(\delta namk)y)) \\
&=\Gamma_Q(H_Q(a)-T_Q,-H_Q(k_Q(\delta namk)y)) \\
&=\Gamma_Q(H_Q(a)-T_Q,-H_Q(ky)). 
\end{split}\]
To sum up, the integrand for the term indexed by $Q$ in the above formula for $J_\fo^{G,T}(\eta,s,f^y)$ is independent of $n\in N_{Q_H}(F)\bs N_{Q_H}(\BA)$. We can choose the Haar measure such that $\vol(N_{Q_H}(F)\bs N_{Q_H}(\BA))=1$. Then
  \[\begin{split}
    J_\fo^{G,T}(\eta,s,f^y)=&\sum_{\{Q:\wt{P}_0\subseteq Q\}} \int_{K_H}\int_{M_{Q_H}(F)\bs M_{Q_H}(\BA)\cap M_Q(\BA)^1}\int_{A_Q^{G,\infty}}\sum_{\{P:\wt{P}_0\subseteq P\subseteq Q\}} (-1)^{\dim(A_P/A_{Q})} \\
    &\sum_{\delta\in (P_H(F)\cap M_{Q_H}(F))\bs M_{Q_H}(F)} \wh{\tau}_P^Q(H_P(\delta m)-T_P) \Gamma_Q(H_Q(a)-T_Q,-H_Q(ky)) \\
    &\cdot e^{2\rho_{Q,+}(H_{Q}(a))}k_{f,P,\fo}(\delta mk) \eta(\Nrd (mky)) |\Nrd(a_1 m_1 y_1)|_\BA^s e^{-2\rho_{Q_H}(H_{Q_H}(am))}dadmdk.
  \end{split}\]
First, let us compute the integral on $A_Q^{G,\infty}$, which is
  \[\begin{split}
    &\int_{A_Q^{G,\infty}} |\Nrd (a_1)|_\BA^s e^{(2\rho_{Q,+}-2\rho_{Q_H})(H_Q(a))} \Gamma_Q(H_Q(a)-T_Q,-H_Q(ky)) da \\
    :=&\int_{\fa_Q^{G}} e^{(2\rho_{Q,+}-2\rho_{Q_H}+s\Sigma_1)(T_1)} \Gamma_Q(T_1-T_Q,-H_Q(ky)) dT_1 \\
    =&e^{(2\rho_{Q,+}-2\rho_{Q_H}+s\Sigma_1)(T_Q^G)} \int_{\fa_Q^{G}} e^{(2\rho_{Q,+}-2\rho_{Q_H}+s\Sigma_1)(T_1)} \Gamma_Q(T_1,-H_Q(ky)) dT_1 \\
    =&e^{(2\rho_{Q,+}-2\rho_{Q_H}+s\Sigma_1)(T_Q^G)} p_{Q,s}(-H_Q(ky)).
  \end{split}\]
Next, we consider the integral on $K_H$, which is
$$ \int_{K_H} k_{f,P,\fo}(\delta mk) \eta(\Nrd(k)) p_{Q,s}(-H_Q(ky)) dk. $$
As in the proof of Theorem \ref{exppol1}, we see that
$$ k_{f,P,\fo}(\delta mk)=e^{2\rho_{Q,+}(H_{Q_H}(m))} \sum_{X\in\fm_P(F)\cap\fo} \int_{(\fn_P^Q\cap\fs)(\BA)} dU \int_{(\fn_Q\cap\fs)(\BA)} f(k^{-1}((\delta m)^{-1}(X+U)\delta m+V)k) dV, $$
so we can write
  \[\begin{split}
    &\int_{K_H} k_{f,P,\fo}(\delta mk)\eta(\Nrd (k)) p_{Q,s}(-H_Q(ky)) dk \\
    =&e^{2\rho_{Q,+}(H_{Q_H}(m))} \sum_{X\in\fm_P(F)\cap\fo} \int_{(\fn_P^Q\cap\fs)(\BA)} f_{Q,s,y}^\eta((\delta m)^{-1}(X+U)\delta m) dU \\
    =&e^{2\rho_{Q,+}(H_{Q_H}(m))} \sum_{j=1}^{t} k_{f_{Q,s,y}^\eta,P\cap M_Q,\fo_j}^{M_Q} (\delta m)
  \end{split}\]
by (\ref{LevikfPo1}). 
Hence
  \[\begin{split}
    J_\fo^{G,T}(\eta,s,f^y)=&\eta(\Nrd(y))|\Nrd(y_1)|_\BA^s \sum_{\{Q:\wt{P}_0\subseteq Q\}} e^{(2\rho_{Q,+}-2\rho_{Q_H}+s\Sigma_1)(T_Q^G)}\int_{M_{Q_H}(F)\bs M_{Q_H}(\BA)\cap M_Q(\BA)^1} \\
    &\sum_{\{P:\wt{P}_0\subseteq P\subseteq Q\}} (-1)^{\dim(A_P/A_{Q})} \sum_{\delta\in (P_H(F)\cap M_{Q_H}(F))\bs M_{Q_H}(F)} \wh{\tau}_P^Q(H_P(\delta m)-T_P)  \\
    &\left(\sum_{j=1}^{t} k_{f_{Q,s,y}^\eta,P\cap M_Q,\fo_j}^{M_Q} (\delta m)\right) \eta(\Nrd (m)) |\Nrd (m_1)|_\BA^s e^{(2\rho_{Q,+}-2\rho_{Q_H})(H_{Q_H}(m))} dm.
  \end{split}\]
As in the proof of Theorem \ref{exppol1}, we notice that
  \[\begin{split}
    J_{\fo}^{Q,T}(\eta,\{s^Q_i+s\},f_{Q,s,y}^\eta)=&\int_{M_{Q_H}(F)\bs M_{Q_H}(\BA)\cap M_Q(\BA)^1} \sum_{\{P:\wt{P}_0\subseteq P\subseteq Q\}} (-1)^{\dim(A_P/A_{Q})} \sum_{\delta\in (P_H(F)\cap M_{Q_H}(F))\bs M_{Q_H}(F)} \\
    &\wh{\tau}_P^Q(H_P(\delta m)-T_P) \left(\sum_{j=1}^{t} k_{f_{Q,s,y}^\eta,P\cap M_Q,\fo_j}^{M_Q}(\delta m)\right) \eta(\Nrd (m)) |\Nrd (m_1)|_\BA^s \\         &\cdot e^{(2\rho_{Q,+}-2\rho_{Q_H})(H_{Q_H}(m))} dm.
  \end{split}\]
Then we finish the proof. 
\end{proof}


\section{\textbf{An infinitesimal trace formula for $\Mat_{p\times q, D}\oplus \Mat_{q\times p, D} // GL_{p, D}\times GL_{q, D}$}}\label{traceformula}

\begin{thm}\label{tf1}
  For $f\in\CS(\fs(\BA))$ and $s\in\BR$,
  $$ \sum_{\fo\in\CO}J_\fo^{G}(\eta,s,f)=\sum_{\fo\in\CO}J_\fo^{G}(\eta,s,\hat{f}), $$
  where $\hat{f}$ is the Fourier transform of $f$ defined by (\ref{fourier}), and $J_\fo^{G}(\eta,s,\cdot)$ denotes the constant term of $J_\fo^{G, T}(\eta,s,\cdot)$. 
\end{thm}

\begin{proof}
  From the Poisson summation formula, we know that for any $x\in H(\BA)$,
  $$ \sum_{X\in\fs(F)} f(x^{-1}Xx)=\sum_{X\in\fs(F)} \hat{f}(x^{-1}Xx), $$
  i.e.,
  $$ k_{f,G}(x)=k_{\hat{f},G}(x). $$

  Using Corollary \ref{comparewithnaive1}, for all sufficiently regular $T$ satisfying $\alpha(T)\geq\epsilon_0\parallel T\parallel$ for any $\alpha\in\Delta_{P_0}$, we have
  $$ \left|J^{G,T}(\eta,s,f)-\int_{H(F)\bs H(\BA)\cap G(\BA)^1}F^{G}(x,T)k_{f,G}(x)\eta(\Nrd (x))|\Nrd (x_1)|_\BA^s dx\right|\leq C_1 e^{-N\parallel T\parallel} $$
  and
  $$ \left|J^{G,T}(\eta,s,\hat{f})-\int_{H(F)\bs H(\BA)\cap G(\BA)^1}F^{G}(x,T)k_{\hat{f},G}(x)\eta(\Nrd (x))|\Nrd (x_1)|_\BA^s dx\right|\leq C_2 e^{-N\parallel T\parallel}. $$
  Thus
  $$ |J^{G,T}(\eta,s,f)-J^{G,T}(\eta,s,\hat{f})|\leq (C_1+C_2) e^{-N\parallel T\parallel}. $$

  By Corollary \ref{polynomial1}, we know that both of $J^{G,T}(\eta,s,f)$ and $J^{G,T}(\eta,s,\hat{f})$ are exponential polynomials in $T$. Because we can choose $N$ to be large enough, we deduce that
  $$ J^{G,T}(\eta,s,f)=J^{G,T}(\eta,s,\hat{f}). $$
  Since
  $$ J^{G,T}(\eta,s,f)=\sum_{\fo\in\CO}J_\fo^{G,T}(\eta,s,f) $$
  and
  $$ J^{G,T}(\eta,s,\hat{f})=\sum_{\fo\in\CO}J_\fo^{G,T}(\eta,s,\hat{f}), $$
we obtain
$$ \sum_{\fo\in\CO}J_\fo^{G,T}(\eta,s,f)=\sum_{\fo\in\CO}J_\fo^{G,T}(\eta,s,\hat{f}). $$
We may conclude by taking the constant terms of both sides. 
\end{proof}


\section{\textbf{The second modified kernel}}\label{secondmodifiedkernel}

In this section and the next, we shall focus on the case where $p=q=n$ in order to get better description for distributions associated to regular semi-simple orbits. We shall change our notation by denoting $G:=GL_{2n, D}$ and $H:=GL_{n,D}\times GL_{n,D}$ without further mention. 

Let $f\in\CS(\fs(\BA))$, $P$ be a relatively standard parabolic subgroup of $G$ and $\fo\in\CO_{rs}$ (see Section \ref{inv1}). For $x\in P_H(F)\bs H(\BA)$, define
$$ j_{f,P,\fo}(x):=\sum_{X\in\fm_P(F)\cap\fo} \sum_{n\in N_{P_H}(F)} f((nx)^{-1}Xnx). $$
Let $T\in\fa_0$. For $x\in H(F)\bs H(\BA)$, define
$$ j_{f,\fo}^T(x):=\sum_{\{P: \wt{P}_0\subseteq P\}} (-1)^{\dim(A_P/A_{G})} \sum_{\delta\in P_H(F)\bs H(F)} \wh{\tau}_P^G(H_{P}(\delta x)-T_P)\cdot j_{f,P,\fo}(\delta x). $$
By Lemma \ref{artlem51}, we know that the sum over $\delta\in P_H(F)\bs H(F)$ is finite. Recall that since $\fo\in\CO_{rs}\subseteq\CO^\times$, if $\fm_P(F)\cap\fo\neq\emptyset$, then $P$ is $\omega$-stable by Proposition \ref{propOinvert1}. Thus the above definitions only involve the relatively standard parabolic subgroups that are $\omega$-stable. 

\begin{lem}\label{regsslemma1}
  Let $P$ be a relatively standard parabolic subgroup of $G$ and $\fo\in\CO_{rs}$. For $X\in\fm_P(F)\cap\fo$, the map
  $$ N_{P_H}\ra\fn_P\cap\fs, n\mapsto n^{-1}Xn-X $$
  is an $F$-isomorphism of algebraic varieties and preserves the Haar measures on $\BA$-points. 
\end{lem}

\begin{proof}
  Since $P$ is relatively standard and $\omega$-stable, we can suppose
  $$ P=\mat(\fp_{n,D},\fp_{n,D},\fp_{n,D},\fp_{n,D})^\times, $$
  where
  $$ P_{n,D}=
  \left( \begin{array}{cccc}
  GL_{n_1,D} & \Mat_{n_1\times n_2,D} & \cdots & \Mat_{n_1\times n_l,D} \\
                   & GL_{n_2,D}                  & \cdots & \Mat_{n_2\times n_l,D} \\
                   &                                   & \ddots & \vdots                \\
                   &                                   &           & GL_{n_l,D}            \\
  \end{array} \right). $$
   Then we have
   $$ \fm_P\cap\fs=\mat(,\fm_{P_{n,D}},\fm_{P_{n,D}},), N_{P_H}=\mat(N_{P_{n,D}},,,N_{P_{n,D}}), \fn_P\cap\fs=\mat(,\fn_{P_{n,D}},\fn_{P_{n,D}},). $$

Let $$X=
  \left( \begin{array}{cccccc}
        &           &        & A_1 &            &       \\
        &           &        &       & \ddots  &        \\
        &           &        &       &           & A_l  \\
  B_1 &          &        &       &            &         \\
        & \ddots &       &       &            &         \\
        &           & B_l &       &            &          \\
  \end{array} \right)\in\fm_P(F)\cap\fo, $$
where $A_i, B_i\in GL_{n_i}(D)$ for $1\leq i\leq l$, and $$n=
 \left( \begin{array}{cccccccc}
      1 &  C_{12}   &  \cdots  & C_{1l}   &    &            &           &      \\
        &          1    &  \cdots  &  C_{2l}  &     &           &           &     \\
        &               &   \ddots  &  \vdots  &     &            &           &       \\
        &               &              &  1         &     &            &            &      \\
        &               &              &             &  1  & D_{12} & \cdots & D_{1l}  \\
        &               &              &             &      &       1   &  \cdots & D_{2l} \\
        &               &              &             &      &            &  \ddots & \vdots \\
        &               &              &             &      &            &            & 1    \\
  \end{array} \right)\in N_{P_H}, $$
where $C_{ij}, D_{ij}\in \Mat_{n_i\times n_j,D}$ for $1\leq i< j\leq l$. Then $$Xn-nX=
 \left( \begin{array}{cccccccc}
        &                                         &              &                                 &    0  &   A_1 D_{12}-C_{12} A_2   &   \cdots    & A_1 D_{1l}-C_{1l} A_l  \\
        &                                         &              &                                 &        &                                   0    &   \cdots   & A_2 D_{2l}-C_{2l} A_l \\
        &                                        &                &                                &         &                                        &    \ddots  &                      \vdots   \\
        &                                        &                &                                &         &                                        &               &                             0    \\
    0  &   B_1 C_{12}-D_{12} B_2   &   \cdots    & B_1 C_{1l}-D_{1l} B_l &      &                                          &               &                                  \\
        &                                   0    &   \cdots   & B_2 C_{2l}-D_{2l} B_l &       &                                         &               &                                   \\
        &                                        &    \ddots  &                      \vdots   &       &                                        &                &                                   \\
        &                                        &               &                             0   &       &                                         &               &                                   \\
  \end{array} \right)\in \fn_P\cap\fs. $$

We claim that the morphism of $F$-affine spaces
\[\begin{split}
\Mat_{n_i\times n_j,D}\oplus \Mat_{n_i\times n_j,D} &\ra \Mat_{n_i\times n_j,D}\oplus \Mat_{n_i\times n_j,D} \\
(C_{ij}, D_{ij}) &\mapsto (A_i D_{ij}-C_{ij} A_j, B_i C_{ij}-D_{ij} B_j) \\
\end{split}\]
induces an $F$-linear isomorphism on $F$-points. In fact, since it gives an $F$-linear map between finite dimensional linear spaces of the same dimension, we only need to prove that this map is injective under base change to an algebraic closure of $F$. Then without loss of generality, it suffices to consider the case where $D=F$. If $A_i D_{ij}-C_{ij} A_j=B_i C_{ij}-D_{ij} B_j=0$, then $C_{ij} A_j B_j=A_i D_{ij} B_j=A_i B_i C_{ij}$ and $D_{ij} B_j A_j=B_i C_{ij} A_j=B_i A_i D_{ij}$. Since $X$ is regular semi-simple, $A_i B_i$ and $A_j B_j$ (resp. $B_i A_i$ and $B_j A_j$) have no common eigenvalue. By the classical theory of Sylvester equation \cite{zbMATH02701590}, we know that $C_{ij}=D_{ij}=0$ and conclude.

From this claim, we know that the map
$$ N_{P_H}\ra\fn_P\cap\fs, n\mapsto Xn-nX $$
is an $F$-isomorphism of algebraic varieties and preserves the Haar measures on $\BA$-points. Notice that $n^{-1}Xn-X=n^{-1}(Xn-nX)$. It is not hard to check that here $n^{-1}$ functions as some translation $A_i D_{ij}-C_{ij} A_j\mapsto A_i D_{ij}-C_{ij} A_j+(\text{a polynomial of } C_{i'j'} \text{and } D_{i'j'}, i'>i, j'\leq j \text{ or } i'\geq i, j'<j)$, so an analogous assertion still holds for the map $n\mapsto n^{-1}Xn-X$. 
\end{proof}

\begin{thm}\label{secondkernel1}
  For all sufficiently regular $T$, all $s\in\BR$ and $\fo\in\CO_{rs}$,
  $$ \int_{H(F)\bs H(\BA)\cap G(\BA)^1} |j_{f,\fo}^T(x)| |\Nrd(x_1)|_\BA^s dx < \infty, $$
  where we write $x=(x_1,x_2)\in GL_{n,D}(\BA)\times GL_{n,D}(\BA)$. Moreover, for $s\in\BC$, 
  $$ J_\fo^{G,T}(\eta,s,f)=\int_{H(F)\bs H(\BA)\cap G(\BA)^1} j_{f,\fo}^T(x) \eta(\Nrd (x)) |\Nrd(x_1)|_\BA^s  dx. $$
\end{thm}

\begin{proof}
  As in the proof of Theorem \ref{convergence1}, using the left invariance of $j_{f,P,\fo}$ by $P_H(F)$, we reduce ourselves to proving
  $$ \int_{P_{1,H}(F)\bs H(\BA)\cap G(\BA)^1} \chi_{P_1, P_2}^T(x) |j_{P_1, P_2, \fo}(x)| |\Nrd(x_1)|_\BA^s dx < \infty, $$
  where $P_1\subsetneq P_2$ are a pair of relatively standard parabolic subgroups of $G$ and for $x\in P_{1,H}(F)\bs H(\BA)$, we put
  $$ j_{P_1, P_2, \fo}(x):=\sum_{\{P:P_1\subseteq P\subseteq P_2\}} (-1)^{\dim(A_P/A_{G})} j_{f,P,\fo}(x). $$
  In addition,
  $$ j_{f,P,\fo}(x)=\sum_{\{R:P_1\subseteq R\subseteq P\}} \sum_{\xi\in\wt{\fm}_{P_1}^R(F)\cap\fo} \sum_{X\in(\fn_R^P\cap\fs)(F)} \sum_{n\in N_{P_H}(F)} f((nx)^{-1}(\xi+X)nx), $$
  where we use the notations $\wt{\fm}_{P_1}^R$ and $\fn_R^P$ in the proof of Proposition \ref{proppfofconv1}. 

  Applying Lemma \ref{regsslemma1}, we get
  \[\begin{split}
    j_{f,P,\fo}(x)&=\sum_{\{R:P_1\subseteq R\subseteq P\}} \sum_{\xi\in\wt{\fm}_{P_1}^R(F)\cap\fo} \sum_{X\in(\fn_R^P\cap\fs)(F)} \sum_{u\in (\fn_P\cap\fs)(F)} f(x^{-1}(\xi+X+u)x) \\
    &=\sum_{\{R:P_1\subseteq R\subseteq P\}} \sum_{\xi\in\wt{\fm}_{P_1}^R(F)\cap\fo} \sum_{X\in(\fn_R\cap\fs)(F)} f(x^{-1}(\xi+X)x).
  \end{split}\]
  Hence
  \[\begin{split}
    j_{P_1, P_2, \fo}(x)&=\sum_{\{P:P_1\subseteq P\subseteq P_2\}} (-1)^{\dim(A_P/A_{G})} \left(\sum_{\{R:P_1\subseteq R\subseteq P\}} \sum_{\xi\in\wt{\fm}_{P_1}^R(F)\cap\fo} \sum_{X\in(\fn_R\cap\fs)(F)} f(x^{-1}(\xi+X)x)\right) \\
    &=\sum_{\{R:P_1\subseteq R\subseteq P_2\}} \sum_{\xi\in\wt{\fm}_{P_1}^R(F)\cap\fo} \left(\sum_{\{P:R\subseteq P\subseteq P_2\}} (-1)^{\dim(A_P/A_{G})}\right) \sum_{X\in(\fn_R\cap\fs)(F)} f(x^{-1}(\xi+X)x).
  \end{split}\]
  By \cite[Proposition 1.1]{MR518111}, we have
  $$ j_{P_1, P_2, \fo}(x)=(-1)^{\dim(A_{P_2}/A_{G})} \sum_{\xi\in\wt{\fm}_{P_1}^{P_2}(F)\cap\fo} \sum_{X\in(\fn_{P_2}\cap\fs)(F)} f(x^{-1}(\xi+X)x). $$
  Applying Lemma \ref{regsslemma1} again, we obtain
  $$ j_{P_1, P_2, \fo}(x)=(-1)^{\dim(A_{P_2}/A_{G})} \sum_{\xi\in\wt{\fm}_{P_1}^{P_2}(F)\cap\fo} \sum_{n_2\in N_{P_{2,H}}(F)} f((n_2 x)^{-1}\xi n_2 x), $$
where we denote $P_{2,H}:=P_2\cap H$.

  Decomposing the integral over $x\in P_{1,H}(F)\bs H(\BA)\cap G(\BA)^1$ into double integrals $n_1\in N_{P_{1,H}}(F)\bs N_{P_{1,H}}(\BA)$ and $y\in M_{P_{1,H}}(F)N_{P_{1,H}}(\BA)\bs H(\BA)\cap G(\BA)^1$, and using the fact that $\chi_{P_1, P_2}^T(x)$ is left invariant under $N_{P_{1,H}}(\BA)$, we have
  \[\begin{split}
    &\int_{P_{1,H}(F)\bs H(\BA)\cap G(\BA)^1} \chi_{P_1, P_2}^T(x) |j_{P_1, P_2, \fo}(x)||\Nrd(x_1)|_\BA^s dx \\
    =&\int_{M_{P_{1,H}}(F)N_{P_{1,H}}(\BA)\bs H(\BA)\cap G(\BA)^1} \int_{N_{P_{1,H}}(F)\bs N_{P_{1,H}}(\BA)} \chi_{P_1, P_2}^T(n_1 y) \\
    &\cdot \left|\sum_{\xi\in\wt{\fm}_{P_1}^{P_2}(F)\cap\fo} \sum_{n_2\in N_{P_{2,H}}(F)} f((n_2 n_1 y)^{-1}\xi n_2 n_1 y)\right| |\Nrd(y_1)|_\BA^s dn_1 dy \\
    \leq&\int_{M_{P_{1,H}}(F)N_{P_{1,H}}(\BA)\bs H(\BA)\cap G(\BA)^1} \chi_{P_1, P_2}^T(y) \sum_{\xi\in\wt{\fm}_{P_1}^{P_2}(F)\cap\fo} \\
    &\left(\int_{N_{P_{1,H}}(F)\bs N_{P_{1,H}}(\BA)} \sum_{n_2\in N_{P_{2,H}}(F)} |f((n_2 n_1 y)^{-1}\xi n_2 n_1 y)| dn_1\right) |\Nrd(y_1)|_\BA^s dy.
  \end{split}\]
  Since $P_{1,H}\subseteq P_{2,H}$ and $\vol(N_{P_{2,H}}(F)\bs N_{P_{2,H}}(\BA))=1$, we see that
  \[\begin{split}
    &\int_{N_{P_{1,H}}(F)\bs N_{P_{1,H}}(\BA)} \sum_{n_2\in N_{P_{2,H}}(F)} |f((n_2 n_1 y)^{-1}\xi n_2 n_1 y)| dn_1 \\
    =&\int_{N_{P_{1,H}}(F)\bs N_{P_{1,H}}(\BA)} \int_{N_{P_{2,H}}(F)\bs N_{P_{2,H}}(\BA)} \sum_{n_2\in N_{P_{2,H}}(F)} |f((n_2 n n_1 y)^{-1}\xi n_2 n n_1 y)| dn dn_1 \\
    =&\int_{N_{P_{1,H}}(F)\bs N_{P_{1,H}}(\BA)} \int_{N_{P_{2,H}}(\BA)} |f((n n_1 y)^{-1}\xi n n_1 y)| dn dn_1 \\
    =&\int_{N_{P_{1,H}}(F)\bs N_{P_{1,H}}(\BA)} \int_{(\fn_{P_2}\cap\fs)(\BA)} |f((n_1 y)^{-1}(\xi+U) n_1 y)| dU dn_1,
  \end{split}\]
  where we have applied Lemma \ref{regsslemma1} in the last equality.
  Therefore
  \[\begin{split}
    &\int_{P_{1,H}(F)\bs H(\BA)\cap G(\BA)^1} \chi_{P_1, P_2}^T(x) |j_{P_1, P_2, \fo}(x)| |\Nrd(x_1)|_\BA^s dx \\
    \leq&\int_{M_{P_{1,H}}(F)N_{P_{1,H}}(\BA)\bs H(\BA)\cap G(\BA)^1} \chi_{P_1, P_2}^T(y) \sum_{\xi\in\wt{\fm}_{P_1}^{P_2}(F)\cap\fo} \\
    &\left(\int_{N_{P_{1,H}}(F)\bs N_{P_{1,H}}(\BA)} \int_{(\fn_{P_2}\cap\fs)(\BA)} |f((n_1 y)^{-1}(\xi+U) n_1 y)| dU dn_1\right) |\Nrd(y_1)|_\BA^s dy \\
    =&\int_{P_{1,H}(F)\bs H(\BA)\cap G(\BA)^1} \chi_{P_1, P_2}^T(x) \sum_{\xi\in\wt{\fm}_{P_1}^{P_2}(F)\cap\fo} \left(\int_{(\fn_{P_2}\cap\fs)(\BA)} |f(x^{-1}(\xi+U)x)| dU\right) |\Nrd(x_1)|_\BA^s dx,
  \end{split}\]
  whose convergence results from that of the formula \eqref{equation1.1} when $R=P_2$.

  Now we begin to prove the second statement. From the first statement, now we have the right to write
\[\begin{split}
&\int_{H(F)\bs H(\BA)\cap G(\BA)^1} j_{f,\fo}^T(x) \eta(\Nrd (x)) |\Nrd(x_1)|_\BA^s dx \\ =&\sum_{\{P_1,P_2:\wt{P}_0\subseteq P_1\subseteq P_2\}} \int_{P_{1,H}(F)\bs H(\BA)\cap G(\BA)^1} \chi_{P_1, P_2}^T(x) j_{P_1,P_2,\fo}(x) \eta(\Nrd (x)) |\Nrd(x_1)|_\BA^s dx, \\
\end{split}\]
  where
  \[\begin{split}
    j_{P_1,P_2,\fo}(x)&=\sum_{\{P:P_1\subseteq P\subseteq P_2\}} (-1)^{\dim(A_P/A_{G})} j_{f,P,\fo}(x) \\
    &=\sum_{\{P:P_1\subseteq P\subseteq P_2\}} (-1)^{\dim(A_P/A_{G})} \left(\sum_{X\in\fm_P(F)\cap\fo} \sum_{n\in N_{P_H}(F)} f((nx)^{-1}Xnx)\right).
  \end{split}\]
  Decompose the integral over $x\in P_{1,H}(F)\bs H(\BA)\cap G(\BA)^1$ into double integrals over $n_1\in N_{P_{1,H}}(F)\bs N_{P_{1,H}}(\BA)$ and $y\in M_{P_{1,H}}(F)N_{P_{1,H}}(\BA)\bs H(\BA)\cap G(\BA)^1$. Since $N_{P_{1,H}}(F)\bs N_{P_{1,H}}(\BA)$ is compact, by Lemma \ref{regsslemma1} and \cite[\S41]{MR0165033},
  $$ \sum_{X\in\fm_P(F)\cap\fo} \sum_{n\in N_{P_H}(F)} |f((n n_1 y)^{-1}X n n_1 y)|=\sum_{X\in\fm_P(F)\cap\fo} \sum_{u\in(\fn_P\cap\fs)(F)} |f((n_1 y)^{-1}(X+u) n_1 y)| $$
  is bounded independently of $n_1\in N_{P_{1,H}}(F)\bs N_{P_{1,H}}(\BA)$. Then using the fact that $\chi_{P_1, P_2}^T(x)$ is left invariant under $N_{P_{1,H}}(\BA)$, we have
  \[\begin{split}
    &\int_{H(F)\bs H(\BA)\cap G(\BA)^1} j_{f,\fo}^T(x) \eta(\Nrd (x)) |\Nrd(x_1)|_\BA^s dx \\
    =&\sum_{\{P_1,P_2:\wt{P}_0\subseteq P_1\subseteq P_2\}} \int_{M_{P_{1,H}}(F)N_{P_{1,H}}(\BA)\bs H(\BA)\cap G(\BA)^1} \chi_{P_1, P_2}^T(y) \sum_{\{P:P_1\subseteq P\subseteq P_2\}} (-1)^{\dim(A_P/A_{G})} \\
    &\sum_{X\in\fm_P(F)\cap\fo} \left(\int_{N_{P_{1,H}}(F)\bs N_{P_{1,H}}(\BA)} \sum_{n\in N_{P_H}(F)} f((nn_1 y)^{-1}X nn_1 y) dn_1\right) \eta(\Nrd (y)) |\Nrd(y_1)|_\BA^s dy.
  \end{split}\]
  Since $P_{1,H}\subseteq P_H$ and $\vol(N_{P_H}(F)\bs N_{P_H}(\BA))=1$, we see that
  \[\begin{split}
    &\int_{N_{P_{1,H}}(F)\bs N_{P_{1,H}}(\BA)} \sum_{n\in N_{P_H}(F)} f((nn_1 y)^{-1}X nn_1 y) dn_1 \\
    =&\int_{N_{P_{1,H}}(F)\bs N_{P_{1,H}}(\BA)} \int_{N_{P_H}(F)\bs N_{P_H}(\BA)} \sum_{n\in N_{P_H}(F)} f((n n_2 n_1 y)^{-1}X n n_2 n_1 y) dn_2 dn_1 \\
    =&\int_{N_{P_{1,H}}(F)\bs N_{P_{1,H}}(\BA)} \int_{N_{P_H}(\BA)} f((n n_1 y)^{-1}X n n_1 y) dn dn_1 \\
    =&\int_{N_{P_{1,H}}(F)\bs N_{P_{1,H}}(\BA)} \int_{(\fn_{P}\cap\fs)(\BA)} f((n_1 y)^{-1}(X+U) n_1 y) dU dn_1,
  \end{split}\]
  where we have applied Lemma \ref{regsslemma1} in the last equality.
  Therefore
  \[\begin{split}
    &\int_{H(F)\bs H(\BA)\cap G(\BA)^1} j_{f,\fo}^T(x) \eta(\Nrd (x)) |\Nrd(x_1)|_\BA^s dx \\
    =&\sum_{\{P_1,P_2:\wt{P}_0\subseteq P_1\subseteq P_2\}} \int_{M_{P_{1,H}}(F)N_{P_{1,H}}(\BA)\bs H(\BA)\cap G(\BA)^1} \chi_{P_1, P_2}^T(y) \sum_{\{P:P_1\subseteq P\subseteq P_2\}} (-1)^{\dim(A_P/A_{G})} \\
    &\sum_{X\in\fm_P(F)\cap\fo} \left(\int_{N_{P_{1,H}}(F)\bs N_{P_{1,H}}(\BA)} \int_{(\fn_{P}\cap\fs)(\BA)} f((n_1 y)^{-1}(X+U) n_1 y) dU dn_1\right) \eta(\Nrd (y)) |\Nrd(y_1)|_\BA^s dy \\
    =&\sum_{\{P_1,P_2:\wt{P}_0\subseteq P_1\subseteq P_2\}} \int_{P_{1,H}(F)\bs H(\BA)\cap G(\BA)^1} \chi_{P_1, P_2}^T(x) \sum_{\{P:P_1\subseteq P\subseteq P_2\}} (-1)^{\dim(A_P/A_{G})} \\
    &\cdot \left(\sum_{X\in\fm_P(F)\cap\fo} \int_{(\fn_{P}\cap\fs)(\BA)} f(x^{-1}(X+U)x) dU\right) \eta(\Nrd (x)) |\Nrd(x_1)|_\BA^s dx \\
    =&\sum_{\{P_1,P_2:\wt{P}_0\subseteq P_1\subseteq P_2\}} \int_{P_{1,H}(F)\bs H(\BA)\cap G(\BA)^1} \chi_{P_1, P_2}^T(x) k_{P_1,P_2,\fo}(x) \eta(\Nrd (x)) |\Nrd(x_1)|_\BA^s dx.
  \end{split}\]
  From the proof of Theorem \ref{convergence1}, we are authorised to write
  \[\begin{split}
    J_\fo^{G,T}(\eta,s,f)&=\int_{H(F)\bs H(\BA)\cap G(\BA)^1} k_{f,\fo}^T(x) \eta(\Nrd (x)) |\Nrd(x_1)|_\BA^s dx \\
    &=\sum_{\{P_1,P_2:\wt{P}_0\subseteq P_1\subseteq P_2\}} \int_{P_{1,H}(F)\bs H(\BA)\cap G(\BA)^1} \chi_{P_1, P_2}^T(x) k_{P_1,P_2,\fo}(x) \eta(\Nrd (x)) |\Nrd(x_1)|_\BA^s dx,
  \end{split}\]
  which completes the proof.
\end{proof}


\section{\textbf{Weighted orbital integrals}}\label{weightedorbitalintegral}

As in the last section, we shall assume that $p=q=n$ in the following discussion. Moreover, we shall suppose that $s=0$ in the orbital integral for convenience, since $|\Nrd(x_1)|_\BA^s$ is not invariant under the translation by $A_G^\infty$. Recall that for $\fo\in\CO$ and $f\in\CS(\fs(\BA))$, we denote by $J_\fo^G(\eta, f)$ the constant term of $J_\fo^{G,T}(\eta, 0, f)$. 

\subsection{Weyl groups}

From Section \ref{indep1}, we may choose $P_0$ to be the stabiliser in $G$ of the flag
  $$ 0\subsetneq\langle e_1\rangle_D\subsetneq\langle e_1,f_1\rangle_D\subsetneq\langle e_1,f_1,e_2\rangle_D\subsetneq\langle e_1,f_1,e_2,f_2\rangle_D\subsetneq\cdot\cdot\cdot\subsetneq\langle e_1,f_1\cdot\cdot\cdot,e_n,f_n\rangle_D=V\oplus W $$
by the notation in Section \ref{relstdpar1}. Then all $\omega$-stable relatively standard parabolic subgroups of $G$ contain $P_0$. Denote by $\msp_0$ the stabiliser in $G$ of the flag
  $$ 0\subsetneq\langle e_1,f_1\rangle_D\subsetneq\langle e_1,f_1,e_2,f_2\rangle_D\subsetneq\cdot\cdot\cdot\subsetneq\langle e_1,f_1\cdot\cdot\cdot,e_n,f_n\rangle_D=V\oplus W. $$
It is the minimal $\omega$-stable relatively standard parabolic subgroup of $G$. A parabolic subgroup $P$ of $G$ is relatively standard and $\omega$-stable if and only if $\msp_0\subseteq P$. Let $\msp_{0,n}$ be the group of upper triangular matrices in $GL_{n,D}$. We can talk about positive roots for $G, H$ and $GL_{n,D}$ with respect to $P_0, \wt{P}_0$ and $\msp_{0,n}$ respectively. 

\begin{lem}\label{lemid1}
Let $P_1=\mat(\fp_{1,n},\fp_{1,n},\fp_{1,n},\fp_{1,n})^\times$ and $P_2=\mat(\fp_{2,n},\fp_{2,n},\fp_{2,n},\fp_{2,n})^\times$ be a pair of $\omega$-stable relatively standard parabolic subgroups of $G$, where $P_{1,n}$ and $P_{2,n}$ are standard parabolic subgroups of $GL_{n,D}$. 

1) The map $s_n\mapsto s=\mat(s_n,,,s_n)$ induces a bijection from
\begin{enumerate}[\indent a)]
\item the set of representatives $s_n$ of $\Omega^{GL_{n,D}}(\fa_{P_{1,n}},\fa_{P_{2,n}})$ in $\Omega^{GL_{n,D}}$ such that $s_n^{-1}\alpha>0$ for all $\alpha\in\Delta_{\msp_{0,n}}^{P_{2,n}}$
\end{enumerate}
to
\begin{enumerate}[\indent b)]
\item the set of representatives $s$ of $\Omega^G(\fa_{P_1},\fa_{P_2})$ in $\Omega^G$ such that $s^{-1}\alpha>0$ for all $\alpha\in\Delta_{P_0}^{P_2}$.
\end{enumerate}

2) The map $s_n\mapsto s=\mat(s_n,,,s_n)$ induces a bijection from
\begin{enumerate}[\indent a)]
\item the set of representatives $s_n$ of $\Omega^{GL_{n,D}}(\fa_{P_{1,n}}; P_{2,n})$ in $\Omega^{GL_{n,D}}$ such that $s_n(\fa_{P_{1,n}})\supseteq\fa_{P_{2,n}}$ and $s_n^{-1}\alpha>0$ for all $\alpha\in\Delta_{\msp_{0,n}}^{P_{2,n}}$
\end{enumerate}
to
\begin{enumerate}[\indent b)]
\item the set of representatives $s$ of $\Omega^G(\fa_{P_1}; P_2)$ in $\Omega^G$ such that $s(\fa_{P_1})\supseteq\fa_{P_2}$ and $s^{-1}\alpha>0$ for all $\alpha\in\Delta_{P_0}^{P_2}$.
\end{enumerate}
\end{lem}

\begin{proof}
Suppose that $P_{1,n}$ and $P_{2,n}$ correspond to the partitions $(n_1,\cdot\cdot\cdot,n_l)$ and $(n'_1,\cdot\cdot\cdot,n'_{l'})$ respectively of $n$. Then $P_1$ and $P_2$ correspond to the partitions $(2n_1,\cdot\cdot\cdot,2n_l)$ and $(2n'_1,\cdot\cdot\cdot,2n'_{l'})$ respectively of $2n$. For an integer $m>0$, denote by $S_m$ the symmetric group of degree $m$. 

1) From \cite[p. 33]{MR2192011}, the set $\Omega^{GL_{n,D}}(\fa_{P_{1,n}},\fa_{P_{2,n}})$ is empty unless $l=l'$, in which case
\begin{equation}\label{lemideq1}
 \Omega^{GL_{n,D}}(\fa_{P_{1,n}},\fa_{P_{2,n}})\simeq\{s_n\in S_l: \forall 1\leq i\leq l, n'_i=n_{s_n(i)}\}. 
\end{equation}
Similarly, the set $\Omega^G(\fa_{P_1},\fa_{P_2})$ is empty unless $l=l'$, in which case
\begin{equation}\label{lemideq2}
 \Omega^G(\fa_{P_1},\fa_{P_2})\simeq\{s\in S_l: \forall 1\leq i\leq l, 2n'_i=2n_{s(i)}\}. 
\end{equation}
The map in the lemma is induced by the obvious bijection between the right hand sides of (\ref{lemideq1}) and (\ref{lemideq2}). 

2) From \cite[p. 59]{MR2192011}, the set a) is identified with the set of $s_n\in S_l\subseteq S_n$ such that $(n_{s_n(1)},\cdot\cdot\cdot,n_{s_n(l)})$ is finer than $(n'_1,\cdot\cdot\cdot,n'_{l'})$, and such that $s_n^{-1}(i)<s_n^{-1}(i+1)$ for any $1\leq i\leq n-1$ that is not of the form $n'_1+\cdot\cdot\cdot+n'_k$ for some $1\leq k\leq l'$. Similarly, the set b) is identified with the set of $s\in S_l\subseteq S_{2n}$ such that $(2n_{s(1)},\cdot\cdot\cdot,2n_{s(l)})$ is finer than $(2n'_1,\cdot\cdot\cdot,2n'_{l'})$, and such that $s^{-1}(i)<s^{-1}(i+1)$ for any $1\leq i\leq 2n-1$ that is not of the form $2n'_1+\cdot\cdot\cdot+2n'_k$ for some $1\leq k\leq l'$. The map in the lemma is induced by the obvious bijection between these two sets. 
\end{proof}

For $P_1$ and $P_2$ a pair of $\omega$-stable relatively standard parabolic subgroups of $G$, denote by $\Omega^H(\fa_{P_1},\fa_{P_2})$ the (perhaps empty) set of distinct isomorphisms from $\fa_{P_1}$ to $\fa_{P_2}$ obtained by restriction of elements in $\Omega^H$. It is a subset of $\Omega^G(\fa_{P_1},\fa_{P_2})$ {\it{a priori}}. However, since the image of the map in Lemma \ref{lemid1}.1) is contained in $\Omega^H$, we actually have $\Omega^H(\fa_{P_1},\fa_{P_2})=\Omega^G(\fa_{P_1},\fa_{P_2})$ (cf. \cite[Lemme 2.8.1]{MR3026269}). 
Denote by $\Omega^{H}(\fa_{P_1};P_2)$ the set of $s\in \bigcup\limits_{\fa_{Q}} \Omega^{H}(\fa_{P_1},\fa_{Q})$ such that $s(\fa_{P_1})\supseteq\fa_{P_2}$ and $s^{-1}\alpha>0$ for each $\alpha\in\Delta_{Q_H}^{P_2\cap H}$, where the union takes over all $\fa_{Q}$ associated to some $\omega$-stable relatively standard parabolic subgroup $Q$ of $G$. Then $\Omega^{H}(\fa_{P_1};P_2)=\Omega^G(\fa_{P_1}; P_2)$ by Lemma \ref{lemid1}.2). 

\subsection{Regular semi-simple terms}\label{regssterms1}

Let $\fo\in\CO_{rs}$ (see Section \ref{inv1}). It is possible to choose an element $X_1\in\fo$ and a relatively standard parabolic subgroup $P_1$ of $G$ such that $X_1\in\fm_{P_1}(F)$ (thus $P_1$ is $\omega$-stable by Proposition \ref{propOinvert1}) but $X_1$ cannot be $H(F)$-conjugated to an element in the Lie algebra of any relatively standard parabolic subgroup $R\subsetneq P_1$. We call such $X_1$ an elliptic element in $(\fm_{P_1}\cap\fs)(F)$. 

Let $P_1=\mat(\fp_{1,n},\fp_{1,n},\fp_{1,n},\fp_{1,n})^\times$ be an $\omega$-stable relatively standard parabolic subgroup of $G$, where $P_{1,n}$ is a standard parabolic subgroup of $GL_{n,D}$. Let $X_1=\mat(0,A_1,B_1,0)\in(\fm_{P_1}\cap\fs)(F)$ be a regular semi-simple element in $\fs$. Then $X_1$ is elliptic in $(\fm_{P_1}\cap\fs)(F)$ if and only if $A_1B_1$ is elliptic in $\fm_{P_{1,n}}(F)$ in the usual sense, i.e., $\Prd_{A_1 B_1}$ is irreducible (see \cite[Proposition 5]{yu2013} for example). Let $H_{X_1}$ be the centraliser of $X_1$ in $H$. Then $X_1$ is elliptic in $(\fm_{P_1}\cap\fs)(F)$ if and only if the maximal $F$-split torus in $H_{X_1}$ is $A_{P_1}$. 

\begin{thm}\label{woi1}
  Let $\fo\in\CO_{rs}, P_1$ be a relatively standard parabolic subgroup of $G$ and $X_1\in\fo$ be an elliptic element in $(\fm_{P_1}\cap\fs)(F)$. For $f\in\CS(\fs(\BA))$, we have
  $$ J_\fo^{G}(\eta,f)=\vol(A_{P_1}^\infty H_{X_1}(F)\bs H_{X_1}(\BA)) \cdot \int_{H_{X_1}(\BA)\bs H(\BA)} f(x^{-1}X_1 x) v_{P_1}(x) \eta(\Nrd (x)) dx, $$
  where $v_{P_1}(x)$ is left-invariant under $H_{X_1}(\BA)$ and equals the volume of the projection onto $\fa_{P_1}^G$ of the convex hull of $\{-H_Q(x)\}$, where $Q$ runs over all semi-standard parabolic subgroups of $G$ with $M_Q=M_{P_1}$. 
\end{thm}

\begin{proof}
  Consider a relatively standard parabolic subgroup $P$ of $G$ and $X\in \fm_{P}(F)\cap\fo$ (thus $P$ is $\omega$-stable by Proposition \ref{propOinvert1}). There exists an $\omega$-stable relatively standard parabolic subgroup $P_2\subseteq P$ and $X_2\in(\fm_{P_2}\cap\fs)(F)$ such that $X_2$ is conjugate to $X$ via an element in $M_{P_H}(F)$ and the maximal $F$-split torus in $H_{X_2}$ is $A_{P_2}$. Then any element in $H(F)$ which conjugates $X_1$ and $X_2$ will conjugate $A_{P_1}$ and $A_{P_2}$. It follows that there exists $s\in\Omega^{H}(\fa_{P_1},\fa_{P_2})$ and $m\in M_{P_H}(F)$ such that
  $$ X=m\omega_{s} X_1 \omega_{s}^{-1}m^{-1}. $$
  Suppose that $P_3\subseteq P$ is another relatively standard parabolic subgroup, $s'\in\Omega^{H}(\fa_{P_1},\fa_{P_3})$ and $m'\in M_{P_H}(F)$ such that
  $$ X=m'\omega_{s'} X_1 \omega_{s'}^{-1}{m'}^{-1}. $$
  Then there is $\zeta\in H_{X}(F)$ such that
  $$ m'\omega_{s'}=\zeta m\omega_{s}. $$
  Since $H_{X}\subseteq M_{P_H}$, we see that
  $$ \omega_{s'}=\xi\omega_{s} $$
  for some $\xi\in M_{P_H}(F)$. 
In sum, for any given $P$ a relatively standard parabolic subgroup of $G$ and $X\in \fm_{P}(F)\cap\fo$, there is a unique $s\in \Omega^{H}(\fa_{P_1};P)$ such that $X=m\omega_{s} X_1 \omega_{s}^{-1}m^{-1}$ for some $m\in M_{P_H}(F)$.

  For $x\in P_H(F)\bs H(\BA)$, we obtain
  \[\begin{split}
    j_{f,P,\fo}(x)&=\sum_{X\in\fm_P(F)\cap\fo} \sum_{n\in N_{P_H}(F)} f((nx)^{-1}X n x) \\
    &=\sum_{s\in\Omega^H(\fa_{P_1};P)} \sum_{m\in M_{P_H, \omega_{s}X_1\omega_{s}^{-1}}(F)\big\backslash M_{P_H}(F)} \sum_{n\in N_{P_H}(F)} f((m nx)^{-1}\omega_s X_1\omega_s^{-1} m nx) \\
    &=\sum_{s\in\Omega^H(\fa_{P_1};P)} \sum_{m\in M_{P_H, \omega_{s}X_1\omega_{s}^{-1}}(F)\big\backslash P_H(F)} f((m x)^{-1}\omega_s X_1\omega_s^{-1} m x),
  \end{split}\]
  where $M_{P_H, \omega_{s}X_1\omega_{s}^{-1}}$ denotes the centraliser of $\omega_{s}X_1\omega_{s}^{-1}$ in $M_{P_H}$. For $T\in\fa_0$ and $x\in H(F)\bs H(\BA)$, we have
  \[\begin{split}
    j_{f,\fo}^T(x)=&\sum_{\{P:\msp_0\subseteq P\}} (-1)^{\dim(A_P/A_{G})} \sum_{\delta\in P_H(F)\bs H(F)} \wh{\tau}_P^G(H_{P}(\delta x)-T_P)\cdot j_{f,P,\fo}(\delta x) \\
    =&\sum_{\{P:\msp_0\subseteq P\}} (-1)^{\dim(A_P/A_{G})} \sum_{\delta\in P_H(F)\bs H(F)} \wh{\tau}_P^G(H_{P}(\delta x)-T_P) \\
    &\cdot \left(\sum_{s\in\Omega^H(\fa_{P_1};P)} \sum_{m\in M_{P_H, \omega_{s}X_1\omega_{s}^{-1}}(F)\big\backslash P_H(F)} f((m \delta x)^{-1}\omega_s X_1\omega_s^{-1} m \delta x)\right) \\
    =&\sum_{\{P:\msp_0\subseteq P\}} (-1)^{\dim(A_P/A_{G})} \sum_{s\in\Omega^H(\fa_{P_1};P)} \sum_{\delta\in M_{P_H, \omega_{s}X_1\omega_{s}^{-1}}(F)\big\backslash H(F)} \wh{\tau}_P^G(H_{P}(\delta x)-T_P) \\
    &\cdot f((\delta x)^{-1}\omega_s X_1\omega_s^{-1} \delta x).
  \end{split}\]
  Notice that the centraliser of $\omega_s X_1\omega_s^{-1}$ in $H$ is actually contained in $M_{P_H}$. We deduce that
  \[\begin{split}
    j_{f,\fo}^T(x)=&\sum_{\{P:\msp_0\subseteq P\}} (-1)^{\dim(A_P/A_{G})} \sum_{s\in\Omega^H(\fa_{P_1};P)} \sum_{\delta\in H_{\omega_{s}X_1\omega_{s}^{-1}}(F)\big\backslash H(F)} \wh{\tau}_P^G(H_{P}(\delta x)-T_P) \cdot f((\delta x)^{-1}\omega_s X_1\omega_s^{-1} \delta x) \\ 
    =&\sum_{\{P:\msp_0\subseteq P\}} (-1)^{\dim(A_P/A_{G})} \sum_{s\in\Omega^H(\fa_{P_1};P)} \sum_{\delta\in H_{X_1}(F)\backslash H(F)} \wh{\tau}_P^G(H_{P}(\omega_s \delta x)-T_P) \cdot f((\delta x)^{-1}X_1 \delta x). \\ 
  \end{split}\]
  For $y\in H(\BA)$, write
  $$ \chi_T(y):=\sum_{\{P:\msp_0\subseteq P\}} (-1)^{\dim(A_P/A_{G})} \sum_{s\in\Omega^H(\fa_{P_1};P)} \wh{\tau}_P^G(H_{P}(\omega_s y)-T_P). $$
  Then
  $$ j_{f,\fo}^T(x)=\sum_{\delta\in H_{X_1}(F)\backslash H(F)} f((\delta x)^{-1}X_1 \delta x)\cdot \chi_T(\delta x). $$

  For sufficiently regular $T$, using Theorem \ref{secondkernel1} and the fact that $j_{f,\fo}^T(x)\eta(\Nrd (x))$ is left invariant by $A_{G}^\infty$, we have
  \[\begin{split}
    J_\fo^{G,T}(\eta,0,f)&=\int_{H(F)\bs H(\BA)\cap G(\BA)^1} j_{f,\fo}^T(x) \eta(\Nrd(x)) dx \\
    &=\int_{A_{G}^\infty H(F)\bs H(\BA)} \left(\sum_{\delta\in H_{X_1}(F)\backslash H(F)} f((\delta x)^{-1}X_1 \delta x)\cdot \chi_T(\delta x)\right) \eta(\Nrd (x)) dx. \\
  \end{split}\]
Hence, 
\begin{equation}\label{lastequ1}
    J_\fo^{G,T}(\eta,0,f)=\vol(A_{P_1}^\infty H_{X_1}(F)\bs H_{X_1}(\BA))\cdot \int_{H_{X_1}(\BA)\bs H(\BA)} f(x^{-1}X_1 x) v_{P_1}(x,T) \eta(\Nrd(x)) dx.
\end{equation}
  where
  $$ v_{P_1}(x,T):=\int_{A_{G}^\infty\bs A_{P_1}^\infty} \chi_T(ax) da. $$
Here we have cheated by assuming that $v_{P_1}(x,T)$ is well-defined and left-invariant under $H_{X_1}(\BA)$ in the last equality, which is explained below along with its geometric interpretation. 

Let $Q$ be a parabolic subgroup of $G$ containing $P_0$. Since $\msp_0\subseteq P_1$, by the charaterisation in \cite[p. 59]{MR2192011}, $\Omega^G(\fa_{P_1}; Q)$ is empty unless $\msp_0\subseteq Q$, in which case we have $\Omega^G(\fa_{P_1}; Q)=\Omega^{H}(\fa_{P_1};Q)$ by Lemma \ref{lemid1}.2). Therefore, we have
$$ \chi_T(y)=\sum_{\{Q:P_0\subseteq Q\}} (-1)^{\dim(A_Q/A_{G})} \sum_{s\in\Omega^G(\fa_{P_1};Q)} \wh{\tau}_Q^G(H_{Q}(\omega_s y)-T_Q). $$
Compared to \cite[p. 951]{MR518111}, $v_{P_1}(x,T)$ is nothing but the restriction to $H(\BA)$ of Arthur's weight for $G(\BA)$. It is shown in \cite[Corollary 3.3]{MR0412348} that the integral over $a$ can be taken over a compact subset. From \cite[Corollary 3.5]{MR0412348}, $v_{P_1}(x,T)$ equals the volume  of the projection onto $\fa_{P_1}^G$ of the convex hull of $\{T_Q-H_Q(x)\}$, where $Q$ takes over all semi-standard parabolic subgroups of $G$ with $M_Q=M_{P_1}$. For $y\in H_{X_1}(\BA)\subseteq M_{P_1\cap H}(\BA)$, the convex hull associated to $v_{P_1}(yx, T)$ is a translation of that associated to $v_{P_1}(x, T)$, so they have the same volume, i.e., $v_{P_1}(yx, T)=v_{P_1}(x, T)$. By taking constant terms of both sides of (\ref{lastequ1}), we obtain the theorem. 
\end{proof}

\begin{remark}
As mentioned in the proof of Theorem \ref{woi1}, the weights we get for regular semi-simple orbits are the restriction to $H(\BA)$ of Arthur's weights (see \cite[p. 951]{MR518111}) for $G(\BA)$. They are also the same as those (see \cite[p. 131]{MR3026269}) appearing in the twisted trace formula for $(GL_{n,D}\times GL_{n,D})\rtimes \sigma$, where $\sigma$ acts on $GL_{n,D}\times GL_{n,D}$ by $\sigma(x,y):=(y,x)$. For $P_n$ a standard parabolic subgroup of $GL_{n,D}$ and $P=\mat(\fp_n,\fp_n,\fp_n,\fp_n)^\times$ an $\omega$-stable relatively standard parabolic subgroup of $G$, we may identify $\fa_P$ with the $\sigma$-invariant subspace of $\fa_{P_n\times P_n}$. The $\omega$-stable relatively standard parabolic subgroups of $G$ here play the role of the $\sigma$-stable standard parabolic subgroups of $GL_{n,D}\times GL_{n,D}$, which correspond to the standard parabolic subsets of $(GL_{n,D}\times GL_{n,D})\rtimes \sigma$ in the sense of \cite[\S2.7]{MR3026269}. However, we need more (namely relatively standard) parabolic subgroups in our truncation to deal with $\fo\notin\CO^\times$. 
\end{remark}


\bibliography{References}
\bibliographystyle{plain}

\medskip

\begin{flushleft}
Universit\'{e} de Paris, Sorbonne Universit\'{e}, CNRS, Institut de Math\'{e}matiques de Jussieu-Paris Rive Gauche, F-75013 Paris, France \\
\medskip
E-mail: huajie.li@imj-prg.fr \\
\end{flushleft}

\end{document}